\documentclass[11pt,a4paper]{amsart}

\usepackage[utf8]{inputenc}
\usepackage{amssymb,amsmath,amsbsy,amscd,amsfonts,amsthm,latexsym}
\usepackage[mathcal]{eucal}
\usepackage{times, euler}
\usepackage{eurosym}

\usepackage{mathtools}

\usepackage{stackrel, enumerate}
\usepackage{mathabx}

\usepackage[usenames]{color}
 
\usepackage[top=3cm,bottom=3cm,left=2cm,right=2cm,footskip=17mm]{geometry}

\usepackage{hyperref}

\usepackage{tikz,pgf}
\usepackage{tikzit}

\tikzstyle{1function}=[fill=white, draw=black, shape=rectangle, minimum width=0.75cm, minimum height=1 cm]
\tikzstyle{2function}=[fill=white, draw=black, shape=rectangle, minimum width=1cm, minimum height=1 cm]
\tikzstyle{3function}=[fill=white, draw=black, shape=rectangle, minimum width=1.5cm, minimum height=1cm]
\tikzstyle{multi function}=[fill=white, draw=black, shape=rectangle, minimum width=5cm, minimum height=1cm]
\tikzstyle{multi function small}=[fill=white, draw=black, shape=rectangle, minimum width=4cm, minimum height=1cm]
\tikzstyle{Multi function smaller}=[fill=white, draw=black, shape=rectangle, minimum width=3.5 cm, minimum height=1 cm]
\tikzstyle{graph node}=[fill=black, draw=black, shape=circle]
\tikzstyle{smallbox}=[fill=white, draw=black, shape=rectangle, minimum width=0.5 cm, minimum height=0.5cm]
\tikzstyle{Circle1}=[fill=white, draw=black, shape=circle, minimum size=26 mm]

\tikzstyle{black dashed}=[-, draw=black, dashed]
\tikzstyle{new edge style 0}=[thick, ->]
\tikzstyle{thickedge}=[thick, -]

\hypersetup{
    colorlinks=true,
    linkcolor=blue,
    filecolor=cyan,      
    urlcolor=blue,
    citecolor=magenta
}

\definecolor{darkgreen}{rgb}{0,0.6,0}

\newcommand{\udi}[1]{{\color{red}{#1}}}

\usepackage{xy}
\xyoption{all}

\numberwithin{equation}{section}

\theoremstyle{plain}
    \newtheorem{theorem}[equation]{Theorem}
    \newtheorem{lemma}[equation]{Lemma}
    \newtheorem{lemma-definition}[equation]{Lemma-Definition}
\newtheorem{thmABC}{Theorem}

    \newtheorem{corollary-definition}[equation]{Corollary-Definition}
    \newtheorem{corollary}[equation]{Corollary}
    \newtheorem{proposition}[equation]{Proposition}

    \newtheorem*{def*}{Definition}
     
\theoremstyle{definition}
    \newtheorem{definition}[equation]{Definition}
    
    \newtheorem{example}[equation]{Example}

    \newtheorem{remark}[equation]{Remark}

\newcommand{\C}{\mathcal{C}}
\newcommand{\Cp}{\mathcal{C}_{+1}}

\newcommand{\N}{\mathbb{N}}
\newcommand{\Z}{\mathbb{Z}}
\newcommand{\Q}{\mathbb{Q}}
\newcommand{\F}{\mathbb{F}}

\renewcommand{\H}{\mathcal H}
\renewcommand{\k}{\Bbbk}
    
\renewcommand{\phi}{\varphi}
\renewcommand{\epsilon}{\varepsilon}

\newcommand{\ol}{\overline}

\newcommand{\ot}{\otimes}

\newcommand{\Ga}{\Gamma}

\newcommand{\uu}{u}
\newcommand{\vv}{v}

\DeclareMathOperator{\Irr}{Irr}
\DeclareMathOperator{\Hom}{Hom}
\DeclareMathOperator{\End}{End}
  
\DeclareMathOperator{\Ker}{Ker}

\DeclareMathOperator{\Aut}{Aut}
\DeclareMathOperator{\Ind}{Ind}

\DeclareMathOperator{\Inf}{Inf}

\DeclareMathOperator{\GL}{GL}

\DeclareMathOperator{\Rep}{\mathrm{Rep}}

\DeclareMathOperator{\res}{res}

\DeclareMathOperator{\Span}{Span}

\DeclareMathOperator{\Mod}{Mod}

\DeclareMathOperator{\ev}{ev}
\DeclareMathOperator{\coev}{coev}
\DeclareMathOperator{\Fr}{Fr}

\newcommand{\one}{\mathbf{1}}

\newcommand{\cA}{\mathcal{A}}  	

\newcommand{\cJ}{\mathcal{J}}  	

\newcommand{\lmod}{{\operatorname{-Mod}}}

\newcommand{\U}{\mathfrak{U}}
  	
\renewcommand{\o}{\mathfrak o}	   
\newcommand{\f}{\mathfrak f}	   

\newcommand{\p}{\mathfrak p}	           

\renewcommand{\l}{\ell}                         

\DeclareMathOperator{\tr}{Tr}  	 

\newcommand{\Ow}{\mathfrak o}	  

\newcommand{\cF}{\mathcal{F}}     

\newcommand{\wt}{\widetilde}

\newcommand{\la}{\lambda}
\newcommand{\Id}{\text{Id}}
\newcommand{\Fun}{\text{Fun}}
\renewcommand{\Im}{\text{Im}}

\newcommand{\Tr}{\text{Tr}}
\newcommand{\D}{\mathcal{D}}

\newcommand{\Ss}{\mathbb{S}}
\newcommand{\parti}{\vdash}
\newcommand{\M}{\text{M}}
\newcommand{\stab}{\text{Stab}}
\newcommand{\rank}{\text{rank}}
\newcommand{\Supp}{\text{Supp}}
\newcommand{\Inj}{\text{Inj}}
\newcommand{\Gr}{\text{Gr}}
\renewcommand{\leqq}{\preccurlyeq}


\begin{document}

	\title[Functor morphing  and representations of automorphism groups of modules]{Functor morphing  and representations of automorphism groups of modules}

	\date{\today} 
	

\author{Tyrone Crisp}
\address{Department of Mathematics and Statistics, University of Maine, Orono, ME 04469-5752, USA}
  \email{tyrone.crisp@maine.edu}

\author{Ehud Meir}
\address{Institute of Mathematics, University of Aberdeen, Fraser Noble Building, Aberdeen AB24 3UE, UK}
  \email{meirehud@gmail.com}

\author{Uri Onn}
\address{Mathematical Sciences Institute, The Australian National University, Canberra, Australia}
  \email{uri.onn@anu.edu.au}

	\begin{abstract}  We introduce a strategy to study irreducible representations of automorphism groups of finite modules over finite rings. We prove that these automorphism groups fit in a hierarchy that facilitates a stratification of their irreducible representations in terms of smaller building blocks using symmetric monoidal categories and invariant theory. 
	\end{abstract}
	\maketitle
	\setcounter{tocdepth}{1} \tableofcontents{}

\section{Introduction}





Representations of reductive groups over finite and local fields were studied extensively in the past 50 years with remarkable success. A common narrative featuring in both cases, albeit using different methods, is the {\em Harish-Chandra philosophy of cusp forms}. In practice, parabolic induction and restriction functors facilitate efficient transport of representations across different groups between a reductive group and its Levi subgroups, essentially reducing the study to two separate problems. The first of arithmetic nature, namely, finding the (super)cuspidal representations---those representations that do not occur in parabolically induced representations; and the second of combinatorial nature---controlling induced representations.

\smallskip

In \cite{CMO1,CMO3, CMO2} 
functors that generalise parabolic induction and restriction were introduced to facilitate a similar inductive scheme for compact reductive groups. A prominent example is the family of general linear groups over the ring of integers $\o$ in a non-archimedean local field $\f$. 
The supercuspidal representations of~$\GL_n(\f)$ were constructed in~\cite{Bushnell-Kutzko} and can be described as compactly-induced representations from certain irreducible representations of $\GL_n(\o)$ called {\em types};  see~\cite{Paskunas}. The latter are inflated from finite congruence quotients $\GL_n(\o_\l)$, where $\p$ is the maximal ideal in $\o$ and $\o_\l=\o/\p^\l$. For example, the depth zero supercuspidals arise from the cuspidal representations of $\GL_n(\o_1)$. It is natural to ask 
\begin{enumerate}
    \item[(a)] How does the family of groups $\{ \GL_n(\o) \mid n \in \N\}$  fit into the philosophy of cusp forms? 
    \item[(b)] What can be said about {\em all} their irreducible representations?
    \item[(c)] What role do the types play in this context?
\end{enumerate}

 The variants of parabolic induction and restriction studied in \cite{CMO1,CMO3,CMO2} were introduced to address these questions. Conjecturally, these functors satisfy a Mackey-type formula similar to the known formulas for reductive groups over finite fields \cite[Ch.\ 5]{DM} and over non-archimedean local fields \cite{BZ_reductive}, but this conjectural formula has been proven so far only on the subcategory generated by {\em strongly cuspidal representations}; see \cite[Theorem D]{CMO3}. The aim of the present paper is to introduce a new perspective on the questions (a), (b), and (c) listed above which is, to a large extent, independent of and complementary to our earlier approach.

\smallskip

Given a field $E$, the family of groups $\Aut_E(E^n) \cong \GL_n(E)$ is closed under taking substructures: all automorphism groups of subspaces are already included. Replacing the field $E$ by a ring, say $\o_\l$, one encounters automorphism groups of arbitrary modules, not only free. Furthermore, the relative position of flags in $E^n$ is determined  combinatorially by elements of the symmetric group, whereas determining the relative position of flags of modules is a wild problem; see \cite{RingelSchmidmeier2006}. We work in the following more general set-up to allow flexibility. 

\smallskip

Let $R$ be a finite not-necessarily commutative $\o_\l$-algebra and let $M$ be a finite $R$-module.  Let~$K$ be an algebraically closed field of characteristic zero. Our aim is to study the irreducible $K$-linear representations of~$\Aut_R(M)$ systematically. Let $KM=\Span_k\{u_m \mid m \in M\}$ be the $K$-vector space with basis $M$. Then~$KM$ is the {\em tautological} representation of $\Aut_R(M)$ via the natural permutation action on basis elements. The importance of the representation $KM$ cannot be overestimated: first, it determines the group $\Aut_R(M)$, being its group of symmetries (see the discussion following Remark~\ref{rem:interpolation}); second, it is faithful and therefore tensor-generates the category $\Rep(\Aut_R(M))$ of $K$-linear representations of $\Aut_R(M)$; and third, it carries a geometric flavour, which we heavily use as explained in \S\ref{sec:alg.structures}.

\subsection{Main results} 
 Let $M=M_1^{a_1} \oplus \cdots \oplus M_n^{a_n}$, where $M_1, \ldots, M_n$ are non-isomorphic indecomposable $R$-modules and $a_i \in \N$. Let $\C=\langle M_1,\ldots, M_n\rangle$ be the full subcategory of $R\lmod$ of all left $R$-modules that can be written in the form $\bigoplus_i M_i^{b_i}$ for some $b_i\in\N_0 = \N \cup \{0\}$. Since the group $\Aut_R(M)$ is finite and the representation $KM$ is faithful, every irreducible representation $V$ of $\Aut_R(M)$ appears in some tensor power $KM^{\ot n}$ of $KM$ \cite[Problem 2.37]{Fulton-Harris}. The smallest $n$ for which this is true is called the {\em tensor rank} of~$V$. Labelling representations by their tensor rank already gives a stratification on the representation theory of $\Aut_R(M)$. We use ideas from invariant theory and symmetric monoidal categories to show that this stratification has a natural refinement, which we explain now. 
 
\smallskip

Let $\Cp=\Fun(\C, \o_\l\lmod)$ be the category of $\o_\l$-linear functors from $\C$ to $\o_\l\lmod$. For every $F\in \Cp$ we have the natural representation $KF(M) = \Span_K\{u_m \mid m\in F(M)\}$ of $\Aut_R(M)$. 
For $F,G\in \Cp$ we write $F\preccurlyeq G$ if $F$ is isomorphic to a subquotient of $G$, and $F\prec G$ if $F$ is isomorphic to a proper subquotient of $G$. In Section \ref{sec:stratification} we prove the following:
\begin{thmABC}\label{thm:main1} Let $V$ be an irreducible representation of $\Aut_R(M)$. 
There is a unique (up to isomorphism) functor $F \in \Cp$ such that $\Hom_{\Aut_R(M)}(KF(M),V)\neq 0$ and $\Hom_{\Aut_R(M)}(KG(M),V)=0$ for all $G\prec F$. 
\end{thmABC}
For each $F \in \Cp$ we write $\ol{KF(M)}$ for the largest subrepresentation of $KF(M)$ that has the property  ${\Hom_{\Aut_R(M)}(KG(M),\ol{KF(M)})=0}$ for every $G\prec F$. 
We will show that for every $F,G\in \Cp$, the space $\Hom_{\Aut_R(M)}(KF(M),KG(M))$ has a natural spanning set, which is a basis when $M$ large enough, indexed by subfunctors of $F\oplus G$. We then use this to prove the following epimorphism theorem:
\begin{thmABC}\label{thm:main2} Let $F\in \Cp$ and $M \in \mathrm{ob}(\C)$. Then there exists a canonical algebra epimorphism 
\[
\Phi_F(M):K\Aut_{\Cp}(F)\to \End(\ol{KF(M)}). 
\]
If $M$ is large enough then $\Phi_F(M)$ is also injective.   
\end{thmABC}
Theorem~\ref{thm:main1} associates to every irreducible representation $V$ of $\Aut_R(M)$ a functor $F\in \Cp$ up to isomorphism. Theorem~\ref{thm:main2} allows us to further associate with $V$ an irreducible representation $\wt{V}$ of $\Aut_{\Cp}(F)$ by means of the Wedderburn Theorem (see Subsection \ref{subsec:con.epi}). 
\begin{def*}[Functor morphing]
    We call the assignment $(M,V)\leadsto (F,\wt{V})$ the functor morphing  of $(M,V)$. 
\end{def*}
This in particular gives us a stratification 
\[
\Rep(\Aut_R(M)) = \bigoplus_{F\in \mathrm{ob}(\Cp)/\cong} \Rep(\Aut_R(M))_F,
\]
where $\Rep(\Aut_R(M))_F$ is the subcategory generated by representations coming from the $F$-layer, namely, those with functor morphing of the form $(F,\wt{V})$. These are exactly direct sums of the irreducible constituents of $\ol{KF(M)}$. By Theorem \ref{thm:main2}, $\Rep(\Aut_R(M))_F$ can be identified with a full subcategory of $\Aut_{\Cp}(F)$; see \S\ref{sec:stratification}.

\smallskip

To each irreducible representation $V$ of $\Aut_R(M)$ we associate its {\em degree} $$\mathrm{deg}(M,V)=(\dim V, |\Aut_R(M)|,\Supp(M))\in \N^3,$$ where $\Supp(M)$ stands for the number of isomorphism classes of indecomposable direct summands of $M$. Assume that $(M,V)\leadsto (F,\wt{V})$. 
We can also assign a degree to $(F,\wt{V})$, namely $(\dim\wt{V}, |\Aut_{\Cp}(F)|, \Supp(F))$, where $\Supp(F)$ is well-defined because $\Cp$ is an abelian category in which the Krull-Schmidt Theorem holds, since it is equivalent to the category of finitely generated modules over a finite ring; see Appendix \ref{appendix}. 

By Yoneda's Lemma, we know that the functor $\C^{op}\to \Cp$ given by $M\mapsto F_M := \Hom_R(M,-)$ is a full embedding. In particular, we have an isomorphism $\Aut_R(M)\cong \Aut_{\Cp}(F_M)$. We have the following result:

\begin{thmABC}\label{thm:main3}
If $(M,V)\leadsto (F,\wt{V})$ then 
$\mathrm{deg}(M,V) \geq \mathrm{deg}(F,\wt{V})$ in the lexicographical order. If equality holds then necessarily $F\cong F_M$ and $V\cong \wt{V}$ as $\Aut_R(M) \cong \Aut_{\Cp}(F_M)$-representations. 
\end{thmABC}

If equality holds in Theorem~\ref{thm:main3} then we say that the pair $(M,V)$ admits \emph{trivial functor morphing}. We show that admitting trivial functor morphing is a very restrictive property. 
To do so, we introduce in~\S\ref{sec:graph} the left maximal graph $LM(M_1,\ldots, M_n)$ for a tuple of indecomposable $R$-modules. For a general finite $\o_\l$-algebra~$R$ this graph can be quite complicated. However, in case of a trivial functor morphing we have the following result:
\begin{thmABC}\label{thm:main4}
Let $V\in \Irr(\Aut_R(M))$, where $M=M_1^{a_1}\oplus\cdots\oplus M_n^{a_n}$ with $M_i$ non-isomorphic indecomposable modules and $a_i>0$.
Assume that $V$ admits trivial functor morphing. Then one of the following holds:
\begin{enumerate}
    \item The module $M$ has a non-trivial splitting $M=M'\oplus M''$ such that {${\Aut_R(M) =\Aut_{R}(M')\times \Aut_{R}(M'')}$}.
    \item The representation $V$ is induced from a proper subgroup of $\Aut_R(M)$ that is isomorphic to $\Aut_{\wt{R}}(\wt{M})$, for explicitly constructed  ring $\wt{R}$  and an $\wt{R}$-module $\wt{M}$. 
    \item The graph $LM(M_1,\ldots, M_n)$ is a disjoint union of singletons with single self-edges. 
\end{enumerate} 
\end{thmABC}
\begin{remark}
    If the second condition of the above theorem holds then even more is true. The full subcategory of representations that admit trivial functor morphing can be shown, by Clifford theory, to be equivalent to a subcategory of $\Rep(\Aut_{\wt{R}}(\wt{M}))$. 
\end{remark}











\subsection{Context and related works} The thrust of the present paper is to introduce methods from invariant theory and monoidal categories to study the representation theory of automorphism groups of modules. It generalises, contextualises and complements several results and notions from the representation theory of automorphism groups of finite $\o_\l$-modules \cite{Bader-Onn, Onn, GurevichHowe2021}; at the same time, it fits into a more general framework in tensor categories \cite{Deligne-St, Knop2007, Knop2022, meir21}.

\subsubsection{Representation theory} Theorem~\ref{thm:main1} gives a stratification of the irreducible representations of the automorphism groups of finite $R$-modules. In the special case $R=\F_q$, a finite field, and $M=\F_q^n$, this filtration coincides with the tensor rank filtration studied by Gurevich--Howe in~\cite{GurevichHowe2021}; see \S\ref{ex:glnfq} for more details, along with the connection to Harish-Chandra multiplication and Zelevinsky's PSH algebra \cite{Zelevinsky}. 

For $R=\o_\l$ and $M=\o_\l^n$, a shadow of this filtration was observed in~\cite{Bader-Onn}, where a filtration on a specific representation (the Grassmann representation) of $\GL_n(\o_\l)=\Aut_{\o_\l}(\o_\l^n)$ was introduced. Indeed, the representation $KX$, with $X=\left\{N \leq \o_\l^n \mid N \cong \o_\l^m\right\}$, decomposes according to isomorphism types $\lambda$ of submodules of~$\o_\l^m$. These correspond to hom-functors of the form $\Hom(M_\lambda,-)$ in the terminology of the present paper.  

More generally, the irreducible representations of $\GL_n(\o_\l)$ can be partitioned to families according to orbits of characters of the smallest principal congruence subgroup $I+M_n(\p^{\l-1}/\p^\l) \cong (M_n(\F_q),+)$. Their study can be reduced to nilpotent orbits; see \cite{CMO3,Hill_Jord}. These so-called nilpotent representations admit functor morphing  to representations of $\Aut_{\o_\l}(M_\la)$ with $M_\lambda < \o_\l^n$; see \S\ref{subsec:Grassmann}. 


\subsubsection{Tensor categories} 

The approach we take in this paper to $\Rep\left(\Aut_R(M)\right)$ follows the construction of symmetric monoidal categories done by the second author in \cite{meir21}. This construction generalises the construction of Deligne of the category $\Rep(S_t)$ from \cite{Deligne-St}.  Another approach to construction of symmetric monoidal categories was given by Knop (see \cite{Knop2007} and \cite{Knop2022}). Knop's construction relies on a regular category $\cA$ in which every object has finitely many subobjects, and a degree function $\delta: C(\cA)\to K$, where $C(\cA)$ is the class of all surjective morphisms in $\cA$. He then constructs a $K$-linear category $\mathcal{T}(\cA,\delta)$, in which the objects are the objects of $\cA$ and the morphisms are linear combinations of correspondences (subobjects of the direct product), and considers possible tensor ideals and Tannakian quotients of $\mathcal{T}(\cA,\delta)$. If one takes $\cA = \Cp$ and the degree function $\delta$ to be $\delta(r:F\to G) = |\ker(r(M):F(M)\to G(M))|$ then the maximal tensor ideal $\mathcal{N}$ of $\mathcal{T}(\cA,\delta)$ will satisfy that $\mathcal{T}(\cA,\delta)/\mathcal{N}\cong \Rep(\Aut_R(M))$. This follows from Theorem 9.8. in \cite{Knop2007}, where we take the functor $P$ to be $P(F) = F(M)\in \mathrm{Set}$. Theorem~\ref{thm:epimorphism} then becomes a consequence of Corollary 5.3 in \cite{Knop2007}. In this paper we follow the construction from \cite{meir21} as it arises in a more natural way in our context and amenable to a self-contained analysis. We study the representation theory of $\Aut_R(M)$, and we therefore start with the tautological representation $KM$ and a generating set of structure tensors. 

\subsection*{Acknowledgments}  
The first and second authors were partly supported by the Danish National Research Foundation through the Centre for Symmetry and Deformation (DNRF92). 
The first author was also supported by fellowships from the Max Planck Institute for Mathematics in Bonn, and from the Radboud Excellence Initiative at Radboud University Nijmegen.
The second author was also supported by the Research Training Group 1670 ``Mathematics Inspired by String Theory and Quantum Field Theory.'' 
The second author thanks the Mathematical Sciences Institute at the ANU for funding his visit to the ANU to work on this project. The third author was supported by the Australian Research Council (ARC FT160100018) and was supported in part at the Technion by a fellowship from the Lady Davis Foundation. The first author gratefully acknowledges support from the American Mathematical Society and the Simons Foundation.

\section{Automorphism groups of finite modules as algebraic structures over $K$}\label{sec:alg.structures}
Throughout $\o$ denotes a complete discrete valuation ring with uniformiser $\pi$ and finite residue field of cardinality~$q$. For $\l \in \N$, we write $\o_\l=\o/\pi^\l$.
Let $R$ be a finite $\Ow_\l$-algebra and let $M$ be a finite $R$-module. Let $K$ be an algebraically closed field of characteristic zero.
In this paper we propose a strategy to study the representation theory of $\Aut_R(M)$ over~$K$.
We will do so by considering the category $\Rep(\Aut_R(M))$ as a symmetric monoidal category. 
We usually think of this category as being constructed from the top down; we start with the category of vector spaces, and we then consider only those vector spaces that are equipped with a linear action of $\Aut_R(M)$. In this paper we will concentrate on a construction of the category $\Rep(\Aut_R(M))$ from the bottom up by using the language of algebraic structures, as we shall explain next. 

\subsection{Intertwiners and Constructible maps}
Let $W$ be a $K$-vector space equipped with a finite collection of linear maps, called structure tensors $x_i:W^{\ot q_i}\to W^{\ot p_i}$, $i\in I$. The case $p_i=0$ or $q_i=0$ is possible, where $W^{\ot 0} = K$, the ground field.  
The tuple $((p_i,q_i))\in \N^{2|I|}$ is called the \textit{type} of $W$. 
For every permutation $\sigma\in S_n$ we write 
\begin{equation}
L_{\sigma}^{(n)}:W^{\ot n}\to W^{\ot n}, \qquad 
w_1\ot\cdots\ot w_n\mapsto w_{\sigma^{-1}(1)}\ot\cdots\ot w_{\sigma^{-1}(n)},
\end{equation}
and
\begin{equation}
\begin{split}
 &\ev:W^*\ot W \to K, \qquad  f\ot v\mapsto f(v)  \\
 &\coev:K \to W\ot W^*, \qquad \coev(1) = \sum_i e_i\ot e^i,
 \end{split}
 \end{equation}
where $\{e_i\}$ is a basis for $W$ and $\{e^i\}$ is the dual basis of $W^*$.

\begin{definition}
The set of {\em constructible morphisms} $\mathrm{Const}(W,\{x_i\})\subseteq \bigsqcup_{p,q}\Hom(W^{\ot q},W^{\ot p})$ is the smallest collection of linear maps that contains the identity map on $W$ and the structure tensors, is closed under tensor products, tensor permutations, and evaluations. Explicitly, $\mathrm{Const}(W,\{x_i\})$ is the smallest collection that satisfies the following conditions:
\begin{enumerate}
\item $x_i:W^{\ot q_i}\to W^{\ot p_i}$ for $i\in I$, and $1_W:W\to W$ are constructible. 
\item If $x: W^{\ot a}\to W^{\ot b}$ and $y:W^{\ot c}\to W^{\ot d}$ are constructible, so is $x\ot y:W^{\ot (a+c)}\to W^{\ot (b+d)}$. 
\item If $x:W^{\ot a}\to W^{\ot b}$ is constructible, and $\sigma\in S_a$ and $\tau\in S_b$ are permutations, then $L^{(b)}_{\tau}xL^{(a)}_{\sigma}$ is  constructible. 
\item If $x:W^{\ot a}\to W^{\ot b}$ is constructible, then so is the map $W^{\ot (a-1)} \rightarrow W^{\ot (b-1)}$ given by 
\[
\begin{split}
W^{\ot (a-1)}\xrightarrow{1\ot \coev_W} &W^{\ot a}\ot  W^*\xrightarrow{x\ot 1}W^{\ot b}\ot W^* = \\ 
&=W^{\ot (b-1)}\ot W\ot W^*\xrightarrow{1\ot c_{W,W^*}} W^{\ot (b-1)}\ot W^*\ot W\xrightarrow{1\ot \ev} W^{\ot (b-1)},
\end{split}
\]
where $c_{W,W^*}:W\ot W^*\to W^*\ot W$ is given by $w\ot f\mapsto f\ot w$.  
\end{enumerate}
\end{definition}
\begin{remark}
    By abuse of notations, we will also refer to non-zero scalar multiples of constructible morphisms as constructible morphisms. 
\end{remark}

\begin{definition} Let $W$ be a finite dimensional $K$-vector space and let $x_i:W^{\ot q_i}\to W^{\ot p_i}$, $i\in I$, be a finite collection of structure tensors.
The group $\Aut(W,(x_i))$ is defined to be the group of all linear automorphisms $g$ of $W$ satisfying $g^{\ot p_i}x_i (g^{-1})^{\ot q_i} = x_i$ for all $i\in I$. 
\end{definition}
By fixing a basis of $W$, we can express the structure tensors using structure constants.  That allows us to think of the structure tensors as points in the affine space $\mathbb{A}^{N}$, where $N=\sum_i\dim(W)^{(p_i +q_i)}$. The group $\GL(W)$ acts on this space by a change of basis.
\begin{definition} We say that $(W,(x_i))$ has a closed orbit if the $\GL(W)$-orbit of $(x_i)$ in $\mathbb{A}^N$ is closed. 
\end{definition}

By the definition of constructible morphisms, it is easy to see that if $g\in \Aut(W,(x_i))$ then ${g^{\ot p}x (g^{-1})^{\ot q} = x}$ for every constructible $x:W^{\ot q}\to W^{\ot p}$. In particular, every constructible $x:W^{\ot q}\to W^{\ot p}$ is an $\Aut(W,(x_i))$-intertwiner.

\begin{theorem}[Theorem 1.2 in \cite{meir21}]\label{thm:of.udi} Assume that $(W,(x_i))$ has a closed orbit. Then the constructible elements in $\Hom(W^{\ot q}, W^{\ot p})$ span $\Hom_{\Aut(W,(x_i))}(W^{\ot q},W^{\ot p})$.
\end{theorem}
\begin{remark}\label{rem:interpolation}
In Section \ref{sec:interpolation} we  show how to interpolate the categories $\Rep(\Aut_R(M))$ into more general families of symmetric monoidal categories. 
\end{remark}

\subsection{The tautological representation $KM$ of $\Aut_R(M)$} The group $\Aut_R(M)$ has the tautological representation $W=KM=\Span_K\{\uu_m\}_{m\in M}$. This representation is faithful and therefore tensor-generates the category $\Rep(\Aut_R(M))$ (that is- every irreducible representation of $\Aut_R(M)$ will appear in some tensor power $(KM)^{\ot n}$). It is also equipped with a rich structure of a commutative cocommutative Hopf algebra, where the multiplication is given by $\mu(\uu_{m_1}\ot \uu_{m_2}) = \uu_{m_1+m_2}$, the comultiplication is given by $\Delta(\uu_m) = \uu_m\ot \uu_m$, the unit is $\eta(1)=\uu_0$, the counit is given by $\epsilon(\uu_m) = 1$, and the antipode is given by $S(\uu_m) = \uu_{-m}$. Moreover, $W$ is equipped with a family of operators $T_r:W\to W$ for $r\in R$ given by $T_r(\uu_m) = \uu_{rm}$, so the antipode $S$ is given by $T_{-1}$.
We have a natural identification $\Aut(W,\mu,\Delta,\eta,\epsilon,(T_r)) = \Aut_R(M)$. Indeed, any automorphism $\phi:M\to M$ of $R$-modules gives rise to a $K$-linear automorphism $\phi:KM\to KM$ that commutes with all the structure tensors, and if $\psi:KM\to KM$ is an automorphism that commutes with all the structure tensors, then it must preserve the group of group-like elements in $KM$, and thus induces a bijection $\wt{\psi}:M\to M$. Since $\psi$ commutes with the multiplication it follows that $\wt{\psi}$ is an automorphism of $M$ as an abelian group. Since $\psi$ commutes with all the operators $T_r$, for $r\in R$, it is in fact an automorphism of $R$-modules.    
All the structure tensors $\mu,\Delta,u,\epsilon,T_r$ can be thought of as $\Aut_R(M)$-intertwiners between tensor powers of $W$. 

The $\GL(W)$-orbit of $(W,\mu,\Delta,\eta,\epsilon,(T_r))$ is closed by the following argument. The condition of being a commutative cocommutative Hopf algebra is a closed condition. Every such Hopf algebra is of the form $KG$ for some finite abelian group $G$ (we can recover $G$ as the group of group-like elements in $W$). The automorphism group of such a Hopf algebra is then identifiable with the automorphism group of $G$, which is finite. Since the stabilizers of a point is identifiable with the group of automorphisms of the Hopf algebra given by the point, this implies that all stabilizers for the $\GL(W)$-action are finite, and therefore the dimensions of all the orbits of such Hopf algebras in $\mathbb{A}^{N}$ are equal. But this means that no orbit can be contained in the closure of another orbit, as this would imply that it has a smaller dimension.
We conclude the above discussion in the following corollary:
\begin{corollary}\label{cor:span.construct} 
Let $R$ be a finite ring, let $M$ be a finite $R$-module, and let $W=KM$ be equipped with the structure tensors $(\mu,\Delta, u,\epsilon,(T_r))$. Then the constructible maps $W^{\ot q}\to W^{\ot p}$ span $\Hom_{\Aut_R(M)}(W^{\ot q},W^{\ot p})$ for every $p,q\in \N$. 
\end{corollary}
\begin{proof}
This follows Theorem~\ref{thm:of.udi}, and from our observations above that the orbit of a commutative cocommutative Hopf algebra is necessarily closed, and that the automorphism group of $(W,\mu,\Delta,u,\epsilon,(T_r))$ is isomorphic to $\Aut_R(M)$. 
\end{proof}

\subsection{Some constructible maps}
We begin by proving that $W$ is self dual in $\Rep(\Aut_R(M))$. This can also be done directly, using character theory, but we prefer to do this using the language of constructible maps.  
Consider then the following constructible maps (read from bottom to top):
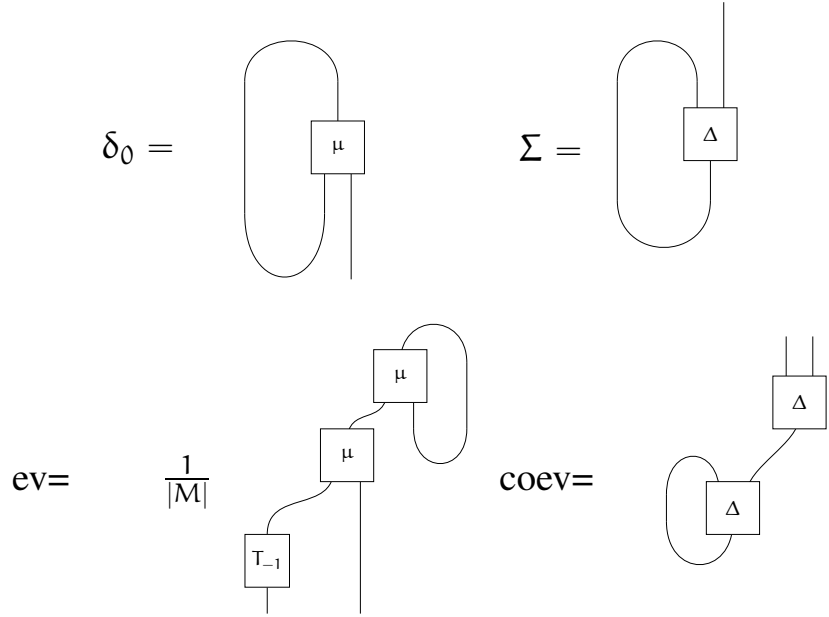
\begin{center}\scalebox{0.7}{
    \begin{tikzpicture}
	\begin{pgfonlayer}{nodelayer}
		\node [style=2function] (0) at (-9, 3.5) {$\mu$};
		\node [style=2function] (1) at (-2, 3.75) {$\Delta$};
		\node [style=none] (2) at (-8.75, 1) {};
		\node [style=none] (3) at (-9.25, 2.25) {};
		\node [style=none] (4) at (-9, 3.75) {};
		\node [style=none] (5) at (-1.75, 4) {};
		\node [style=none] (6) at (-2.25, 4) {};
		\node [style=none] (7) at (-2, 2.5) {};
		\node [style=none] (8) at (-9, 4.75) {};
		\node [style=none] (9) at (-10.75, 4.75) {};
		\node [style=none] (10) at (-10.75, 2.25) {};
		\node [style=none] (11) at (-2.25, 4.75) {};
		\node [style=none] (12) at (-3.75, 4.75) {};
		\node [style=none] (13) at (-3.75, 2.5) {};
		\node [style=none] (14) at (-2, 3.25) {};
		\node [style=none] (15) at (-9.25, 3.25) {};
		\node [style=none] (16) at (-8.75, 3.25) {};
		\node [style=none] (17) at (-1.75, 6.25) {};
		\node [style=none] (18) at (-12.75, 3.5) {\huge{$\delta_0=$}};
		\node [style=none] (19) at (-5, 3.5) {\huge{$\Sigma = $}};
	\end{pgfonlayer}
	\begin{pgfonlayer}{edgelayer}
		\draw (4.center) to (8.center);
		\draw [bend left=270, looseness=1.50] (8.center) to (9.center);
		\draw (9.center) to (10.center);
		\draw [bend right=90, looseness=2.75] (10.center) to (3.center);
		\draw (13.center) to (12.center);
		\draw [bend left=90, looseness=1.75] (12.center) to (11.center);
		\draw (11.center) to (6.center);
		\draw [bend left=90, looseness=1.75] (7.center) to (13.center);
		\draw (14.center) to (7.center);
		\draw (15.center) to (3.center);
		\draw (16.center) to (2.center);
		\draw (5.center) to (17.center);
	\end{pgfonlayer}
\end{tikzpicture}
}
\end{center}
\begin{center}\scalebox{0.7}{
\begin{tikzpicture}
	\begin{pgfonlayer}{nodelayer}
		\node [style=2function] (0) at (1.75, 3.5) {$\mu$};
		\node [style=2function] (1) at (0.75, 2) {$\mu$};
		\node [style=1function] (2) at (-0.75, 0) {$T_{-1}$};
		\node [style=none] (3) at (1, 1.75) {};
		\node [style=none] (4) at (0.5, 1.75) {};
		\node [style=none] (5) at (-0.75, 0.5) {};
		\node [style=none] (7) at (0.75, 2.25) {};
		\node [style=none] (9) at (2, 2.5) {};
		\node [style=none] (10) at (1.75, 3.75) {};
		\node [style=none] (11) at (3, 3.75) {};
		\node [style=none] (12) at (3, 2.5) {};
		\node [style=none] (13) at (2, 3.25) {};
		\node [style=none] (14) at (1.5, 3.25) {};
		\node [style=none] (15) at (1, -1) {};
		\node [style=none] (16) at (-0.75, -1) {};
		\node [style=none] (17) at (-2.25, 1.5) {\huge{$\frac{1}{|M|}$}};
		\node [style=none] (18) at (-5, 1.5) {\huge{ev=}};
		\node [style=none] (19) at (4.5, 1.5) {\huge{coev=}};
		\node [style=2function] (20) at (8, 1) {$\Delta$};
		\node [style=2function] (21) at (9.25, 3) {$\Delta$};
		\node [style=none] (22) at (8.25, 1.25) {};
		\node [style=none] (23) at (7.75, 1.25) {};
		\node [style=none] (24) at (8, 0.75) {};
		\node [style=none] (25) at (9.25, 2.75) {};
		\node [style=none] (26) at (9.5, 3.25) {};
		\node [style=none] (27) at (9, 3.25) {};
		\node [style=none] (28) at (6.75, 1.25) {};
		\node [style=none] (29) at (6.75, 0.75) {};
		\node [style=none] (30) at (9.5, 4.25) {};
		\node [style=none] (31) at (9, 4.25) {};
	\end{pgfonlayer}
	\begin{pgfonlayer}{edgelayer}
		\draw [in=-90, out=90, looseness=1.25] (5.center) to (4.center);
		\draw [bend left=90, looseness=2.00] (10.center) to (11.center);
		\draw (11.center) to (12.center);
		\draw [bend left=90, looseness=2.25] (12.center) to (9.center);
		\draw (13.center) to (9.center);
		\draw [in=90, out=-90, looseness=1.50] (14.center) to (7.center);
		\draw (15.center) to (3.center);
		\draw (5.center) to (2);
		\draw (16.center) to (2);
		\draw [bend left=270, looseness=2.50] (23.center) to (28.center);
		\draw (28.center) to (29.center);
		\draw [bend right=90, looseness=2.25] (29.center) to (24.center);
		\draw [in=-90, out=90, looseness=0.75] (22.center) to (25.center);
		\draw (31.center) to (27.center);
		\draw (30.center) to (26.center);
	\end{pgfonlayer}
\end{tikzpicture}
}
\end{center}
These maps are given explicitly by 
\[
\begin{split}
&\delta_0:KM \to K, \quad 
\delta_0(u_m) = |M|\delta_{0,m} \\
&\Sigma: K \to KM, \quad \Sigma(1) = \sum_m \uu_m \\ 
&\ev: KM \otimes KM \to K, \quad \ev(\uu_m\ot \uu_{m'})=\delta_{m,m'}, \quad \text{and} \\
& \coev: K \to KM \otimes KM, \quad \coev(1)=\sum_{m\in M} \uu_m\ot \uu_m .
\end{split}
\]
A direct verification shows that the last two morphisms establish a duality between $W$ and itself. We can thus identify $W$ and $W^*$, and we will do so henceforth.
This means, in particular, that we get an isomorphism $\Hom_{\Aut_R(M)}(K,W^{\ot a+b})$ and $\Hom_{\Aut_R(M)}(W^{\ot a}, W^{\ot b})$. This isomorphism sends a map $\phi:K\to W^{\ot a+b}$ to the map $$W^{\ot a} = W^{\ot a}\ot K\to W^{\ot a}\ot W^{\ot a + b} = W^{\ot a}\ot W^{\ot a}\ot W^{\ot b}\to W^{\ot b},$$ where the last map is given by applying evaluation $W^{\ot a}\ot W^{\ot a}\to K$ $a$ times. Since The evaluation is a constructible morphism, we have the following:
\begin{lemma}\label{lem:coevs}
    We have a bijection between the constructible morphisms $K\to W^{\ot a+b}$ and the constructible morphisms $W^{\ot a}\to W^{\ot b}$.   
\end{lemma}
We thus see that it is enough to study constructible elements in~$W^{\ot n}$ for $n\in\N$.  

Under this identification the multiplication $\mu$ corresponds to the element 
\[
\sum_{m_1+m_2=m_3}\uu_{m_1}\ot \uu_{m_2}\ot \uu_{m_3},
\]
 the comultiplication $\Delta$ corresponds to 
 \[
 \sum_{m\in M} \uu_m\ot \uu_m\ot \uu_m,
 \]
  the operator $T_r$ corresponds to $\sum_{m\in M} \uu_m\ot \uu_{rm}$, the unit $u$ is $\uu_0$ and the counit $\epsilon$ corresponds to $\sum_{m\in M} \uu_m$. 

\medskip

\begin{definition}Let $M_1,\ldots , M_k$  be finite indecomposable $R$-modules. We denote by $\C=\langle M_1,\ldots , M_k \rangle$ the full subcategory of the category of left $R$-modules consisting of modules of the form $\bigoplus_{i=1}^k M_i^{a_i}$, where $a_i\in \N$. We then write $\Cp$ for the category of $\o_\l$-linear functors $\Fun_{\o_\l}(\C,\o_\l\lmod)$.
\end{definition}
\begin{definition}
    For $n\in \N$ we write $\Fr_n:=\Hom_R(R^n,-)\in \Cp$. It holds that $\Fr_n(M)\cong M^n$.
\end{definition}
\begin{remark}    
We have $W^{\ot n} = KM^{\ot n} = KM^n.$
\end{remark}
\begin{definition}
If $F\subseteq \Fr_n$ is a subfunctor, then $KF(M)\subseteq W^{\ot n}$.
For every such subfunctor $F$ we write 
\[
\vv_F(M) = \sum_{(m_1,\ldots, m_n)\in F(M)}\uu_{m_1}\ot\cdots\ot \uu_{m_n} \in W^{\ot n}.
\]
If $a+b=n$ we write $$\vv_F^{a,b}(M) = \sum_{(m_1,\ldots, m_n)\in F(M)}\uu_{m_1}\ot\cdots\ot \uu_{m_a}\ot \uu^{m_{a+1}}\ot\cdots\ot \uu^{m_n}\in (KM)^{\ot a}\ot (KM^*)^{\ot b},$$
where $\{\uu^m\}_{m\in M}$ is the dual basis of $\{\uu_m\}_{m \in M}$ in $(KM)^*$. 
\end{definition}

If $M$ and $M'$ are two finite $R$-modules, we have a natural identification $K(M\oplus M')\cong KM\ot KM'$ given by $\uu_{(m,m')}\mapsto \uu_m\ot \uu_{m'}$. This gives an identification $(K(M\oplus M'))^{\ot n}\cong KM^{\ot n}\ot KM'^{\ot n}$. The additivity of $F$ implies that under this identification $\vv_F(M\oplus M')=\vv_F(M)\ot \vv_F(M')$. 
\begin{lemma}\label{lem:constructibles}
Every constructible element in $\Hom_{\Aut_R(M)}(\one,W^{\ot n})$ is of the form $\vv_F(M)$ for some subfunctor ${F\subseteq \Fr_n}$.
\end{lemma}
\begin{proof}
Every constructible element is built from the structure tensors, using tensor permutations, tensor products, and evaluations. 
We prove the lemma by showing that the tensors of the form $\vv_F$ contain the structure tensors and are closed under the above operations. 
We start with the structure tensors. The multiplication can be written as $\vv_{F_{\mu}}$ where $F_{\mu}(M) = \{(m_1,m_2,m_3) \mid m_i \in M, m_1+m_2=m_3\}\subseteq M^3$. 
The comultiplication can be written as $\vv_{F_{\Delta}}$ where $F_{\Delta}(M) = \{(m,m,m) \mid m \in M\}\subseteq M^3$. 
A similar argument works for the rest of the structure tensors. 
If $F\subseteq \Fr_n$ and $G\subseteq \Fr_{n'}$, then $\vv_F(M)\ot \vv_G(M) = \vv_{F\oplus G}(M)$, where $F\oplus G\subseteq \Fr_{n+n'}$. 
If $F\subseteq \Fr_n$, and $\sigma\in S_n$ then $\sigma$ induces a natural automorphism $\Fr_n\to \Fr_n$. This natural automorphism can be applied to the functor $F$, and we have the subfunctor $\sigma(F)\subseteq \Fr_n$. We have 
$$L^{(n)}_{\sigma}(\vv_F(M)) = \vv_{\sigma(F)}(M).$$

Finally,  we show that the set of tensors $\vv_F(M)$ is closed under evaluation. Since the set $\{\vv_F(M)\}$ is closed under tensor permutations it is enough to show that it is closed under evaluation of the last two coordinates. This linear operation is given explicitly by
\[
\ev(\uu_{m_1}\ot\cdots\ot \uu_{m_n}) = \delta_{m_{n-1},m_n}\uu_{m_1}\ot\cdots\ot \uu_{m_{n-2}}.
\]
The functor $F$ has a subfunctor $\wt{F}$ given by $$\wt{F}(M) = \{(m_1,\ldots, m_n)\mid (m_1,\ldots, m_n)\in F(M) \text{ and } m_{n-1}=m_n\}.$$
The functor $\wt{F}$ has a quotient functor $F'\subseteq \Fr_{n-2}$ that is given by the formula
$$F'(M) = \{(m_1,\ldots, m_{n-2})\mid \exists m'\in M : (m_1,\ldots, m_{n-2},m',m')\in F(M)\}.$$
A direct verification shows that both $\wt{F}$ and $F'$ are $\o_\l$-linear functors, and that 
\[
\begin{split}
\ev(\vv_F(M)) &= \ev\Bigl(\sum_{(m_1,\ldots,m_n)\in F(M)}\uu_{m_1}\ot\cdots\ot \uu_{m_n}\Bigr)\\
&=\ev\Bigl(\sum_{(m_1,\ldots,m_{n-2},m',m')\in \wt{F}(M)}\uu_{m_1}\ot\cdots\ot \uu_{m_{n-2}}\ot \uu_{m'}\ot \uu_{m'}\Bigr) \\
&=q^c\sum_{(m_1,\ldots, m_{n-2})\in F'(M)}\uu_{m_1}\ot\cdots\ot \uu_{m_{n-2}}\\&=  q^c \vv_{F'}(M),
\end{split}
\]

where $q^c$ is the cardinality of the kernel of $\wt{F}(M)\to F'(M)$. We know that the cardinality of this kernel is indeed a power of $q$, because this is true for any $\Ow_l$-module. Notice that $q^c$ depends on the module $M$.
\end{proof}

In view of Corollary~\ref{cor:span.construct}, Lemma~\ref{lem:constructibles} ensures that the set $\{\vv_F(M)\}_{F\subseteq \Fr_n}$ spans $\Hom_{\Aut_R(M)}(\one,W^{\ot n})$. The following proposition shows that for $M$ \lq large enough\rq~this set is in fact a basis for all modules in $\C$. 

\begin{proposition}\label{prop:basis} Let  $M = M_1^{a_1}\oplus\cdots\oplus M_k^{a_k}$ and let $n\in \N$. There is a constant $c(n)$ such that if $a_1,a_2,\ldots, a_k >c(n)$ then $\{\vv_F(M)\}_{F\subseteq \Fr_n}$ is a basis of $\Hom_{\Aut_R(M)}(\one,W^{\ot n})$. 
\end{proposition} 
\begin{proof}
We first claim that all the vectors $\{\vv_F(M)\}_{F\subseteq \Fr_n}$ are different when $a_i\gg 0$ for all $i=1,\ldots,k$. This follows immediately from the fact that if $F\neq F'$ then $F(M)\neq F'(M)$ as subgroups of $\Fr_n(M)$ for $M$ large enough.
Let $A$ be the set of all subfunctors of $\Fr_n$. Note that $A$ is a finite set, since every $F \in A$ is additive and hence completely determined by the submodules $F(M_1),\ldots, F(M_k)$ of the finite modules $M_1^n, \ldots, M_k^n$. 
Define the following subset of the power set of $A$ $$S(M) = \{B \subseteq A \mid  \{\vv_F(M)\}_{F\in B} \text{ is linearly dependent} \}\subseteq P(A).$$ We claim that if $S(M)\neq \varnothing$ then $S(M\oplus M)\subsetneq S(M)$. As all the relevant sets are finite, this will prove the claim. 

We first show that $S(M\oplus M)\subseteq S(M)$. Assume  that $B\notin S(M)$. This means that the set $\{\vv_F(M)\}_{F\in B}$ is linearly independent in $KM^{\ot n}$. This implies that the set $\{\vv_F(M)\ot \vv_F(M)\}_{F\in B}$ is linearly independent in $KM^{\ot n}\ot KM^{\ot n}$. By identifying the last vector space with $K(M\oplus M)^{\ot n}$ we get that the set $\{\vv_F(M\oplus M)\}_{F\in B}$ is linearly independent, and therefore $B \notin S(M\oplus M)$.

We now exhibit an element $B \in S(M) \smallsetminus S(M\oplus M)$. Let $B\in S(M)$ be a subset of minimal cardinality $a+1$. 
Write $B=\{F_1,\ldots, F_{a+1}\}$. By the minimality of $B$ we can write $\vv_{F_{a+1}}(M) = \sum_{i=1}^a \alpha_i \vv_{F_i}(M)$ with $\alpha_i\neq 0$ for $i=1,\ldots,a$. 
Using the identification $KM^{\ot n}\ot KM^{\ot n}\cong K(M\oplus M)^{\ot n}$  we get that $$\vv_{F_{a+1}}(M\oplus M) = \vv_{F_{a+1}}(M)\ot \vv_{F_{a+1}}(M) = \sum_{i,j=1}^a\alpha_i\alpha_j \vv_{F_i}(M)\ot \vv_{F_j}(M).$$ 
It holds that $\{F_1,\ldots, F_a\}\notin S(M)$ by the assumption on the cardinality of $B$. Therefore, $\{F_1,\ldots, F_a\}\notin S(M\oplus M)\subseteq S(M)$ as well, since $S(M\oplus M)\subseteq S(M)$. 
If $B\in S(M\oplus M)$ then we can also write 
\[
\vv_{F_{a+1}}(M\oplus M) = \sum_{i=1}^a\beta_i \vv_{F_i}(M\oplus M).
\] for some $\beta_i\in K$. 
By comparing the two linear expressions above, and using the fact that $\{\vv_{F_i}(M)\}_{i=1,\ldots, a}$ is linearly independent, we get that $\alpha_i\alpha_j=0$ if $i\neq j$ and $\alpha_i^2=\beta_i$. 
But since $\alpha_i\neq 0$ for $1 \leq i \leq a$, this is possible only if $a=1$. But this means that $\vv_{F_1}(M)$ and $\vv_{F_2}(M)$ are linearly dependent. Since the coefficients of all the basis vectors $\uu_{m_1}\ot\cdots \ot \uu_{m_n}$ in both $\vv_{F_1}(M)$ and $\vv_{F_2}(M)$ are all 1, this means that $\vv_{F_1}(M)=\vv_{F_2}(M)$. But we have already seen that this is impossible when $M$ is big enough, so we are done. 
\end{proof}
\begin{definition}
    Let $M$ be an $R$-module. Write $M = \bigoplus M_i^{a_i}$, where the $M_i$ are indecomposable modules. We say that an argument holds when $M$ is \textit{large enough} if there is some $c>0$ such that the argument is true whenever $a_i>c$ for every $i$.  
\end{definition}


\subsection{The representation $KF(M)$ as an object of $\Rep(\Aut_R(M))$}\label{the.rep.kfm}
The projective objects in the category $\Cp$ are all of the form $\Hom_R(X,-)$ and every functor is a quotient of such a Hom-functor; see Appendix~\ref{appendix}. 
\begin{lemma}\label{lem:subquo} Every functor in $\Cp$ is a subquotient of the functor $\Fr_n$ for some $n\in\N$.
\end{lemma}
 \begin{proof}
Let $F\in \Cp$. Since every functor is a quotient of a Hom-functor, we have an exact sequence of the form $$\Hom_R(Y,-)\to \Hom_R(X,-)\to F\to 0$$ for some $R$-modules $X$ and $Y$. 
Since $X$ is an $R$-module, we have a surjective map $R^n\to X$ for some $n\in\N$. This gives an inclusion of functors $\Hom_R(X,-)\to \Hom_R(R^n,-) = \Fr_n$. 
This implies that $F$ is a quotient of $\Hom_R(X,-)$, which in turn is a subfunctor of $\Fr_n$, so $F$ is a subquotient of $\Fr_n$ indeed. 
\end{proof}
The lemma gives a concrete way to study functors in $\Cp$, as they are subquotients of concrete functors. 
We will begin with studying the subfunctors of $Fr_n$. Every subfunctor $F\subseteq Fr_n$ is the image of some natural transformation $\Hom_R(X,-)\to Fr_n$. By Yoneda's Lemma, such a natural transformation is of the form $f^*$ for some $f:R^n\to X$. Let $\{ e_1, \ldots, e_n\}$ be the standard basis of $R^n$ and write $f(e_i)=x_i\in X$. We have
\[
F(M) = \{(\phi(x_1),\ldots, \phi(x_n)) \mid \phi:X\to M\}\subseteq M^n.
\]
We have the following equality:
$$\vv_F(M) = \frac{1}{|\Hom_R(X/\langle x_1,\ldots, x_n\rangle,M)|}\sum_{\phi:X\to M}\uu_{\phi(x_1)}\ot\cdots\ot \uu_{\phi(x_n)}.$$

We thus see that every subfunctor $F\subseteq Fr_n$ arises from an $R$-module $X$ and a tuple $(x_1,\ldots, x_n)\in X^n$. In such case we write $F=Fun_{(X,x_1,\ldots, x_n)}$. 

\begin{lemma}
    The morphism $\vv_F(M)$ is constructible for every $F\subseteq \Fr_n$ 
\end{lemma}
\begin{proof}
    We know that $F$ can be written as $Fun_{(X,x_1,\ldots, x_n)}$ for some $X\in\C$ and $x_1,\ldots, x_n\in X$. Complete the set $\{x_1,\ldots, x_n\}$ to a generating set $\{x_1,\ldots, x_n,\ldots, x_a\}$ of $X$ as an $R$-module. It then holds that $X$ is a quotient of $R^a$, and we can write $X\cong R^a/(\sum_j r_{ij}e_j)_{i=1}^b$ for some $r_{ij}\in R$. It then holds that a tuple $(m_1,\ldots, m_a)$ in $M^a$ lies in $Fun_{(X,x_1,\ldots, x_a)}$ if and only if $\sum_j r_{ij}m_j=0$ for $i=1,\ldots, b$. 

    We claim first that $\vv_{F'}(M)$ is constructible, where $F'= Fun_{(X,x_1,\ldots, x_a)}\subseteq \Fr_a$. Indeed, by rearranging the tensors in $\Delta^{b}(\Sigma^{\ot a})$ we get the constructible map 
    $$\sum_{m_1,\ldots, m_a}(\uu_{m_1}\ot\cdots\ot \uu_{m_a})^{\ot {b+1}}.$$ Applying the operators $T_{r_{ij}}$ and the multiplication operator, we get the constructible map 
    $$\sum_{m_1,\ldots, m_a} \uu_{m_1}\ot\cdots\ot \uu_{m_a}\ot \uu_{f_1}\ot\cdots\ot \uu_{f_b},$$ where $f_i = \sum_j r_{ij}m_j$. By applying $\frac{1}{|M|^b}\delta_0$ on the last $b$ tensor factors we get 
    $$\sum_{\substack{m_1,\ldots, m_a\\ \forall i, \sum_j r_{ij}m_j=0}} \uu_{m_1}\ot\cdots\ot \uu_{m_a},$$ which is exactly $\vv_{F'}(M)$.
     To get down from $a$ to $n$ we just apply $\epsilon^{a-n}$ on the last $a-n$ tensor factors in the above constructible map. In this way we get $\vv_F(M)$ as a constructible morphism.
\end{proof}

\medskip

Let $F\subseteq Fr_n$. Our next goal is to write $KF(M)$ and $K(Fr_n/F)(M)$ as the images of projections $KM^n\to KM^n$ (recall that $KM^n = K(M^n)$). These projections are $\Aut_R(M)$-equivariant, and so, by Theorem \ref{thm:of.udi} we know that they are spanned by constructible morphisms.  
 Since $\Hom_{\Aut_R(M)}(KM^n,KM^n)\cong \Hom_{\Aut_R(M)}(\one, KM^{2n})$, such a projection will be given by a linear combination of $\vv_{G}$, where $G$ ranges over subfunctor $G\subseteq Fr_{2n}$. We will show now that in fact a single subfunctor is enough. We define 
\[
\begin{split}
F'(M) &= \{(m_1,\ldots, m_n,m_1,\ldots, m_n) \mid (m_1,\ldots, m_n)\in F(M)\}\subseteq M^{2n}, ~\text{ and } \\
F''(M) &= \{(m_1,\ldots, m_n,m_{n+1}\ldots, m_{2n}) \mid (m_1-m_{n+1},\ldots, m_n-m_{2n})\in F(M)\}\subseteq M^{2n}.
\end{split}
\]
The following lemma is straightforward.
\begin{lemma}
The constructible morphism $\vv_{F'}^{n,n}(M):KM^n\to KM^n$ is a projection with image $span\{\uu_{m_1}\ot\cdots\ot \uu_{m_n}\}_{(m_1,\ldots, m_n)\in F(M)}$ and kernel $span\{\uu_{m_1}\ot\cdots\ot \uu_{m_n}\}_{(m_1,\ldots, m_n)\notin F(M)}$. 
The morphism $\frac{1}{|F(M)|}\vv_{F''}^{n,n}(M)$ is the projection
$$\uu_{m_1}\ot\cdots\ot \uu_{m_n}\mapsto \frac{1}{|F(M)|}\sum_{(z_1,\ldots, z_n)\in F(M)} \uu_{m_1+z_1}\ot\cdots\ot \uu_{m_n+z_n}.$$ 
\end{lemma}
\begin{definition}
For $F\subseteq Fr_n$ we write $P'_F(M)=\vv_{F'}^{n,n}(M)$ and $P''_F(M) = \frac{1}{|F(M)|}\vv^{n,n}_{F''}(M)$ for the projections from the lemma above.
\end{definition}
\begin{lemma}
Assume that $F_1\subseteq F_2\subseteq Fr_n$. Then $P'_{F_2}(M)$ and $P''_{F_1}(M)$ commute. The composition $P'_{F_2}(M)P''_{F_1}(M)$ is then a projection $KM^n\to KM^n$, whose image is isomorphic to $K(F_2(M)/F_1(M))$ as $\Aut_R(M)$-representations.
\end{lemma}
\begin{proof}
We first prove that the two operators commute. Since $F_1\subseteq F_2$ it holds that for every $(z_1,\ldots, z_n)\in F_1(M)$ we have $(m_1,\ldots, m_n)\in F_2(M)$ if and only if $(m_1+z_1,\ldots, m_n+z_n)\in F_2(M)$. We then have 
$$P'_{F_2}(M)P''_{F_1}(M)(\uu_{m_1}\ot\cdots\ot \uu_{m_n}) = \frac{1}{|F_1(M)|}\sum_{(z_1,\ldots, z_n)\in F_1(M)}P'_{F_2}(M)(\uu_{m_1+z_1}\ot\cdots\ot \uu_{m_n+z_n}).$$
If $(m_1,\ldots, m_n)\in F_2(M)$ then this is equal to 
$$\frac{1}{|F_1(M)|}\sum_{(z_1,\ldots, z_n)\in F_1(M)}\uu_{m_1+z_1}\ot\cdots\ot \uu_{m_n+z_n} = P''_{F_1}(M)P'_{F_2}(M)(\uu_{m_1}\ot\cdots\ot \uu_{m_n}).$$
If $(m_1,\ldots, m_n)\notin F_2(M)$ the latter equals $0$, which is the same as $P''_{F_1}(M)P'_{F_2}(M)(\uu_{m_1}\ot\cdots\ot \uu_{m_n})$, as desired.

Since the composition of any two commuting projections is a projection, we only need to prove the assertion about the image. For this, notice that the image is spanned by vectors of the form $\sum_{(z_1,\ldots, z_n)\in F_1(M)} \uu_{m_1+z_1}\ot\cdots\ot \uu_{m_n+z_n}$ where $(m_1,\ldots, m_n)\in F_2(M)$, and this is canonically isomorphic with $K(F_2(M)/F_1(M))$. 
\end{proof} 
\begin{definition} Let $F,G\in \Cp$ be functors, and let $H\subseteq F\oplus G$ be a subfunctor. We write 
\[
\begin{split}
T_H(M):KF(M)&\to KG(M) \\
\uu_c &\mapsto \sum_{\substack{d\in G(M) \\ (d,c)\in H(M)}} \uu_d.
\end{split}
\]
\end{definition}

\begin{proposition}\label{prop:hom.space.basis}
The linear maps $\{T_H(M)\}_{H\subseteq G\oplus F}$ span $\Hom_{\Aut_R(M)}(KF(M),KG(M))$. If $M$ is large enough then the above set is also linearly independent. 
\end{proposition}
\begin{proof}
By Lemma \ref{lem:constructibles} and Lemma \ref{lem:coevs} We know that $\Hom_{\Aut_R(M)}(KM^n,KM^m)$ is spanned by elements of the form $\vv_H^{n,m}(M)$ where $H\subseteq Fr_{n+m}$ is a subfunctor. Moreover, by Proposition \ref{prop:basis} we know that if $M$ is large enough then these elements are also linearly independent.
Write $F= F_2/F_1$ and $G= G_2/G_1$, where $F_1\subseteq F_2\subseteq \Fr_n$ and $G_1\subseteq G_2\subseteq \Fr_m$ are functors. 
We then have $KF(M)\cong  P'_{F_2}(M)P''_{F_1}(M)KM^n$ and $KG(M) \cong P'_{G_2}(M)P''_{G_1}(M)KM^m$. This implies that 
\[
\begin{split}
\Hom_{\Aut_R(M)}(KF(M),KG(M)) &\cong \Hom_{\Aut_R(M)}(P'_{F_2}(M)P''_{F_1}(M)KM^n,P'_{G_2}(M)P''_{G_1}(M)KM^m) \\
&=P'_{G_2}(M)P''_{G_1}(M)\Hom_{\Aut_R(M)}(KM^n,KM^m)P'_{F_2}(M)P''_{F_1}(M).
\end{split}
\]
Since $P'_{G_2}(M)P''_{G_1}(M)$ and $P'_{F_2}(M)P''_{F_1}(M)$ are projections, we get from the above isomorphism that the above hom space is spanned by elements of the form $P'_{G_2}(M)P''_{G_1}(M)\vv_H^{n,m}(M)P'_{F_2}(M)P''_{F_1}(M)$ where $H\subseteq \Fr_{n+m}$ is a subfunctor. 

 Let now $H\subseteq Fr_{n+m}$. We have 
\[
P'_{G_2}(M)P''_{G_1}(M)\vv_H^{n,m}(M)P'_{F_2}(M)P''_{F_1}(M)=P'_{G_2}(M)P''_{G_1}(M)\vv_{H\cap (G_2\oplus F_2)}^{n,m}(M)P'_{F_2}(M)P''_{F_1}(M).
\]
Moreover, if $H\subseteq G_2\oplus F_2$ then 
$$P'_{G_2}(M)P''_{G_1}(M)\vv_H^{n,m}(M)P'_{F_2}(M)P''_{F_1}(M)=P''_{G_1}(M)\vv_H^{n,m}(M)P''_{F_1}(M).$$
Similarly, 
$$P''_{G_1}(M)\vv_H^{n,m}(M)P''_{F_1}(M)=P''_{G_1}(M)\vv_{H+(G_1\oplus F_1)}^{n,m}(M)P''_{F_1}(M).$$
Moreover, if $H\supseteq G_1\oplus F_1$ then it holds that  $P''_{G_1}(M)\vv_H^{n,m}P''_{F_1}(M)=\vv_H^{n,m}$.
The result follows now easily from the third isomorphism theorem applied to the quotient $G_2\oplus F_2/G_1\oplus F_1 = G\oplus F$. 
\end{proof}
\begin{remark}
    By a slight abuse of notation, we will call the morphisms $T_H(M):KF(M)\to KG(M)$ constructible. 
\end{remark}

\begin{remark} The elements $\vv_F$ appeared first in the work of Knop, under the name correspodences, see \cite{Knop2007}. In the framework of \cite{Knop2007}, one starts with the category $\C_{+1}$, and construct the categories $\Rep(\Aut_R(M))$ and their interpolations, since the category of projective objects in $\C_{+1}$ can be identified with the category $\C$. See \S\ref{sec:interpolation} for more details.
 \end{remark}
 We finish this section with some specific morphisms. Let $\alpha:F\to G$ be a natural transformation of functors. We write 
 \[
 \begin{split}
    \alpha_*:KF(M)&\to KG(M) \qquad \qquad \alpha^*:KG(M)\to KF(M) \\
    \uu_c &\mapsto \uu_{\alpha_M(c)}\qquad \qquad \qquad \quad \quad \enspace \uu_d \mapsto \sum_{\substack{c\in F(M) \\ \alpha_M(c) = d}} \uu_c.
 \end{split}
 \]
  
  If we write $H\subseteq G\oplus F$ for the subfunctor given by $\{(\alpha_M(c),c)\}_{c\in F(M)}$, then $\alpha_* = T_H$. Similarly, $\alpha^* = T_H$ if we consider $H$ as a subfunctor of $F\oplus G$ in the obvious way.
Moreover, a direct calculation shows that if $H\subseteq G\oplus F$ is any functor, and we write $\alpha:H\to G\oplus F\to F$ and $\beta:H\to G\oplus F\to G$ for the composition of the inclusion of $H$ in $G\oplus F$ with the projections onto $F$ and $G$ respectively, then $T_H  =\beta_*\alpha^*$. 

The following lemma will be crucial in the proof of the main result of the next section:
\begin{lemma}\label{lem:pushouts}
Consider the commutative diagram:
    $$\xymatrix{ H\ar[r]^{\beta}\ar[d]^{\alpha} & F_2\ar[d]^{\gamma} \\ F_1\ar[r]^{\delta} & G }.$$
    \begin{enumerate}
        \item If the diagram is a pullback then $\gamma$ is injective if and only if $\alpha$ is injective.
        \item If the diagram is a pushout then $\gamma$ is surjective if and only if $\alpha$ is surjective.
        \item If the diagram is a pullback then $\gamma^*\delta_* = \beta_*\alpha^*$.
        \item If the diagram is a pushout then $|\Ker(\alpha_M)\cap \Ker(\beta_M)|\gamma^*\delta_* = \beta_*\alpha^*$. In particular, if $H\subseteq F_1\oplus F_2$ then $\gamma^*\delta_* = \beta_*\alpha^*$
    \end{enumerate}
\end{lemma}
 \begin{proof}
 If the diagram is a pullback then $H(M) = \{(m_1,m_2)| \delta(m_1) = \gamma(m_2)\}\subseteq (F_1(M)\oplus F_2(M))$. If the diagram is a pushout then $G(M) \cong (F_1(M)\oplus F_2(M))/ ((\alpha(m),-\beta(m)))_{m\in H(M)}$. The first two assertions are well known and follow froma direct calculation.

If the diagram is a pullback then 
$$\gamma^*\delta_*(\uu_m) = \gamma^*(\uu_{\delta(m)}) = \sum_{\substack{m': \\  \gamma(m') = \delta(m)}} \uu_{m'},$$
$$\beta_*\alpha^*(\uu_m) = \beta^*(\sum_{\substack{m': \\  \gamma(m') = \delta(m)}} \uu_{(m,m')}) = \sum_{\substack{m': \\  \gamma(m') = \delta(m)}} \uu_{m'},$$ so we get the desired equality. 

If the diagram is a pushout then we get 
 $$\gamma^*\delta_*(\uu_m) = \gamma^*(\uu_{\ol{(m,0)}}) = \sum_{\substack{m': \exists m''\in H(M): \\  \beta(m'')= m', \alpha(m') = m}} \uu_{m'},$$
 $$\beta_*\alpha^*(\uu_m) = \beta^*(\sum_{\substack{m'' :\\ \alpha(m'')= m}} \uu_{m''}) = \sum_{\substack{m'' :\\ \alpha(m'')=m}} \uu_{\beta(m'')}$$
The first part follows from the fact that $G / F_1\oplus F_2/(\alpha(h),-\beta(h))_{h\in H}$. It then follows directly that if $m\in F_1(M)$ then $$\gamma^*\delta_*(\uu_m) = \gamma^*(\uu_{(\alpha(m),0)}) = \sum_{m'\in H(M), \beta(m')= \alpha(m) \uu_{\beta(m')})} = \beta_*\alpha^*(\uu_m).$$
We get the same sum, except that in the first sum we count every element once, and in the second sum every element that appears with non-zero coefficient appears $|\Ker(\alpha_M)\cap \Ker(\beta_M)|$ times. This finishes the proof of the lemma. 
 \end{proof}

 \section{Stratification on the irreducible representations and proofs of Theorems A and B}\label{sec:stratification}
Let $R$ be a finite ring. Let $M_1,\ldots, M_n$ be pairwise non-isomorphic finite indecomposable $R$-modules, and let $\C = \langle M_1,\ldots, M_n\rangle$ be the category of all $R$-modules of the form $\bigoplus_i M_i^{a_i}$. We write $\Cp=\Fun_{\Ow_l}(\C,\Ow_l-\Mod)$ for the category of $\Ow_l$-linear functors from $\C$ to $\Ow_l$-mod.  
We have a natural partial order on the functors in $\Cp$ defined as follows. 

\begin{definition} Let $\C$ be as above. For  $F,G \in \Cp$ we write $F \preccurlyeq G$ if $F$ is isomorphic to a subquotient of~$G$.  If $F \preccurlyeq G$ but $G \not \cong F$ we write $F \prec G$.
\end{definition}

It is easy to see that the relation $\preccurlyeq$ is reflexive and transitive. Moreover, if $F \preccurlyeq G$ and $G \preccurlyeq F$ then $F \cong G$. This follows from the fact that all modules we consider here are finite. It follows that $\preccurlyeq$ induces a well-defined partial order on isomorphism types of functors.

\begin{definition} Let $\C$ be as above and $M \in \mathrm{ob}(\C)$. An irreducible representation $V$ of $\Aut_R(M)$ is called {\em $F$-typical} for $F \in \mathrm{ob}(\Cp)$ if $V$ is isomorphic to a subrepresentation of $KF(M)$ and, for every $G\prec F$, $V$ is not isomorphic to a subrepresentation  of $KG(M)$. For a functor $F\in \mathrm{ob}(\Cp)$ and $M\in \C$ we write $\ol{KF(M)}$ for the maximal subrepresentation of $KF(M)$ such that $\Hom_{\Aut_R(M)}(KGM,\ol{KF(M)})=0$ for every $G \prec F$.
\end{definition}
We will prove that every irreducible representation is $F$-typical for a unique (up to isomorphism) functor $F$.

\begin{proposition}\label{prop:largest.F} Let $M$ be an object in $\C$ and $F_1,F_2 \in \Cp$. Let $\theta: KF_1(M)\to KF_2(M)$ be a constructible morphism. Then either $\theta$ is induced from an isomorphism $F_1\cong F_2$, or there is a functor $F_3\preccurlyeq F_1,F_2$ and constructible morphisms $\theta_1:KF_1(M)\to KF_3(M)$ and $\theta_2:KF_3(M)\to KF_2(M)$ such that $\theta = \theta_2\theta_1$, and $F_3$ is strictly smaller than $F_1$ or $F_2$.  
\end{proposition}
\begin{proof}
We know that any constructible morphism is of the form $T_H$, where $H\subseteq F_1\oplus F_2$. Write $\alpha:H\to F_1\oplus F_2\to F_1$ and $\beta:H\to F_1\oplus F_2\to F_2$ for the compositions of the inclusion of $H$ into $F_1\oplus F_2$ with the projections onto $F_1$ and $F_2$ respectively. Then it holds that $T_H = \beta_*\alpha^*$. Since $\Cp$ is an abelian category, we can write $\alpha = \alpha_1\alpha_2$ and $\beta=\beta_1\beta_2$ where $\alpha_1:F_3\to F_1$ is injective, $\alpha_2:H\to F_3$ is surjective, $\beta_1:F_4\to F_2$ is injective and $\beta_2:H\to F_4$ is surjective. 
We argue by induction on $s(F_1) + s(F_2)$, where $s(F)$ is the length of the composition sequence of $F$. If $s(F_1) + s(F_2)=0$ then the result holds trivially. In the general case, assume that we already know that the result holds for all pairs of functors $F_1',F_2'$ such that $s(F_1')+s(F_2')<s(F_1)+s(F_2)$. If $\alpha_1$ is not an isomorphism, then $s(F_3)< s(F_1)$. We can then write $T_H$ as $T_H= (\beta_1)_*(\beta_2)_*(\alpha_2)^*(\alpha_1)^*$. Since $s(F_3)+s(F_2) < s(F_1) + s(F_2)$ the  induction hypothesis tells us that the constructible  morphism $(\beta_1)^*(\beta_2)_*(\alpha_2)^*:KF_3(M)\to KF_2(M)$ can be written as a composition of two constructible morphisms of the form $KF_3(M)\stackrel{\theta_1}{\to} KF_5(M)\stackrel{\theta_2}{\to} KF_2(M)$ where $F_5\preccurlyeq F_3,F_2$. But then $T_H$ can be written as the composition $\theta_1\theta_2(\alpha_1)^*:KF_1(M)\to KF_5(M)\to KF_2(M)$, and it holds that $F_5 \preccurlyeq F_1,F_2$ and $F_5\prec F_1$, so we are done in this case. We can thus assume that $\alpha_1$ is an isomorphism, or, in other words, that $\alpha$ is surjective. By a similar argument we can assume that $\beta$ is surjective as well.

We thus reduce to the case where both $\alpha$ and $\beta$ are surjective. 
In this case, we have the following push-out diagram:
$$\xymatrix{ H\ar[r]^{\beta}\ar[d]^{\alpha} & F_2\ar[d]^{\gamma} \\ F_1\ar[r]^{\delta} & G }.$$
By Lemma \ref{lem:pushouts} $\gamma$ and $\delta$ must be surjective as well, and it holds that 
$$T_H = \beta_*\alpha^* = \gamma^*\delta_*.$$ Since $G$ is a common quotient of both $F_1$ and $F_2$ it holds that $G\preccurlyeq F_1,F_2$. If $G\cong F_1\cong F_2$ then $T_H$ is induced from an isomorphism between $F_1$ and $F_2$ indeed. If $G\prec F_1$ or $G\prec F_2$ then we get a splitting of $T_H$ via a functor that is strictly smaller than at least one of $F_1$ or $F_2$, as desired. 
\end{proof}

\begin{theorem}
    Let $V$ be an irreducible representation of $\Aut_R(M)$. Then there is a unique functor $F\in \Cp$ such that $V$ is $F$-typical. 
\end{theorem}
\begin{proof}
Let $V$ be an irreducible representation of $\Aut_R(M)$. 
Since $KM$ tensor-generates $\Rep(\Aut_R(M))$, $V$ appears in $(KM)^{\ot n} = K\Fr_n(M)$ for some $n\in\N$. Let now $F\in \Cp$ be a minimal functor such that $V$ appears in $KF(M)$. Then $V$ is $F$-typical. We claim that $F$ is unique. Indeed, assume that there is another functor $G\in \Cp$ such that $V$ is $G$-typical, and $F\ncong G$. Then there is a non-zero morphism $f\in \Hom_{\Aut_R(M)}(V,KF(M))$ and a non-zero $g\in \Hom_{\Aut_R(M)}(KF(M),KG(M))$ such that $gf\neq 0$. Since the constructible morphisms span $\Hom_{\Aut_R(M)}(KF(M),KG(M))$, we can assume without loss of generality that $g$ is constructible. But then, by Proposition \ref{prop:largest.F} we know that there is a functor $F'\preccurlyeq F,G$ such that $g$ splits as $$g=  KF(M)\stackrel{g_1}{\to} KF'(M)\stackrel{g_2}{\to} KG(M).$$ Since $gf=g_2g_1f\neq 0$, it holds that $g_1f:V\to KF'(M)\neq 0$. The minimality of $F$ and of $G$ implies that $F\cong F'\cong G$, a contradiction.
\end{proof}
  
\begin{definition}If $V$ is $F$-typical, we call $F$ the associated functor of $V$. We denote by $\Irr(\Aut_R(M))_F$ the subset of $F$-typical representations of $\Irr(\Aut_R(M))$, and denote by $\Rep(\Aut_R(M))_F$ the full abelian subcategory they generate.
\end{definition}
It follows that 
\[
\Rep(\Aut_R(M)) = \bigoplus_{F\in \mathrm{ob}(\Cp)} \Rep(\Aut_R(M))_F.
\]
Notice that $\ol{KF(M)} = KF(M)\cap \Rep(\Aut_R(M))_F$. In other words, $\ol{KF(M)}$ is the biggest $F$-typical subrepresentation of $KF(M)$.
We now give another interpretation of the set $\Irr(\Aut_R(M))_F$.

\subsection{The epimorphism theorem}
Fix a functor $F\in \Cp$.  We have a natural algebra map 
\[
\begin{split}
    K\Aut_{\Cp}(F)&\to \End_{\Aut_R(M)}(KF(M)) \\
\uu_{\alpha}&\mapsto \alpha_*
\end{split}
\] 
that restricts to an algebra map $$\Phi_F(M):K\Aut_{\Cp}(F)\to \End_{\Aut_R(M)}(\ol{KF(M)})$$ given explicitly by $\Phi_F(M)(\uu_{\alpha})(\uu_m)= \uu_{\alpha_M(m)}$ where $\alpha\in \Aut_{\Cp}(F)$ and $m\in F(M)$. 

\begin{theorem}\label{thm:epimorphism} The algebra homomorphism $\Phi_F(M)$ is surjective for every $M$. If $M$ is large enough then $\Phi_F(M)$ is also injective. 
\end{theorem}
\begin{proof}
By Proposition \ref{prop:hom.space.basis} we know that the endomorphism algebra $\End_{\Aut_R(M)}(KF(M))$ is spanned by morphisms of the form $T_H(M)$ for $H\subseteq F\oplus F$. In Proposition \ref{prop:largest.F} we have seen that there is a dichotomy among such morphisms. Either they split via $KF'(M)$ for some $F'\prec F$, or they are induced by an automorphism of $F$. We will write $B_1$ for the set of the first type of morphisms and $B_2$ for the set of the second type of morphisms. Thus, $\{T_H(M)\}_{H\subseteq F\oplus F} = B_1\sqcup B_2$.

The representation $\ol{KF(M)}$ is a characteristic subrepresentation of $KF(M)$, and restriction therefore gives us a surjective algebra morphism $$p:\End_{\Aut_R(M)}(KF(M))\to \End_{\Aut_R(M)}(\ol{KF(M)}).$$ By the definition of $\ol{KF(M)}$, and by semisimplicity, the kernel of $p$ is spanned by all morphisms of the form $KF(M)\to V\to KF(M)$, where $V$ is an irreducible $F'$-typical representation, for some $F'<F$. This means that $\Ker(p)$ is spanned by morphisms that split via $KF'(M)$ for some $F'\prec F$. But this is exactly $\Span B_1$. It follows that the image of $p$ is spanned by $p(B_2)$. But this is the same as the image of $\Phi_F(M)$, proving surjectivity.

For the injectivity, assume that $M$ is big enough. Proposition \ref{prop:hom.space.basis} tells us that $\{T_H(M)\}_{H\subseteq F\oplus F}$ is linearly independent. Then the fact that $\Ker(p)$ is spanned by $B_1$ implies that $p(B_2)$ is a linearly independent set. Since $p(B_2)$ is the same as the image of the group elements in $\Aut_{\Cp}(F)$ under $\Phi_F(M)$, we get that $\Phi_F(M)$ is also injective in this case.    
\end{proof}

\subsection{Consequences of the epimorphism theorem}\label{subsec:con.epi}
Let $F\in\Cp$. Write $\{V_i\}$ for a set of representatives of the $F$-typical irreducible representations of $\Aut_R(M)$. We can decompose $\ol{KF(M)}$ as an $\Aut_R(M)$ representation as $$\ol{KF(M)} = \bigoplus_i V_i\ot W_i,$$ where $W_i$ are vector spaces. It holds that $\End_{\Aut_R(M)}(\ol{KF(M)})\cong \bigoplus_i \End_K(W_i)$. Since the action of $\Aut_{\Cp}(F)$ on $\ol{KF(M)}$ commutes with the action of $\Aut_R(M)$, the vector spaces $W_i$ are also $\Aut_{\Cp}(F)$-representations. Write $\{U_j\}$ for a set of representatives of the irreducible $\Aut_{\Cp}(F)$-representations. Wedderburn decomposition gives $$K\Aut_{\Cp}(F)\cong \bigoplus_j \End_K(U_j).$$ The fact that $\Phi_F(M)$ is surjective implies that for every $i$ there is a unique $j$ such that the direct summand $\End_K(U_j^*)$ is mapped isomorphically onto $\End_K(W_i)$ by $\Phi_F(M)$, and thus $W_i\cong U_j^*$.  

\begin{definition} Using the terminology above, we call the correspondence $(M,V_i)\leadsto (F,U_j)$ the functor morphing  of the representation $V_i$ of $\Aut_R(M)$.
\end{definition}
This gives us, in particular, a natural identification of $\Rep(\Aut_R(M))_F$ with a subcategory of $\Rep(\Aut_{\Cp}(F))$. 
\begin{remark}\label{rem:whydual}
With the above notation we have $\Hom_{\Aut_R(M)}(\ol{KF(M)},V_i) = W_i^* = U_j$. 
We write the isomorphism $W_i\cong U_j^*$ rather than $W_i\cong U_j$, to ensure that when we have trivial functor morphing, every representation maps to itself rather than its dual. 
\end{remark}
%


%
Define $$\Rep(\Aut_R(M))_{\preccurlyeq F} = \bigoplus_{G\preccurlyeq F} \Rep(\Aut_R(M))_G,$$ 
$$\Rep(\Aut_R(M))_{\prec F} = \bigoplus_{G \prec F} \Rep(\Aut_R(M))_G.$$

\begin{proposition}\label{prop:coherent.order} If $V_1\in \Rep(\Aut_R(M))_{\leqq F_1}$ and $V_2\in \Rep(\Aut_R(M))_{\leqq F_2}$ then $V_1\ot V_2\in \Rep(\Aut_R(M))_{\preccurlyeq F_1\oplus F_2}$. A similar result holds for strict inequalities. 
\end{proposition}
\begin{proof}
We can assume, without loss of generality, that $V_1$ and $V_2$ are irreducible. 
    The conditions of the proposition then imply that $V_i$ is a subrepresentation of $KF_i(M)$ for $i=1,2$. Then $V_1\ot V_2$ is a subrepresentation of $KF_1(M)\ot KF_2(M)\cong K(F_1\oplus F_2)(M)$, which means that $V_1\ot V_2$ is in $\Rep(\Aut_R(M))_{\preccurlyeq F_1\oplus F_2}$ as desired.
\end{proof}
Recall that if $G_1$ and $G_2$ are finite groups, and if $V_i$ is a representation of $G_i$ for $i=1,2$, then we write $V_1\boxtimes V_2$ for the $G_1\times G_2$ representation $V_1\times V_2$ with the action $(g_1,g_2)\cdot (v_1\ot v_2) = g_1v_1\ot g_2v_2$. 

\begin{proposition}
Assume that $V_i\in \Rep(\Aut_R(M))_{F_i}$ for $i=1,2$. In other words, assume that $(M,V_i)\leadsto (F_i,\wt{V_i})$ where $\wt{V_i}$ is an irreducible representation of $\Aut_{\Cp}(F_i)$ for $i=1,2$.  
Then the projection of $V_1\ot V_2$ on $\Rep(\Aut_R(M))_{F_1\oplus F_2}$ is given by 
$$\cJ(\Ind_{\Aut(F_1)\times \Aut(F_2)}^{\Aut(F_1\oplus F_2)}(\wt{V_1}\boxtimes \wt{V_2}))$$ where $\cJ$ is the left adjoint of the inclusion $\Rep(\Aut_R(M))_{F_1\oplus F_2}\to \Rep(\Aut(F_1\oplus F_2))$. 
\end{proposition}
\begin{proof}
We first give a general way to write down the adjoint functor $\cJ$. If $\wt{V}=K\Aut_{\Cp}(F)e$ is an irreducible representation of $\Aut_{\Cp}(F)$, where $e$ is some idempotent, then $$\cJ(\wt{V}) = \text{Im}(eKF(M)\to \ol{KF(M)}) = \Phi_F(M)(e)\ol{KF(M)}.$$ We can write $\wt{V_i} = K\Aut_{\Cp}(F_i)e_i$ where $e_i\in K\Aut_{\Cp}(F_i)$ are idempotents. 
It holds that $e_iKF_i(M) = V_i\oplus V_i'$ where $V_i'\in \Rep(\Aut_R(M))_{\prec F_i}$ for $i=1,2$. 
This, together with Proposition \ref{prop:coherent.order}, implies that $e_1KF_1(M)\ot e_2KF_2(M) = (V_1\ot V_2)\oplus V''$, where $V''\in \Rep(\Aut_R(M))_{\prec F_1\oplus F_2}$. It follows that the projection of $e_1KF_1(M)\ot e_2KF_2(M)\cong (e_1\ot e_2)K(F_1\oplus F_2)(M)$ on $\Rep(\Aut_R(M))_{F_1\oplus F_2}$ is the same as the projection of $V_1\ot V_2$ on that subcategory. 

On the other hand, 

\[
\begin{split}
\Ind_{\Aut(F_1)\times \Aut(F_2)}^{\Aut(F_1\oplus F_2)}(\wt{V_1}\boxtimes \wt{V_2})) &= \Ind_{\Aut(F_1)\times \Aut(F_2)}^{\Aut(F_1\oplus F_2)}K\Aut_{\Cp}(F_1)e_1\ot K\Aut_{\Cp}(F_2)e_2 \\
&=K\Aut_{\Cp}(F_1\oplus F_2)(e_1\ot e_2),
\end{split}
\]
and therefore 
$\cJ(\Ind_{\Aut(F_1)\times \Aut(F_2)}^{\Aut(F_1\oplus F_2)}(\wt{V_1}\boxtimes \wt{V_2}))$ is the projection of $(e_1\ot e_2)K(F_1\oplus F_2)(M)$ on $\Rep(\Aut_R(M))_{F_1\oplus F_2}$, and so we get the desired equality.  
%
    \end{proof}



\section{The partial order and a proof of Theorem \ref{thm:main3}}\label{sec:p.order} 
We now use the technique of functor morphing  to link representations across a hierarchy of families of automorphism groups. 

\subsection{Degrees and their partial order}

\begin{definition}\label{def:degree} Let $R$ be a finite $\o_\l$-algebra and $\C$ a full subcategory of $R\lmod$ with finitely many indecomposable objects. Let $M$ be a finite $R$-module in $\C$ and $V$ be an irreducible representation of $\Aut_R(M)$. The {\em degree} of $(M,V)$ is the triplet 
\[
\deg(M,V) =(\dim V, |\Aut_R(M)|, |\Supp(M)|),
\]    
where $\Supp(M)$ is the number of isomorphism classes of indecomposable $R$-modules that are direct smmands of $M$.
\end{definition}

\begin{remark} Given $\C=\langle M_1,\ldots,M_n\rangle$ as in definition~\ref{def:degree}, the category $\Cp$ is abelian (kernels and cokernels are taken pointwise), and is equivalent to the category $R'\lmod$ with $R'=\End_R(\oplus_i M_i)$ (see \cite[Exercise 2,P. 138]{GelfandManin}. The equivalnce sends a functor $F$ to the $R'$-module $F(\bigoplus_i M_i)$). We define the degree of $(F,U)$, where $F \in \Cp$ and $U$ an irreducible representation of $\Aut_{\Cp}(F)$ as the degree of their corresponding interpretation in the category of modules, or directly as $(\dim U, |\Aut_{\Cp}(F)|,|\mathrm{Supp}(F)|)$.

\end{remark}

For the rest of this section, let $V$ be an irreducible representation of $\Aut_R(M)$ with functor morphing $(M,V)\leadsto (F,\wt{V})$. 
We will show that the degree of $(F,\wt{V})$ is less than or equal to the degree of $(M,V)$ with respect to the lexicographical order, and show that the case where the degrees are the same is very particular. 

We study the functor $F$ with tools from homological algebra. To simplify notation, let $\D$ be an $\o_\l$-linear abelian category with enough projectives, in which all the hom-spaces are finite and every object has finitely many subobjects (The category $\Cp$ is an example for such a category). 
Recall that a \textit{projective cover} of an object $F$ in $\D$ is a surjective morphism $p:P\to F$ with $P$ projective such that no direct summand of $P$ maps onto $F$. 

The projective objects of $\Cp$ are all of the form $\Hom_{\C}(X,-)$; see Proposition~\ref{prop:projectives.in.C+1}. Choose a projective resolution 
\begin{equation}\label{proj.res.F}
\Hom_{\C}(Y,-)\xrightarrow{f^*} \Hom_{\C}(X,-)\xrightarrow{p} F\to 0 
\end{equation}
of $F$ such that $\Hom_{\C}(X,-)$ is a projective cover of $F$, and $\Hom_{\C}(Y,-)$ is a projective cover of $\Ker(p)$. 
Write $X = \bigoplus X_i$ and $Y=\bigoplus_j Y_j$ for the decomposition of $X$ and $Y$ into indecomposable objects. The morphism $f:X\to Y$ can then be written as the direct sum of $f_{ij}:X_i\to Y_j$. We will need the following lemma in the proof of Theorem \ref{thm:triv.red}
\begin{lemma}\label{lem:non.triv.comp} Assume that $Y\neq 0$. 
    For every $j$ there is an $i$ such that $f_{ij}\neq 0$. As a result, if $X_i\cong Y_j$ for some $i,j$ then there is a morphism $g:Y\to X$ such that $gf\neq 0$
\end{lemma}
\begin{proof}
    Assume that for some $j$ it holds that $f_{ij}=0$ for all $i$. Then the map $f:X\to Y$ splits as $$X\stackrel{f_1}{\to} \bigoplus_{k\neq j} Y_k\stackrel{f_2}{\to} Y.$$ It then holds that $f^* = f_1^*f_2^*$, and since $f_2$ is a direct summand inclusion, $f_2^*$ is surjective. It follows that $\Im(f^*) = \Im(f_1^*),$ contradicting the minimality of the projective cover of $\Ker(p)$. 

    For the second part, if $\phi:Y_j\to X_i$ is an isomorphism, define $g$ to be $g|_{Y_j} = \phi$ and $g|_{Y_k}=0$ for $k\neq j$.  
\end{proof}

We now prove some structural results that follow from the fact that $(M,V)\leadsto(F,\wt{V})$. We need the following definition.
\begin{definition}
Let $F\in \mathrm{ob}(\Cp)$. A subgroup $H$ of $\Aut_R(M)$ is called {\em $F$-adequate} if there is $a\in F(M)$ such that $\stab_{\Aut_R(M)}(a) = H$. 
\end{definition}
\begin{lemma}\label{lem:invs}
Let $U\in \Irr(\Aut_R(M))$. Then $\Hom(KF(M),U)\neq 0$ if and only if 
$U^H\neq 0$ for some $F$-adequate subgroup $H$ of $\Aut_R(M)$. 
As a result, $\Hom(\ol{KF(M)},U)\neq 0$ if and only if $U^H\neq 0$ for some $F$-adequate subgroup $H$ of $\Aut_R(M)$ and $U^{H'}=0$ for every subgroup $H'<\Aut_R(M)$ that is $F'$-adequate for some $F' \prec F$. 
\end{lemma}
\begin{proof}
Recall that $F(M)$ is a finite set with an $\Aut_R(M)$-action. Write $F(M) = \bigsqcup_a \Aut_R(M)\cdot a$ where the sum is taken over representatives of the different $\Aut_R(M)$-orbits in $F(M)$. Write $H_a = \stab_{\Aut_R(M)}(a)$ so that $\Aut_R(M)\cdot a\cong \Aut_R(M)/H_a$. We have 
$$KF(M) = \bigoplus_a K\Aut_R(M)/H_a,$$
and therefore $\Hom_{\Aut_R(M)}(KF(M),U)\neq 0$ if and only if $\Hom_{\Aut_R(M)}(K\Aut_R(M)/H_a,U)\neq 0$ for some~$a$. By Frobenius reciprocity, the last hom-space is isomorphic to $\Hom_{H_a}(\one,U)=U^{H_a}$, which implies that ${\Hom_{\Aut_R(M)}(KF(M),U)\neq 0}$ if and only if $U^{H_a}\neq 0$ for some $a$. Since all the $F$-adequate subgroups are conjugate the to $H_a$ for some $a$, the assertion follows.
\end{proof}

We claim the following:
\begin{lemma}\label{lem:direct.sum}
The object $X$ in \eqref{proj.res.F} is a direct summand of $M$.
\end{lemma}
\begin{proof}
Since $\Hom_{\Aut_R(M)}(KF(M),V)\neq 0$, Lemma \ref{lem:invs} implies that there exists an element $a\in F(M)$ such that $V^{\stab_{\Aut_R(M)}(a)}\neq 0$. Since $p:\Hom_{\C}(X,-)\to F$ is surjective, there is $\phi\in \Hom_{\C}(X,M)$ such that $p(\phi)=a$. We claim that ${\phi:X\to M}$ is a split injection.

For this, consider the composition $p' = p\phi^*:\Hom_{\C}(M,-)\to \Hom_{\C}(X,-)\to F$. Write $F'=\Im(p')\subseteq F$. 
Then $a = p'(\Id_M)\in F'(M)$. Since $V^{\stab_{\Aut_R(M)}(a)}\neq 0$, this implies that $\Hom_{\Aut_R(M)}(KF'(M),V)\neq 0$. By the definition of functor morphing , this can only happen if $F'=F$, as otherwise $F$ is not the minimal functor for which $V$ appears in $KF(M)$. Since $\Hom_{\C}(X,-)$ is a projective cover of $F$, this means that $\Hom_{\C}(X,-)$ is a direct summand of $\Hom_\C(M,-)$. But by Yoneda's Lemma, this happens if and only if $X$ is a direct summand of $M$. 
\end{proof}

The following lemma facilitates the description of the representation $\wt{V}$:
\begin{lemma}
There exists exactly one $\Aut_R(M)$-orbit $\Aut_R(M)\cdot a$ in $F(M)$ such that $V^{\stab_{\Aut_R(M)}(a)}\neq 0$. 
\end{lemma}
\begin{proof}
Let $a\in F(M)$ be an element that satisfies the condition of the lemma. 
We have seen that if we write $a= p(\phi)$ for some $\phi:X\to M$, then $\phi$ is a split injection. By Krull-Schmidt we know that if $M \cong X\oplus Z\cong X\oplus Z'$ then $Z\cong Z'$. This implies that all the split injections in $\Hom_{\C}(X,M)$ are in the same $\Aut_R(M)$-orbit. By the naturality of $p$ we know that $p(M):\Hom(X,M)\to F(M)$ is $\Aut_R(M)$-equivariant. Since $p(M)$ is also surjective, the result follows. 
\end{proof}
We call such an element $a\in F(M)$ a \textit{generic point}. 
We end this section with a concrete description of the representation $\wt{V}$ in terms of the data we have:
\begin{proposition}
Write $H\subseteq \Aut_R(M)$ for the stabiliser of a generic point $a\in F(M)$. Then $\wt{V} = V^{H}$.
\end{proposition}
\begin{proof}
By Remark~\ref{rem:whydual} we have $\wt{V} = \Hom_{\Aut_R(M)}(\ol{KF(M)},V)$, where the action of $\Aut(F)$ is induced from pre-composing on $F(M)$.  Since $F$ is the smallest functor such that $\Hom_{\Aut_R(M)}(KF(M),V)\neq 0$, we have $\Hom_{\Aut_R(M)}(\ol{KF(M)},V) \cong \Hom_{\Aut_R(M)}(KF(M),V)$.
Now, $KF(M)$ splits as the direct sum of modules of the form $K\Aut_R(M)/\stab_{\Aut_R(M)}(x)$, where we take $x$ to be orbit representatives for the action of~$\Aut_R(M)$ on~$F(M)$. We have  
$$\Hom_{\Aut_R(M)}(K\Aut_R(M)/\stab_{\Aut_R(M)}(x),V)\cong \Hom_{\stab_{\Aut_R(M)}(x)}(\one,V)\cong V^{\stab_{\Aut_R(M)}(x)}.$$ By the previous lemma, only the orbit of the generic points will give a non-zero contribution. Write $a\in F(M)$ for a generic point. We then get 
\[
\begin{split}
\wt{V} & = \Hom_{\Aut_R(M)}(\ol{KF(M)},V)\cong \Hom_{\Aut_R(M)}(KF(M),V) \\ 
& \cong \Hom_{\Aut_R(M)}(K\Aut_R(M)/\stab_{\Aut_R(M)}(a),V) \\
& \cong \Hom_{\stab_{\Aut_R(M)}(a)}(\one,V)\cong V^{\stab_{\Aut_R(M)}(a)}. \qedhere
\end{split}
\]
\end{proof}
This already shows that $\dim(\wt{V})\leq \dim(V)$. Next, we describe the group $\Aut(F)$ and its action on $\wt{V}$, and study in detail the case where $\dim(V) = \dim(\wt{V})$. 


\subsection{The group $\Aut(F)$}

Consider the projective resolution \eqref{proj.res.F} and let $\phi:F\to F$ be an automorphism. By the lifting lemma from homological algebra we know that there are maps $\phi_X:X\to X$ and $\phi_Y:Y\to Y$ such that the following diagram is commutative:
\[
\xymatrix{ \Hom_{\C}(Y,-)\ar[r]^{f^*}\ar[d]^{\phi_Y^*} & \Hom_{\C}(X,-)\ar[d]^{\phi_X^*}\ar[r]^{\qquad p} & F \ar[r]\ar[d]^{\phi} & 0 \\ \Hom_{\C}(Y,-)\ar[r]^{f^*} & \Hom_{\C}(X,-)\ar[r]^{\qquad p} & F \ar[r] & 0 }
\]
By Lemma~\ref{lem:lifting.isos} we know that $\phi_X$ is an automorphism of $X$. Write 
\[
A_F:=\{\phi_X \mid \phi:F\to F\text{ is an automorphism}\}\subseteq \Aut_{\C}(X).
\]
 Let $g:Y\to X$ be a morphism. We have a commutative diagram 
\[
\xymatrix{ \Hom_{\C}(X,-)\ar[r]^{\qquad p}\ar[d]^{(1+gf)^*} & F\ar[d]^{1_F} \\ \Hom_{\C}(X,-)\ar[r]^{\qquad p} & F}
\]
By Lemma~\ref{lem:lifting.isos} again, this implies that $1+gf$ is invertible. 
Moreover, the projectivity of $\Hom_{\C}(X,-)$ implies that the kernel of $A_F\to \Aut(F)$ is exactly $B_F:= \{1+gf \mid g:Y\to X\}$. 
We conclude this discussion in the following lemma: 

\begin{lemma}\label{lem:aut.presentation}
We have an isomorphism $\Aut(F)\cong A_F/B_F$. 
\end{lemma}
\hfill\qed
\begin{corollary}
We have $|\Aut(F)|\leq |\Aut_R(M)|$. 
\end{corollary}
\begin{proof}
The group $\Aut(F)$ is a subquotient of $\Aut_{\C}(X)$, and therefore $|\Aut(F)|\leq |\Aut_{\C}(X)|$. Since $X$ is a direct summand of $M$ we can write $M=X\oplus Z$. We can then extend every automorphism of $X$ to an automorphism of $M$ that acts as the identity on $Z$. This gives an embedding of $\Aut_{\C}(X)$ into $\Aut_{\C}(M)$, and therefore $|\Aut_{\C}(X)|\leq |\Aut_{\C}(M)|$. Combining the two inequalities gives the corollary. 
\end{proof}  






The next theorem is the main result of this section, and concludes the proof of Theorem \ref{thm:main3}:
\begin{theorem}\label{thm:triv.red}
Let $(M,V)\leadsto (F,\wt{V})$ be a functor morphing  and let~\eqref{proj.res.F} be a projective resolution of $F$. Then $\deg(V,M)\leq \deg(\wt{V},F)$ in the lexicographical order. If the degrees are equal and all the indecomposable modules in $\C$ are direct summands of $M$, then in the resolution~\eqref{proj.res.F} we have $Y = 0$ and $X=M$. In this case we have a natural isomorphism between $\Aut_{\Cp}(F)$ and $\Aut_R(M)$, and $\wt{V}\cong V$. 
\end{theorem}
\begin{proof}
We have already seen that $\dim(\wt{V})\leq \dim(V)$ and that $|\Aut(F)|\leq |\Aut_R(M)|$. In case either of these inequalities is strict, the degree of the functor morphing  decreases, namely, $\deg(\wt{V},F) < \deg(V,M)$. Assume now that $\dim(\wt{V}) = \dim(V)$ and $|\Aut(F)|= |\Aut_R(M)|$. Recall from Lemma~\ref{lem:direct.sum} that $M=X\oplus Z$. The fact that $|\Aut_R(M)| = |\Aut_R(X)|$ implies that $\Aut_R(Z)=1$ and that $\Hom_R(X,Z) = \Hom_R(Z,X)=0$. In particular, this means that if $M_i$ is a direct summand of $X$ and $M_j$ is a direct summand of $Z$ then $\Hom_R(M_i,M_j) = \Hom_R(M_j,M_i)=0$. Write $Z = L_1^{b_1}\oplus\cdots\oplus L_k^{b_k}$ as a direct sum of indecomposable modules. The fact that $\Aut_R(Z)=1$ implies that $b_i=1$ for all $i$, as otherwise we have non-trivial automorphisms arising from permuting the isomorphic direct summands. It also implies that for every $i$ the group $\Aut_R(L_i)$ is trivial, which happens if and only if $\End_R(L_i) = \F_2$.
In Lemma~\ref{lem:aut.presentation} we have described $\Aut(F)$ as a subquotient of $\Aut_R(X)$. If the two groups are equal, this also means that the group $B_F$ is trivial. But this implies that $gf=0$ for every $g:Y\to X$. In particular, Lemma \ref {lem:non.triv.comp} implies that if $Y\neq 0$ then there is no indecomposable module that appears as a direct summand of $Y$ and of $X$. We know that $M = X\oplus Z = X\oplus L_1\oplus \cdots\oplus L_k$. Since all the indecomposable objects in $\C$ are the indecomposable direct summands of $M$, this means that $Y$ has the form $L_1^{c_1}\oplus\cdots\oplus L_k^{c_k}$ for some $c_i\in \N$. But we saw above, in our discussion about $\Hom_R(X,Z)$, that this implies that $\Hom_R(X,Y)=0$, so $f=0$ and in fact $Y=0$ in this case.

This implies that $F=\Hom_R(X,-)$. The indecomposable direct summands of $F$ as an object in $\Cp$ are then in one to one correspondence with the direct summands of $X$ in $\C$. There are exactly $n-k$ of them, where $n=\Supp(M)$, so $\Supp(F) = n-k$. If $k>0$ then $\deg(V,M)<\deg(\wt{V},F)$. 
If $k=0$ then $Z=0, X=M, Y=0$, and under the isomorphism $\Aut_{\Cp}(\Hom(M,-))\cong \Aut_R(M)$ the representation $\wt{V}$ corresponds to $V$. 
\end{proof}
 \begin{remark}
     Notice that in fact we know that both $\dim(\wt{V})\leq \dim(V)$ and $|\Aut_R(M)|\leq |\Aut_{\Cp}(F)|$. We find it easier to work with the lxicographical order, as this is a complete order.     
 \end{remark}

\begin{definition}
If $(M,V)\leadsto (F,\wt{V})$ where $\wt{V} = V$ and $F=\Hom_{\C}(M,-)$ we will say that $V$ admits a trivial functor morphing . 
\end{definition}

\section{The directed graph on indecomposables and A proof of Theorem \ref{thm:main4}}\label{sec:graph}
In this section we identify the cases in which there is no reduction to a lower degree structure and show how they can be dealt with. To every tuple $(M_1,\ldots, M_n)$ of non-isomorphic indecomposable $R$-modules we associate a directed graph, the graph of left maximal morphisms. We then show that the only case which does not admit a reduction to a smaller structure using functor morphing or Clifford theory, is when this graph is a disjoint union of graphs of the form 
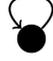
\begin{center}
    \begin{tikzpicture}
	\begin{pgfonlayer}{nodelayer}
		\node [style=graph node] (0) at (-1.5, 1.75) {};
	\end{pgfonlayer}
	\begin{pgfonlayer}{edgelayer}
		\draw [style=new edge style 0, in=135, out=45, loop] (0) to ();
	\end{pgfonlayer}
\end{tikzpicture}

\end{center}

We introduce the graph $LM(M_1,\ldots, M_n)$ in \S\ref{subsec:gen.cong.sub}. Then, we divide the proof of Theorem~\ref{thm:main4} into two parts: we deal with the case that $LM(M_1,\ldots, M_n)$ contains a vertex without outgoing edges in 
 \S\ref{subsec:no.out.arrows}. If every vertex in $LM(M_1,\ldots, M_n)$ contains an outgoing edge, and there is a representation $V\in\Irr(\Aut_R(M))$ that admits trivial functor morphing, we show in \S\ref{subsec:there.is.outgoing} that $LM(M_1,\ldots, M_n)$ must be a disjoint union of cycles. Finally, in \ref{subsec:big.cycle} we show that if one of these cycles has length $>1$ then we can use Clifford theory to reduce the representation $V$ to a representation of a proper subgroup of $\Aut_R(M)$, that is isomorphic to $\Aut_{\wt{R}}(\wt{M})$ for some auxiliary ring $\wt{R}$ and auxiliary $\wt{R}$-module $\wt{M}$.

\subsection{Left maximal morphisms and generalized congruence subgroups}\label{subsec:gen.cong.sub} 
We now focus on the case of trivial functor morphing  and show that it is very restrictive. Throughout this section we fix an $R$-module $M$ and write $M = M_1^{a_1}\oplus\cdots\oplus M_n^{a_n}$ for the decomposition of $M$ into indecomposable modules.
We describe the group $\Aut_R(M)$ explicitly. 
Write 
\begin{equation}\label{eqn:blocks}
    \End_R(M) = \bigoplus_{i,j} \Hom_{\C}(M_j^{a_j},M_i^{a_i}),
\end{equation}
and consider the following subgroups of $\Aut_R(M)$
\[
\begin{split}
L &= \prod_i \Aut_{\C}(M_i^{a_i}), \\
U &= \{1_M + \sum_{i<j} f_{ij} \mid f_{ij}\in \Hom_{\C}(M_j^{a_j},M_i^{a_i})\}, \\
\ol{U} &=  \{1_M + \sum_{i>j} f_{ij} \mid f_{ij}\in \Hom_{\C}(M_j^{a_j},M_i^{a_i})\}.
\end{split}
\] 

Recall that the endomorphism ring of an indecomposable module is a local ring. 
\begin{definition}
We write $J_i$ for the Jacobson radical of the local ring $\End_{\C}(M_i)$. 
We write $E_i:=\End_{\C}(M_i)/J_i$. \end{definition}
\begin{remark} Since $E_i$ is a finite division algebra, it must be a finite field by Wedderburn's Theorem. 
\end{remark}
Recall from \cite{CMO1} that a triplet of subgroups $(U,L,\ol{U})$ of a finite group $G$ is an {\em Iwahori} decomposition if $L$ normalises $U$ and $\ol{U}$ and the map $U \times L \times \ol{U} \to G$ induced by multiplication is a bijection. 

\begin{lemma} 
The triplet $(U,L,\ol{U})$ is an Iwahori decomposition of the group $\Aut_R(M)$.
\end{lemma}
\begin{proof}
We can visualise elements of the group $\Aut_R(M)$ as block matrices. As such, $L$ is the subgroup of block diagonal matrices, $U$ is the subgroup of upper triangular matrices with 1 on the diagonal and $\ol{U}$ is the subgroup of lower triangular matrices with 1 on the diagonal. 
This already makes it easy to see that $L$ normalises $U$ and~$\ol{U}$, and that the product gives an injective map $U \times L \times \ol{U}\to \Aut_R(M)$. It remains to show that this map is surjective. Let $g\in \Aut_R(M)$ and write $g= (g_{ij})$ with $g_{ij}\in \Hom_{\C}(M_j^{a_j},M_i^{a_i})$. We claim that for every $i$, $g_{ii}$ is invertible. Using Gauss elimination it is then straightforward to show that $\Aut_R(M) = U \cdot L \cdot \ol{U}$. 

Write $g^{-1}=(h_{ij})$. It follows that \begin{equation}1_{M_i^{a_i}} = \sum_k g_{ik}h_{ki}.\end{equation} The algebra $\Hom_{\C}(M_i^{a_i},M_i^{a_i})$ can be identified with the algebra $\M_{a_i}(\End_{\C}(M_i))$. 
If $k\neq i$, then every composition $M_i\to M_k\to M_i$ is contained in $J_i$, the Jacobson radical of $\End_{\C}(M_i)$, since if such decomposition was invertible it would follow that $M_i$ is a direct summand of $M_k$ which contradicts the fact that $M_i$ and $M_k$ are two non-isomorphic indecomposable objects.  Reducing the above equation modulo $J_i$ we get $$1 = \sum_k \ol{g_{ik}}\ol{h_{ki}} = \ol{g_{ii}}\ol{h_{ii}}.$$
This means that $\ol{g_{ii}}$ is invertible in $\M_{a_i}(\End_{\C}(M_i)/J_i)$, which implies that $g_{ii}$ is invertible in $\M_{a_i}(\End_{\C}(M_i))$ as well. 
\end{proof}
\begin{definition} Fix a tuple $(M_1,\ldots, M_n)$ of non-isomorphic indecomposable $R$-modules. A non-zero morphism $f:M_j\to M_i$ is called {\em left maximal} if it is not invertible and $gf=0$ for every non-invertible morphism $g:M_i\to M_k$. If $X=M_1^{b_1}\oplus\cdots\oplus M_n^{b_n}$ then a morphism $f:X\to M_i$ is called left maximal if it corresponds to a tuple of left maximal morphisms under the isomorphism $\Hom_{\C}(X,M_i)\cong \prod_j \Hom_{\C}(M_j,M_i)^{b_i}$. We write $LM(X,M_i)$ for the set of all left maximal morphisms from $X$ to $M_i$. 

A {\em right maximal} morphism is defined in a similar way: it is a non-zero morphism $f:M_j\to M_i$ which is not invertible and $fg=0$ for every non-invertible morphism $g:M_k\to M_j$.
\end{definition}

\begin{example} Let $R=\o_\l$ and $M_i=\o_i$ for $1 \leq i \leq \l$. Then  
\[
LM(M_i,M_j)=
\begin{cases} 
\emptyset, \quad &\text{if $j<\l$}; \\ 
\{x \mapsto a x \mid \mathrm{val}(a)=\l-1\}, \quad &\text{if $j=\l$},  \end{cases}
\]
where $\mathrm{val}(a)$ is the valuation of $a$, that is, $a \in \pi^{\mathrm{val}(a)}\o^\times_\l$.
\end{example}

\begin{remark} If there exists any non-invertible non-zero map $M_i\to M_j$ for some $j$, then the condition that an LM-morphism  into $M_i$ should be non-invertible is redundant.
\end{remark}
\begin{lemma}
The set $LM(X,M_i) \cup \{0\}$ is a vector space over $E_i$. 
\end{lemma}
\begin{proof}
It is easy to see that this set is closed under addition. It is also easy to see that it is a module over $\End_{\C}(M_i)$ by post-composing. By the definition of left maximal morphisms, $J_i$ acts trivially on this module and we thus get a structure of an $E_i$-vector space. 
\end{proof}
\begin{definition}
The {\em LM directed graph $\Ga=\Ga(M_1,\ldots, M_n)$} is the graph with vertices $V(\Ga)=\{1,\ldots, n\}$, and $\dim_{E_i}(LM(M_j,M_i))$ directed arrows from the vertex $j$ to the vertex $i$. For an arrow $e$ in $\Ga$ we write $s(e)$ for the source of $e$ and $t(e)$ for the target of $e$. 
\end{definition}
\begin{example} 
Consider the ring $R=\Ow_1[x,y]/(x^2,y^2,xy)$. This is a commutative local ring that is 3-dimensional over $\Ow_1$. We define $M_1 = R$, $M_2 = R/(x)$ and $M_3 = R/(x,y)$. We describe now the graphs $\Ga(M_1,M_2,M_3)$ and $\Ga(M_1,M_3)$. 

We first describe all the hom-spaces. This is relatively easy to do, since all the modules are cyclic with a canonical choice of a generator (i.e. the image of $1\in R$). We have
%
\[
\left(\Hom(M_i,M_j)\right)_{i,j}=
\left(\begin{matrix} 
\Span\{\phi^{11}_1,\phi^{11}_x,\phi^{11}_y\} &  \Span\{\phi^{12}_1,\phi^{12}_y\} &  \Span\{\phi^{13}_1\} \\
\Span\{\phi^{21}_x,\phi^{21}_y\} & \Span\{\phi^{22}_1,\phi^{22}_y\} & \Span\{\phi^{23}_1\}\\
\Span\{\phi^{31}_x,\phi^{31}_y\} & \Span\{\phi^{32}_y\} &\Span\{\phi^{33}_1\}  
\end{matrix}\right)
\]

where $\phi^{ij}_a:M_i\to M_j$ denotes the morphism given by multiplying by $a\in R$ (a direct verification shows that this is well defined in all the cases we have here).
We next determine what are the left maximal morphisms. We do that by just considering all the possible compositions of morphisms. We get that the left maximal morphisms are 
$$\{\phi^{11}_x, \phi^{12}_y,\phi^{21}_x,\phi^{22}_y,\phi^{31}_y,\phi^{32}_y\}.$$
The graph $\Ga(M_1,M_2,M_3)$ is therefore 
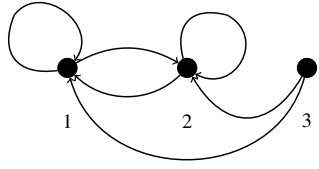
\begin{center}\scalebox{0.7}{
    \begin{tikzpicture}
	\begin{pgfonlayer}{nodelayer}
		\node [style=graph node] (0) at (-7, 1.25) {};
		\node [style=graph node] (1) at (-4.75, 1.25) {};
		\node [style=graph node] (2) at (-2.5, 1.25) {};
		\node [style=none] (3) at (-8, 2.25) {};
		\node [style=none] (4) at (-4, 2.25) {};
		\node [style=none] (5) at (-7, 0.25) {1};
		\node [style=none] (6) at (-4.75, 0.25) {2};
		\node [style=none] (7) at (-2.5, 0.25) {3};
	\end{pgfonlayer}
	\begin{pgfonlayer}{edgelayer}
		\draw [style=new edge style 0, bend left=45] (1) to (0);
		\draw [style=new edge style 0, bend left] (0) to (1);
		\draw [style=new edge style 0, bend left=90, looseness=1.75] (4.center) to (1);
		\draw [style=new edge style 0, bend left=90, looseness=1.50] (3.center) to (0);
		\draw [style=new edge style 0, bend left=60, looseness=1.50] (2) to (1);
		\draw [style=new edge style 0, bend left=75, looseness=1.25] (2) to (0);
		\draw [style=thickedge, bend left=300, looseness=1.50] (4.center) to (1);
		\draw [style=thickedge, bend right=60, looseness=1.50] (3.center) to (0);
	\end{pgfonlayer}
\end{tikzpicture}
}
\end{center}

Consider now the graph $\Ga(M_1,M_3)$. Since the definition of a left maximal morphism depends on the particular set of indecomposable modules that we have, some more morphisms will be left maximal now. Indeed, a direct calculation shows that the morphisms $$\{\phi^{11}_x,\phi^{11}_y,\phi^{31}_x,\phi^{31}_y\}$$ are the left maximal in this case, and the graph $\Ga(M_1,M_3)$ is thus
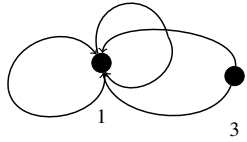
\begin{center}\scalebox{0.7}{
    \begin{tikzpicture}
	\begin{pgfonlayer}{nodelayer}
		\node [style=graph node] (0) at (1.5, 1.25) {};
		\node [style=graph node] (2) at (4, 1) {};
		\node [style=none] (3) at (-0.25, 0.75) {};
		\node [style=none] (5) at (1.5, 0.25) {1};
		\node [style=none] (7) at (4, 0) {3};
		\node [style=none] (8) at (2.75, 1.75) {};
	\end{pgfonlayer}
	\begin{pgfonlayer}{edgelayer}
		\draw [style=new edge style 0, bend left=75, looseness=1.25] (3.center) to (0);
		\draw [style=thickedge, bend right=270, looseness=1.25] (0) to (3.center);
		\draw [style=new edge style 0, bend left=90, looseness=1.50] (8.center) to (0);
		\draw [style=thickedge, bend left=90, looseness=2.00] (0) to (8.center);
		\draw [style=new edge style 0, bend right=90, looseness=0.75] (2) to (0);
		\draw [style=new edge style 0, bend left=75] (2) to (0);
	\end{pgfonlayer}
\end{tikzpicture}
}
\end{center}
Notice that this graph has multiple edges. 
\end{example}
\begin{definition}
Let $i\in \{1,\ldots, n\}$. We define 
\[
N_i = \Bigl\{1 + \sum_j g_jf_j \mid f_j\in LM(M,M_i), g_j\in \Hom_{\C}(M_i,M)\Bigr\}\subseteq \Aut_R(M).
\]
\end{definition}
\begin{lemma}
For every $i=1,\ldots, n$ the subset $N_i$ is a normal abelian subgroup of $\Aut_R(M)$.
 \end{lemma}
 \begin{proof}
If $f,f'\in LM(M,M_i)$ and $g,g'\in \Hom_{\C}(M_i,M)$ then all the components of $gf$ are not invertible and therefore $gfg'f'=0$. This implies that if $g_j,g'_k:M_i\to M$ and $f_j,f'_k\in LM(M,M_i)$ then 
\[
\Bigl(1+\sum_j g_jf_j\Bigr)\Bigl(1+\sum_k g'_kf'_k\Bigr) = 1+ \sum_j g_jf_j + \sum_k g'_kf'_k,
\]
 which shows that $N_i$ is closed under multiplication and is abelian. The normality of $N_i$ follows from the fact that if $h\in \Aut_R(M)$ then 
$$h\Bigl(1+\sum_j g_jf_j\Bigr)h^{-1} = 1+ \sum_j hg_j f_jh^{-1},$$ and that $hg_j:M_i\to M$ and $f_jh^{-1}\in LM(M,M_i)$. 
\end{proof}
For $i,j \in \{1,\ldots,n\}$ let $d_{i,j} = \dim_{E_i}LM(M_j,M_i)$, and let $\{f^i_{j1},\ldots,f^i_{j,d_{i,j}}\}$ be an $E_i$-basis  for $LM(M_j,M_i)$, which might be empty for some indices. 
Using the block decomposition \eqref{eqn:blocks} of the group $\Aut_R(M)$, the group $N_i$ can be written as the group of all matrices of the form 
\[
Id_M+\begin{pmatrix} 
0 & 0 & \cdots & 0 \\ \vdots & \vdots & & \vdots \\
\oplus_{k=1}^{d_{i,1}} f^i_{1,k}\cdot \M_{a_i\times a_1}(E_i) & \oplus_{k=1}^{d_{i,2}} f^i_{2,k}\cdot \M_{a_i\times a_2}(E_i) & \cdots & \oplus_{k=1}^{d_{i,n}} f^i_{n,k}\cdot \M_{a_i\times a_n}(E_i) \\ 
\vdots & \vdots & & \vdots \\
 0 & 0 & \cdots & 0\end{pmatrix}.
\]
%
%

The group $N_i$ is isomorphic to the underlying abelian group of the $E_i$-vector space 
$\bigoplus_{j,k} f^i_{j,k}\cdot \M_{a_i\times a_k}(E_i)$, which is isomorphic to $\M_{a_i\times D_i}(E_i)$, where $D_i = \sum_k a_kd_{i,k} = \dim_{E_i}LM(M,M_i)$. We denote this composition by 
\begin{equation}\label{theta}
\theta_i: N_i \xrightarrow{~\cong~} \M_{a_i\times D_i}(E_i).
\end{equation}

\smallskip

Being a finite abelian group, $\M_{a \times b}(E_i)$ is isomorphic to its Pontryagin dual for all $a,b \in \N$. A convenient realisation of the isomorphism is given as follows. Let $\psi: E_i \to \F_p \hookrightarrow \mathbb{C}^\times$ be the linear character 
\[
\psi: x \mapsto \Tr^{E_i}_{\F_p}(x) \mapsto \zeta^{\Tr^{E_i}_{\F_p}(x)},
\]
where $p$ is the characteristic of $E_i$ and $\zeta$ is a primitive $p$-th root of unity. Then the map $A \mapsto \chi_A$, with
\begin{equation}\label{chi}
\chi_A(B)=\psi(\tr(AB)), 
\end{equation}
gives an isomorphism of $M_{b \times a}(E)$ with the Pontryagin dual $\widehat{\M_{a \times b}(E_i)}$. We then have $\widehat{N_i}\cong \M_{D_i\times a_i}(E_i)$. 

Let now $V$ be an irreducible representation of $\Aut_R(M)$. The restriction of $V$ to $N_i$ splits as a direct sum linear characters, each of which corresponds to a matrix in $\M_{D_i\times a_i}(E_i)$ 
\begin{definition}
We say that the representation $V$ is {\em degenerate} if one of the above matrices in $\M_{D_i\times a_i}(E_i)$ has rank $< D_i$.
\end{definition}
\begin{theorem}\label{thm:degenerate}
Assume that $V$ is a degenerate irreducible representation of $\Aut_R(M)$. Then $V$ admits a non-trivial functor morphing .
\end{theorem}
\begin{proof}
Let $A$ be one of the matrices that represent a character of $N_i$ in $V$ by \eqref{chi} under the identifications $\theta:N_i\cong \M_{a_i\times D_i}(E_i)$ and $\widehat{N_i}\cong \M_{D_i\times a_i}(E_i)$. Assume that $\rank (A)<D_i$. To avoid cumbersome notation we freely identify $LM(M,M_i) \cong E_i^{D_i}$. Let $f\in LM(M,M_i)$ and consider a morphism $g:M_i\to M$. The composition $gf$ depends only on the image of $g$ under the composition
\[
\Hom_{\C}(M_i,M) \xrightarrow{\mathrm{proj}_i} \Hom_{\C}(M_i,M_i^{a_i}) = \Hom_{\C}(M_i,M_i)^{a_i}\to (\End_{\C}(M_i)/J_i)^{a_i}= E_i^{a_i}.
\] 
This is because $f\in LM(M,M_i)$ and so the composition of $f$ with any morphism $M_i\to M_j$ where $i\neq j$ and with any morphism in $J_i$ is zero. It follows that every $g:M_i\to M$ determines a vector $[g] \in E_i^{a_i}$. 

Now, let $f\in LM(M,M_i)\cong E_i^{D_i}$. Under the identifications $N_i\cong \M_{a_i\times D_i}(E_i) \cong E_i^{a_i}\ot E_i^{D_i}$, the element ${1+gf\in N_i}$ corresponds to $[g]\otimes f$, and the associated linear character of $N_i$ is given by 
\begin{equation}\label{explicit.ch}
1+gf \mapsto \chi_A(\theta(1+gf))=\psi(fA[g]).
\end{equation}

The assumption $\rank (A)<D_i$ implies  that there is a vector $0\neq f\in E_i^{D_i}$ such that $fA=0$. The element~$f$, considered as an element in $LM(M,M_i)$, is a non-zero morphism $f:M\to M_i$.  By \eqref{explicit.ch}, for every ${g:M_i\to M}$ we have $\chi_A(\theta(1+gf))=1$. 
Consider now the functor $F$ given by the presentation 
$$\Hom_{\C}(M_i,-)\stackrel{f^*}{\to} \Hom_{\C}(M,-)\stackrel{q}{\to} F\to 0.$$
Since $f$ is non-zero, $f^*$ is not zero as well, and the functor $F$ is a proper quotient of $\Hom_{\C}(M,-)$. 
Let $a=q(\Id_M)\in F(M)$, and let $H = \stab_{\Aut_R(M)}(a)$. We claim that $V^H\neq 0$. This will imply that we have a functor morphing $(M,V) \leadsto (F',\wt{V})$ with $F'  \preccurlyeq F \prec \Hom_{\C}(M,-)$.

Plugging $M$ to the above sequence we have $F(M) = \Hom_{\C}(M,M)/f^*(\Hom_{\C}(M_i,M))$, where $\Aut_R(M)$ acts on $\Hom_{\C}(M,M)$ as post composing. 
It follows that the stabiliser of $a$ is equal to the subgroup of all maps of the form $1+gf$ where $g:M_i\to M$. 
But we have seen that $\chi_A\circ \theta$ is trivial on all these elements. Since the latter appears in $V$ with positive multiplicity, this implies that $V^H\neq 0$ as desired.
\end{proof}
The above theorem has the following corollary:
\begin{corollary}\label{cor:valency.ineq}
Assume that $(V,M)$ admits trivial functor morphing. Then, for every $i=1,\ldots, n$ 
$$\sum_{\{e:t(e)=i\}} a_{s(e)}\leq a_i.$$ 
\end{corollary}
\begin{proof}
Assume that the inequality does not holds. That is, there exits some $i$ with $\sum_{\{e:t(e)=i\}} a_{s(e)} > a_i$. The number on the left-hand-side is $D_i$. If $a_i< D_i$ then any $D_i\times a_i$ matrix has rank $<D_i$. By the previous lemma, this implies that the representation  $V$ has a nontrivial functor morphing , which is a contradiction.
\end{proof}

\subsection{The case of a vertex with no outgoing edges}\label{subsec:no.out.arrows}
 Assume now that $\Ga$ admits a vertex $i$ with no outgoing edges. We will show now that if $V$ is a representation of $\Aut_R(M)$ then either $V$ admits a non-trivial functor morphing, or we can reduce $V$ to a representation of a group of the form $\Aut_R(M'')$ for strictly smaller $M''\subsetneq M$ using Clifford Theory,

 The fact that $i$ has no outgoing edges in $LM(M_1,\ldots, M_n)$ means that there are no non-zero left maximal morphisms $M_i\to M_j$ for any $j$.  
 \begin{lemma}
 Under the above assumption there are no non-zero non-invertible morphisms $M_i\to M_j$ for any~$j$. 
 \end{lemma}
 \begin{proof}
 If $f:M_i\to M_j$ is such a morphism, then either $f$ is left maximal, or there is a non-invertible morphism $g:M_j\to M_k$ such that $0\neq gf:M_i\to M_k$. By assumption, there are no non-zero left maximal morphisms, so we can always extend $f$ to a non-zero $gf$. All the rings $\End_{\C}(M_j)$ are local. Since they are finite, the Jacobson radical of $\End_{\C}(M_j)$ is nilpotent. But this means that after composition with enough morphisms $g_k\cdots g_2g_1f$ we are bound to get zero. This contradicts the fact that there are no non-zero left maximal morphisms.
 \end{proof}
 \begin{corollary}
 If $i$ has no outgoing edges in $\Ga$ then $\Hom_{\C}(M_i,M_j)=0$ for $j\neq i$ and $\End_{\C}(M_i)=E_i$. 
 \end{corollary}
 \begin{proof}
 If $i\neq j$ then the fact that $M_i$ and $M_j$ are indecomposable and non-isomorphic implies that all morphisms $M_i\to M_j$ are non-invertible. By the above lemma we get that $\Hom_{\C}(M_i,M_j)=0$. For $\End_{\C}(M_i)$, the Jacobson radical $J_i$ contains non-invertible elements and thus must be zero. 
 \end{proof}
As before, let $M= M_1^{a_1}\oplus\cdots\oplus M_n^{a_n}$. 
We assume from now on, without loss of generality, that $i=1$ is a vertex with no outdoing edges.  
 Write $M' = M_2^{a_2}\oplus\cdots\oplus M_n^{a_n}$. 
We thus have $M= M_1^{a_1}\oplus M'$. 
We claim the following:
\begin{proposition} With the above notation the following holds:
\begin{enumerate}
\item The group $\Aut_R(M)$ splits as 
the semidirect product $(\GL_{a_1}(E_1)\times \Aut_R(M'))\ltimes \Hom_R(M',M_1^{a_1})$, where the action of $\GL_{a_1}(E_1)\times \Aut_R(M'))$ on $\Hom_R(M',M_1^{a_1})$ is given by $(g_1,g_2)\cdot A = g_1A g_2^{-1}$. 
\item We have an isomorphism  $\Hom_R(M',M_1^{a_1})\cong \mathrm{M}_{a_1\times D_1}(E_1)$, with $D_1 = \sum_{j=2}^n a_j  \dim_{E_1}\Hom_{\C}(M_j,M_1)$.
\end{enumerate}
\end{proposition}
\begin{proof}
Since there are no non-zero non-invertible morphisms $M_1\to M_j$ for any $j$, every morphism from $M_j$ to $M_i$ is left maximal when $j\neq 1$. 
Since $\Hom_R(M_1,M_j)$ for every $j\neq 1$ we can write the group $\Aut_R(M)$ in the following matrix blocks form:
\[
\left(\begin{matrix} \Aut_R(M_1^{a_1}) & \Hom(M',M_1^{a_1})  \\ 0 &\Aut_R(M') \end{matrix}\right) 
\] 
This already gives the description of $\Aut_R(M)$ as a semidirect product. It holds that 
$$\Hom_R(M',M_1^{a_1})\cong \bigoplus_{j\geq 2} \Hom_R(M_j^{a_j},M_1^{a_1})\cong \bigoplus_{j\geq 2} \Hom_R(M_j,M_i)^{a_j\cdot a_i}\cong LM(M_j,M_i)^{a_j\cdot a_i},$$ from which the isomorphism $\Hom_R(M',M_1^{a_1})\cong \M_{a_1\times D}(E_1)$ follows. 
\end{proof}
\begin{remark}
The action of $\Aut_R(M')$ on $\Hom_R(M',M_1^{a_1})\cong \M_{a_1\times D}(E_1)$ gives rise to a homomorphism ${\phi:\Aut_R(M')\to \GL_D(E_1)}$. The action in the semidirect product is given  by $(g_1,g_2)\cdot A = g_1A \phi(g_2)^{-1}$.  
\end{remark}
 

Anticipating the next step we recall Mackey-Wigner's little group method.    
 Let 
 $G = Q \ltimes N$ be a finite semidirect product with $N$ abelian. 
 Let $\omega$ be a one dimensional character of $N$. Then, there is a bijective correspondence between irreducible representations $W$ of $\stab_Q(\omega)$ and irreducible representation $V$ of $G$ such that $\Hom_N(V,\omega)\neq 0$. The bijection is given by mapping $W \in \Irr(\stab_Q(\omega))$ to $V=\Ind_{\stab_{G}(\omega)}^G(W\ot \wt{\omega})$, where $\wt{\omega}$ is an extension of $\omega$ to its stabiliser in $G$. This is a special case of Clifford Theory. 
 
 
Let now $V$ be an irreducible representation of $\Aut_R(M)$. Consider the restriction of $V$ to the normal abelian subgroup $N_1 =\Hom_R(M',M_1^{a_1})\cong\M_{a_1\times D_1}(E_1)$. As we have seen before, the character group of $N_1$ is isomorphic to $\M_{D_1\times a_1}(E_1)$. Let $A\in \M_{D_1\times a_1}(E_1)$ be a matrix such that the character $\chi_A \circ \theta$ appears in the restriction of $V$ to $N_1$; see \eqref{chi} and \eqref{theta}.  If $\rank (A)<D_1$ then we have seen in Theorem~\ref{thm:degenerate} that the representation $V$ admits a non-trivial functor morphing . We therefore focus on the case that $\rank (A) = D_1$. 

Notice that this implies that $D_1\leq a_1$ and we can therefore assume that $A$ has the form 
$A=\begin{pmatrix} 0_{D_1,a_1-D_1} & I_{D_1} \end{pmatrix}$. Indeed, $D_1$-by-$a_1$ matrices with $D_1 \leq a_1$ of full rank form a single $\GL_{a_1}(E_1)$-orbit under the right multiplication action. That, in particular, entails that we can choose an orbit representative of this form in the restriction of $V$ to $N_1$ under the $\GL_{a_1}(E_1) \times \Aut_R(M')$-action .

The stabiliser of $\chi_A$ in $\GL_{a_1}(E_1)\times \Aut_R(M')$ can then be described as the set of all matrices in $\Aut_R(M)$ of the form
$$\begin{pmatrix} B & C & & \\ 0 & \phi(H) & & \\  &  & H \end{pmatrix},$$
where $B\in \GL_{a_1-D_1}(E_1), C\in \M_{(a_1-D_1)\times D_1} (E_1)$, $H\in \Aut_R(M')$ and $\phi:\Aut_R(M')\to \GL_{D_1}(E_1)$ is the group homomorphism we described above.
We claim the following:
\begin{proposition}
The group $\stab_{\GL_{a_1}(E_1)\times \Aut_R(M')}(\chi_A)$ is isomorphic to $\Aut_R(M_1^{a_1-D_1}\oplus M')$. 
\end{proposition}
\begin{proof}
We know that 
$$\Aut_R(M_1^{a_1-D_1}\oplus M') \cong (\GL_{a_1-D_1}(E_1)\times \Aut_R(M'))\ltimes \M_{a_1-D_1 \times D_1}(E_1).$$ 
A direct verification shows that the map 
\[
(B,H,C)\mapsto \begin{pmatrix} B & C & & \\ 0 & \phi(H) & & \\  &  & H \end{pmatrix},
\]
gives an isomorphism $(\GL_{a_1-D_1}(E_1)\times \Aut_R(M'))\ltimes \M_{a_1-D_1}(E_1)\to \stab_{\GL_{a_1}(E_1)\times \Aut_R(M')}(\chi_A)$. The fact that we get the same action on $\M_{a_1-D_1\times D_1}(E_1)$ follows from the fact that the homomorphism $\phi$ does not depend on $a_1$. 
\end{proof}

The above proof shows that we get a reduction to an irreducible representation of a smaller group. Indeed, the only case in which we did not reduce the cardinality of the group is when $D=0$. But in this case the group $\Aut_R(M)$ splits as $\GL_{a_1}(E_1)\times \Aut_R(M')$. 

We conclude this section with the following proposition, which finishes the proof of Theorem \ref{thm:main4} in case the left maximal graph contains a vertex with no outgoing edges (the auxiliary ring $\wt{R}$ here is simply $R$): 
\begin{proposition}
    Assume that $LM(M_1,\ldots, M_n)$ contains a vertex with no outgoing edges. Let $V$ be an irreducible representation of $\Aut_R(M_1^{a_1}\oplus\cdots\oplus M_n^{a_n})$. Then one of the following holds:
    \begin{enumerate}
        \item The representation $V$ admits a non-trivial functor morphing, or
        \item The group $\Aut_R(M)$ splits as $\GL_{a_1}(E_1)\times \Aut_R(M')$, or 
        \item The representation $V$ is induced from a representation of a group of the form $\Aut_R(\wt{M})$, where $\wt{M}$ is a proper direct summand of $M$. 
        \end{enumerate}
\end{proposition}
This shows that by combining functor morphing and Clifford Theory we can always reduce irreducible representations into irreducible representations of smaller groups. 

\subsection{The case where every vertex has an outgoing edge }\label{subsec:there.is.outgoing}
\begin{proposition}\label{prop:circles}
Assume that all vertices in the LM graph $\Ga:=\Ga(M_1,\ldots, M_n)$ have outgoing edges. Assume that $V$ is an irreducible representation of $\Aut_R(M)$, with $M=M_1^{a_1}\oplus \cdots\oplus M_n^{a_n}$, that admits a trivial functor morphing . Then 
\begin{enumerate}
\item The connected components of $\Ga$ are circles. 
\item If $i$ and $j$ are contained in the same connected component of $\Ga$, then $a_i=a_j$. 
\item Every left maximal morphism is also right maximal.
\end{enumerate}
\end{proposition}
\begin{proof}
Take any $i_0\in \{1,\ldots, n\}$. Since every vertex has an outgoing edge, we can find an edge from $i_0$ to some~$i_1$. Since $i_1$ has an outgoing edge we can find an edge from $i_1$ to some $i_2$, and so on. Since the graph is finite, we will eventually go back to a vertex that already appeared. We thus have a path that looks like 
$$i_0\to i_1 \to \cdots \to i_b\to i_k$$ where $0\leq k\leq b$. 
By Corollary~\ref{cor:valency.ineq}, the fact that we have a trivial functor morphing of $V$ implies that $a_{i_0}\leq a_{i_1}\leq\cdots \leq a_{i_b}$. If $k\neq 0$ then the fact that both $i_b$ and $i_{k-1}$ have an edge that goes into $i_k$ imply that $a_{i_b} + a_{i_{k-1}} \leq a_{i_k}$. But by the above inequality we also have $a_{i_k}\leq a_{i_b}$, which is impossible since $a_j>0$ for $j=1,\ldots, n$. So we also get $a_{i_b}\leq a_{i_0}$ and thus $a_{i_0}= a_{i_1}=\cdots = a_{i_b}$. 

This implies that every vertex is contained in a circle, and that if $i$ and $j$ are in the same circle then $a_i=a_j$. It also shows that for every $i$ there is exactly one edge $e$ such that $t(e)=i$, since if there was another such edge $e'$ then we would have had $a_{s(e)} + a_{s(e')}\leq a_i$. But if $j$ is the predecessor of $i$ in the circle then $a_j=a_i$, and since all the $a_k>0$ this is impossible. It follows that the connected components of $\Ga$ are circles.

To prove the last statement, assume that $f:M_i\to M_j$ is a left maximal morphism (it may happen that $i=j$).
If $f$ is not right maximal, then there is a non-zero non-invertible morphism $g:M_k\to M_i$ such that $fg\neq 0$. But then $fg$ is also a left maximal morphism. By the above proof, this is impossible if $k\neq i$. Assume then that $k=i$. Then $g:M_i\to M_i$ is non-invertible, and therefore nilpotent. The fact that $\dim(LM(M_i,M_j))=1$ implies that $fg$ is linearly dependent on $f$. So there is a non-zero scalar $a\in E_j$ such that $fg = af$. Since $g$ is nilpotent, there is some $k>0$ such that $g^k=0$. But then 
$$0 = fg^k = a^nf\neq 0,$$ which is a contradiction. 
\end{proof}

\begin{corollary} In the above case, the right maximal graph of $(M_1,\ldots, M_n)$ is equal to the left maximal graph of $(M_1,\ldots, M_n)$. Moreover, if $i_0\to i_1\to\cdots\to i_{k-1}\to i_0$ is a cycle, then we have natural isomorphisms $E_{i_0}\cong E_{i_1}\cong\cdots\cong E_{i_{k-1}}$. 
\end{corollary}
\begin{proof}
In Proposition~\ref{prop:circles} we showed that every maximal left morphism is also maximal right. The same proof shows that a maximal right morphism is also maximal left, and so the two graphs coincide. For the statement about the fields, choose a basis vector $f\in LM(M_{i_j},M_{i_{j+1}})$. Since $f$ is also right maximal, it holds that the action of $\End_{\C}(M_{i_{j+1}})$ on $LM(M_{i_j},M_{i_{j+1}})$ splits via $E_{i_{j+1}}$. Let $a\in E_{j+1}$. It then holds that there is a unique $\phi(a)\in E_{i_j}$ such that $af = f\phi(a)$. A direct verification shows that this gives a homomorphism of rings $E_{i_{j+1}}\to E_{i_j}$. Since all these rings are fields, all these homomorphisms must be injective. We thus get a chain of inclusions $E_{i_{k-1}}\to E_{i_{k-2}}\to\cdots\to E_{i_0}\to E_{i_{k-1}}$. Since all of these fields are finite, all of these injective homomorphisms must also be surjective. 
\end{proof}
\begin{remark} It is possible that the resulting field isomorphism $E_{i_{k-1}}\to\cdots\to E_{i_0}\to E_{i_{k-1}}$ will not be the identity.
\end{remark}
\subsection{Clifford reduction for the case that $\Ga$ contains a cycle of length $k>1$}\label{subsec:big.cycle}
We  now show  how to handle the case where $\Ga$ contains a cycle of length $k>1$. 
For convenience, assume that the cycle is $1\to 2\to\cdots \to k \to 1$. 
Write $a=a_1=a_2=\cdots= a_k$ and  $E=E_1\cong\cdots\cong E_k$. As mentioned at the end of the previous subsection, we might get a non-trivial automorphism of $E$ by going a full circle on the graph, but we will not use it here.

Let~$V$ be an irreducible representation of $\Aut_R(M)$ that admits a trivial functor morphing . 
Consider the abelian normal subgroups $N_1, \ldots, N_{k-1}$. 
As  shown earlier, all of them are isomorphic to $\M_{a\times a}(E)$. By considering the block decomposition structure of $\Aut_R(M)$ we see that their product $N$ in $\Aut_R(M)$ is isomorphic to $N_1\times\cdots\times N_{k-1}\cong \M_{a\times a}(E)^{k-1}$.

We next describe the action of $g\in \Aut_R(M)$ on $N$. Write $g=(g_{i,j})$, where $g_{i,j}\in \Hom_{\C}(M_j^{a}, M_i^{a})$ and $g_{i,i}\in \Aut_{\C}(M_i^{a})$. For $i=1,\ldots, k$ we denote the image of $g_{i,i}$ in $\GL_{a}(E)$ by $\ol{g_{i,i}}$. The non-trivial entries of elements in the group $N$ are then contained in the blocks $(2,1), (3,2),\ldots, (k,k-1)$. The action of $g$ on a tuple $(A_1,A_2,\ldots, A_k)$ is given by 
$$g\cdot (A_1,\ldots, A_k) = (\ol{g_{2,2}}A_1\ol{g_{1,1}}^{-1},\ldots, \ol{g_{k,k}}A_k\ol{g_{k-1,k-1}}^{-1}),$$
where $\ol{g_{i,i}}\in \GL_a(E_i)\cong \GL_a(E)$. 
The character group of $N$ is then also isomorphic to $N$, and a one-dimensional character of $N$ is then given by a tuple $(A_1,\ldots, A_{k-1})$. 
The fact that $V$ has a trivial functor morphing  implies that all matrices $A_1,\ldots, A_{k-1}$ have rank $a$ (see Theorem~\ref{thm:degenerate}). 
This means that all the matrices are invertible. We can then use the action of $\Aut_R(M)$ on the character group of $N$ to show that such a character is conjugate to the character given by $(I_a,I_a,\ldots,I_a)$
By Clifford Theory, we know that $V$ is then induced from a representation of the stabiliser of this character, which can easily be seen to be equal to 
\[
\Bigl\{g\in \Aut_R(M) \mid \forall 1\leq i,j\leq k,   \ol{g_{ii}} = \ol{g_{jj}}\Bigr\}.
\]

We will next construct a ring $\wt{R}$ and an $\wt{R}$-module $\wt{M}$ such that the above group is $\Aut_{\wt{R}}(\wt{M})$.
To do so, choose $\wt{R} = R\oplus R\ot_{\Z} R$. We shall write the elements in $R\ot_{\Z} R$ as $RxR$. The product in $\wt{R}$ is given by 
$$(a+b_1xb_2)(c+d_1xd_2) = ac + ad_1xd_2 + b_1xb_2c.$$ 
In other words, the two sided ideal generated by $x$ is nilpotent of nilpotency degree 2. 
We then define $\wt{M}=M$ as an $R$-module. 
To define the action of $x\in \wt{R}$ on $\wt{M}$, we choose basis elements $f_i\in LM(M_i,M_{i+1})$ for $i=1,\ldots, k-1$. Write $M = M_1^a\oplus\cdots\oplus M_k^a \oplus\bigoplus_{j>k} M_j^{a_j}$.  Define $$x\cdot (m_1,\ldots, m_a)\in M_i^a=(f_i(m_1),\ldots, f_i(m_a))\in M_{i+1}^a$$ for $i=1,\ldots, k-1$, and zero on all the rest of the direct summands of $M$. A direct verification shows that the above stabilizer is indeed $\Aut_{\wt{R}}(\wt{M})$. 
We summarize this in the following proposition, which finishes the proof of Theorem \ref{thm:main4} also in case every vertex in the left maximal graph has an outdoing edge. 
\begin{proposition}
Assume that $LM(M_1,\ldots, M_n)$ is a disjoint union of circles. Assume that $V$ is an irreducible representation of $\Aut_R(M)$ that admits a trivial functor morphing. If one of the circles in $LM(M_1,\ldots, M_n)$ has length $k>1$ then $V$ is induced from a proper subgroup of $\Aut_R(M)$ that is isomorphic to $\Aut_{\wt{R}}(\wt{M})$.     
\end{proposition}
This shows that the only case in which we cannot use functor morphing and Clifford Theory to reduce the complexity of the representation is when $LM(M_1,\ldots, M_n)$ is a disjoint union of cycles of length 1. 

\begin{remark}
This section can be thought of geometrically in the following way: for every cycle we have chosen a spanning tree (in this case, the chain $1\to 2\to\cdots\to k$) and we collapsed it to a point. In this way, we only need to deal with the case of cycles that contain a single point. 
\end{remark}







\section{Examples}
We begin with the trivial example. Consider the trivial one dimensional representation $\one$ of $\Aut_R(M)$. This representation appears in $K0(M)$, where $0$ stands here for the zero functor. In this case $K0(M) = K$. We thus have $(M,\one)\leadsto (0,\one)$. 
\subsection{The representation theory of $\GL_n(\F_q)$}\label{ex:glnfq}
Write $\k:=\o_1 \cong \F_q$. The group $\GL_n(\k)$ can be thought of as $\Aut_{\k}(\k^n)$. We thus consider the case $\l=1$, $\C=\k\lmod$. 
Harish-Chandra (or parabolic) induction and restriction play a fundamental role in the representation theory of these groups. By definition, if $V_{i}$ is a representation of $\GL_{n_i}(\k)$ for $i=1,2$, the Harish-Chandra multiplication of $V_1$ and $V_2$ is a composition of inflation with indction, given explicitly by
\[
V_1\times V_2 = \Ind_{P_{n_1,n_2}(\k)}^{\GL_{n_1+n_2}(\k)}\Inf^{P_{n_1,n_2}(\k)}_{\GL_{n_1}(\k)\times \GL_{n_2}(\k)}\left(V_1\boxtimes V_2\right),
\]
where $P_{n_1,n_2}(\k)$ is a parabolic subgroup containing $\GL_{n_1}(\k) \times \GL_{n_2}(\k)$ embedded as block diagonal matrices, with block sizes $(n_1,n_2)$.
Harish-Chandra restriction is defined dually.

We have that $\Cp = \Fun_{\k}(\C,\k\lmod)$ is again equivalent to the category $\C$ of finite dimensional $\k$-vector spaces. In particular, every object in $\Cp$ is of the form $\Hom_{\k}(\k^m,-)$ for some $m\geq 0$. We use Yoneda's lemma to identify $\Aut_{\Cp}(\Hom_{\k}(\k^m,-))$ and $\Aut_{\k}(\k^m)= \GL_m(\k)$.  

\begin{lemma} Let $V\in \Irr(\GL_n(\k))$ and let $(\Hom_{\k}(\k^m,-),\wt{V})$ be the functor morphing  of $(\k^n,V)$. Then $V$ is a direct summand of $\wt{V}\times \one_{\GL_{n-m}(\k)}$, and $m$ is the minimal number $j$ such that $V$ is a direct summand of a representation of the form $V'\times \one_{\GL_{n-j}(\k)}$ for some representation $V'$ of $\GL_j(\k)$.  
\end{lemma}

\begin{proof}
Let $V\in \Irr(\GL_n(\k))$.
Let $m$ be the minimal number for which $\Hom_{\GL_n(\k)}(K\Hom_{\k}(\k^m,\k^n),V)\neq 0$. Note that $m \le n$ because $K\Hom_{\k}(\k^n,\k^n)$ contains the regular representation of $\GL_n(\k)$. By choosing a basis we can identify $\Hom(\k^m,\k^n)$ with $M_{n \times m}(\k)$ on which $\GL_n(\k)$ acts by left multiplication. The stabiliser in $\GL_n(\k)$ of the standard embedding $\k^m \hookrightarrow \k^n$ is $H_m = \left\{\left(\begin{smallmatrix} I_m & * \\ 0 & *\end{smallmatrix}\right)\right\} \leq P_{m,n-m}(\k)$. It follows that $V^{H_m}\neq 0$. Then $$V^{H_m}\neq 0  \Longleftrightarrow \Hom_{H_m}(\one_{H_m}, V)\neq 0\Longleftrightarrow \Hom_{\GL_n(\k)} (\Ind_{H_m}^{\GL_n(\k)}\one_{H_m}, V)\neq 0.$$
We have 
\[
\begin{split}
\Ind_{H_m}^{\GL_n(\k)}\one_{H_m}&\cong \Ind_{P_{m,n-m}}^{\GL_n(\k)}\Ind^{P_{m,n-m}(\k)}_{H_m}\one_{H_m} \\&= \Ind_{P_{m,n-m}(\k)}^{\GL_n(\k)}\Inf_{\GL_m(\k)\times \GL_{n-m}(\k)}^{P_{m,n-m}(\k)}(K\GL_m(\k)\boxtimes \one_{\GL_{n-m}(\k)})\\
&= K\GL_m(\k)\times \one_{\GL_{n-m}(\k)}.
\end{split}
\]
Since all the irreducible representations of $\GL_m(\k)$ appear in the regular representation, this implies that $V^{H_m}\neq 0$ if and only if $V$ appears as a direct summand in a representation of the form $V'\times \one_{\GL_{n-m}(\k)}$ for some $V'\in \Irr(\GL_m(\k))$, as desired.
The minimality of $m$ implies that $\Hom_{\GL_n(\k)}(V,K\Hom_{\k}(\k^{j},\k^n))=0$ for $j<m$; see Lemma~\ref{lem:invs}. 
The fact that functor morphing assigns an irreducible representation to an irreducible representation implies that if $(\k^n,V)\leadsto (\k^m,\wt{V})$ then $\wt{V}$ is the unique irreducible representation of $\GL_m(\k)$ such that $V$ appears as a direct summand of $\wt{V}\times \one_{\GL_{n-m}(\k)}$. \end{proof}

Next, we describe the functor morphing  explicitly. To do so, we need to recall some facts about Positive Self adjoint Hopf algebras (PSH-algebras), that were introduced by Zelevinsky in \cite{Zelevinsky}. By definition, a PSH-algebra is an $\N$-graded $\Z$- Hopf algebra that is equipped with an inner product and a distinguished orthonormal basis, such that the multiplication is dual to the comultiplication with respect to that inner product, and all the structure constants with respect to the basis are non-negative. Zelevinsky showed that up to rescaling of the grading there exists exactly one isomorphism type of a \textit{universal} PSH-algebra, and that every PSH-algebra decomposes uniquely as the tensor product of universal PSH-algebras. Moreover, the tensor factors are in one to one correspondence with \textit{cuspidal} elements. By definition, these are basis elements that are primitive with respect to the comultiplication. More precisely, Zelevinsky showed that if $\H$ is a PSH-algebra, then we can write $$\H = \bigotimes_{\rho} \H(\rho),$$ where $$\H(\rho):=\{x\in\H \mid \langle x,\rho^n\rangle\neq 0\text{ for some } n\geq 0\}$$ is a universal PSH-algebra. The index $\rho$ runs here over the cuspidal elements.

For a group $G$, let $\text{K}_0(G)$ denote the Grothendieck group of $\Rep(G)$.
The canonical universal PSH-algebra is $\mathrm{Zel}:=\bigoplus\text{K}_0(S_n)$, where the multiplication and comultiplication are given by induction and restriction. The basis of $\text{K}_0(S_n)$ is then given by the irreducible representations of~$S_n$, which are in one to one correspondence with partitions of $n$. 
Write now $$\H_q = \bigoplus_n \text{K}_0(\GL_n(\k)).$$ In \cite[Chapter III]{Zelevinsky} Zelevinsky showed that Harish-Chandra induction and its adjoint restriction operation induce on~$\H_q$ a structure of a PSH-algebra. For each cuspidal representation $\rho$ of $\GL_{d}(\k)$ we choose an isomorphism $\H_q(\rho) \cong \mathrm{Zel}$, as in \cite[Section 9.4]{Zelevinsky}. We then obtain, for each $n$, an injective map $\Irr(S_n) \to \Irr(\GL_{dn}(\k))$, sending the $S_n$-representation associated to a partition $\lambda \parti n$ to a specific subrepresentation $\mathbb{S}(\rho, \lambda)$ of the induced representation $\rho^n$. The decomposition of $\H_q$ as the tensor product of the $\H_q(\rho)$s then implies that every irreducible representation of $\GL_n(\k)$ can be written uniquely as the Harish-Chandra product $$V = \bigtimes_{\rho} \Ss(\rho,\lambda(\rho)),$$ where $\rho$ runs over the cuspidal representations of the groups $\GL_m(\k)$, $\lambda(\rho)\parti n(\rho)$ , and $n(\rho)\in \N$ are numbers that satisfy $\sum_{\rho}d(\rho)n(\rho)=n$, where $\rho \in \Irr(\GL_{d(\rho)}(\k))$. 
%
%
%

We know that $V$ functor-morphs to $\wt{V}$, which is a representation of $\GL_m(\k)$. We also know that $V$ appears as a direct summand in $\wt{V}\times \one_{\GL_{n-m}(\k)}$. We can also decompose $\wt{V}$ with respect to the Harish-Chandra multiplication. We get 
$$\wt{V} = \bigtimes_{\rho} \Ss(\rho,\mu(\rho))$$ for some partitions $\mu(\rho)\parti m(\rho)$. 
Using the unique factorisation into Harish-Chandra product of representations that are associated to cuspidal elements, we see that necessarily $\mu(\rho) = \la(\rho)$ and $n(\rho) = m(\rho)$ for every $\rho\neq \one$, and that $\Ss(\one,\la(\one))$ is a direct summand of $\Ss(\one,\mu(\one))\times \one_{\GL_{n-m}(\k)}$, where $\one=\one_{\GL_1(\k)}$. This holds in the algebra $\H_q(\one)$, which is a universal PSH-algebra. 

By \cite[Theorem 3.1., page 27]{Zelevinsky} we have $\H_q(\one)\cong \Z[x_1,x_2,\ldots, ]$, a polynomial algebra in infinitely many variables. The isomorphism sends $x_n$ to the trivial irreducible representation of $\GL_n(\k)$. For $s\geq 1$ we define $(x_s)^*:\H_q(\one)\to \H_q(\one)$ to be the  adjoint operator to multiplication by $x_s$. Following \cite[Section 3.6]{Zelevinsky} we write $X^* = \sum_{s\geq 1} (x_s)^*$. In \cite[Section 4.3.]{Zelevinsky} it is proved that \begin{equation}\label{eq:xstar}X^*\left(\Ss(\one,\la)\right) = \sum_{\mu\perp \la} \Ss(\one,\mu),\end{equation} where $\mu\perp \la$ means that the partition $\mu$ is obtained from $\lambda$ by removing at most one element from every column of the associated Ferrers diagram. 
The fact that $\Ss(\one,\mu(\one)) \times \one_{\GL_{n-m}(\k)} = \Ss(\one,mu(\one))\cdot x_{n-m}$ contains $\Ss(\one, \lambda(\one))$ as a direct summand implies that $(x_{n-m})^* \Ss(\one,\lambda(\one)) \neq 0$. Moreover, the minimality of $m$ ensures that $j = n-m$ is the maximal number such that $(x_j)^*(\Ss(\one,\lambda(\one)) \neq 0$.
The maximality of $j$ combined with \eqref{eq:xstar} imply that $\mu(\one)$ is the partition obtained from $\la(\one)$ by removing the last element from each column in $\la(\one)$. In other words, if $\la(\one) = (\la_1,\ldots, \la_r)$ then $\mu(\one) = (\la_2,\ldots, \la_r)$. 
We sum this up in the following proposition:
\begin{proposition}
    Assume that $(\k^n,V)\leadsto (\k^m,\wt{V})$. If $V= \bigtimes_{\rho} \Ss(\rho,\la(\rho))$ then $\wt{V} = \bigtimes_{\rho}\Ss(\rho,\mu(\rho))$ where $\mu(\rho) = \la(\rho)$ for every $\rho\neq \one$. For $\rho=\one$ we have that if $\la(\one) = (\la_1,\la_2,\ldots, \la_r)$ then $\mu(\one)=(\la_2,\ldots, \la_r)$ and $\la_1 = n-m$. In particular, the only case of trivial functor morphing is when $\la(\one)$ is the empty partition of 0.
\end{proposition}


\subsection{Parabolic groups over finite fields}
We consider now the representation theory of the parabolic groups $P_{a,b}(\k) = \left\{ \left(\begin{smallmatrix} * & * \\ 0 & * \end{smallmatrix}\right)\right\}\subseteq \GL_{a+b}(\k)$, where the block sizes are $a$ and $b$ respectively. In this case $\l=1$ and we continue to write $\k= \Ow_1$. We first explain how these groups fall into our framework. For this, we consider the ring $$R = \k\langle v_1,v_2,e\rangle/ \langle v_1^2= v_1, v_2^2=v_2, v_1v_2=v_2v_1=0, v_2ev_1=e, e^2= v_1e = v_2e=0\rangle.$$ This is just the algebra of upper-triangular matrices which is the quiver algebra of the directed graph
\medskip
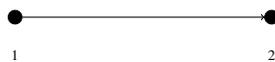
\begin{center}\scalebox{0.5}{
    \begin{tikzpicture}
	\begin{pgfonlayer}{nodelayer}
		\node [style=graph node] (0) at (-7.5, 2.25) {};
		\node [style=graph node] (1) at (-0.75, 2.25) {};
		\node [style=none] (2) at (-7.5, 1.25) {1};
		\node [style=none] (3) at (-0.75, 1.25) {2};
	\end{pgfonlayer}
	\begin{pgfonlayer}{edgelayer}
		\draw [style=new edge style 0] (0) to (1);
	\end{pgfonlayer}
\end{tikzpicture}
}
\end{center}

We consider the modules $M_1 = Rv_1 = \Span_\k\{e,v_1\}$ and $M_2 = Rv_2=\Span_\k\{v_2\}$. We have $\End_R(M_1) \cong \End_R(M_2) = \k$, $\Hom_R(M_1,M_2)=0$ and $\Hom_R(M_2,M_1)\cong\k$, where the last hom-space is spanned by the morphism $f$ that sends $v_2$ to $e$. The morphism $f$ is a left maximal and a right maximal morphism. 
A direct verification now shows that $\Aut_R(M_1^a\oplus M_2^b)\cong P_{a,b}(\k)$. 

Let $\C = \langle M_1,M_2\rangle$. The category $\Cp$ has three indecomposable functors. They are give explicitly by 
\[
\begin{split}
F_1(M_1) &= \k, \enspace F_1(M_2) = 0 \\
F_2(M_1) &= 0, \enspace F_2(M_2) = \k\\
F_3(M_1) &= F_3(M_2) = \k, \enspace F_3(f)=1.
\end{split}
\]
In terms of hom-functors  $F_1 = \Hom_R(M_1,-)$, $F_3 = \Hom_R(M_2,-)$, and $F_2 = \Hom_R(M_2,-)/f^*(\Hom_R(M_1,-)$.
Recall next that $P_{a,b}(\k)$ fits into the following short exact sequence:
$$1\to U_{a,b}(\k)\to P_{a,b}(\k)\to \GL_a(\k)\times \GL_b(\k)\to 1,$$ where $U_{a,b}(\k) := \left\{\left(\begin{smallmatrix} I_a & * \\ 0 & I_b \end{smallmatrix}\right)\right\}.$
The dual group of $U_{a,b}(\k)$ can naturally be identified with $U_{b,a}(\k)$ using the trace form. 
The action of $\GL_a(\k)\times \GL_b(\k)$ on the character groups has $\min\{a,b\}+1$ orbits, parametrised by the rank of the matrix in $U_{b,a}(\k)$. We have seen in Theorem \ref{thm:degenerate} that if the matrix in $U_{b,a}(\k)$ has rank $<b$, then necessarily there is a non-trivial functor morphing .
This also gives us a concrete example to the situation of subsection \ref{subsec:no.out.arrows}, since $LM(M_1,M_2)$ is in fact isomorphic to the quiver with one arrow and two elements.  
\smallskip

Consider now the special case $a=b=1$. We analyse all the irreducible representations and their functor morphing in this case. Let  $V$ be an irreducible representation of $P_{1,1}(\k)$. We think of this group as $\Aut_R(M_1\oplus M_2)$. 
The action of $\GL_1(\k)\times \GL_1(\k)$ on the character group of $U_{1,1}(\k) \cong \k$ has only two orbits corresponding to $\{0\}$ and $\k^\times$. 
If~$U_{1,1}(\k)$ acts trivially on $V$, then the action of $P_{1,1}(\k)$ on $V$ factors through the quotient $\GL_1(\k)\times \GL_1(\k)$. Since this group is abelian, $V=Kv$ is one-dimensional. 
Fix an injective character {$\xi:\k^{\times}\to K^{\times}$}. The fact that $\k^{\times}$ is cyclic of order $q-1$ implies that every  character on $\k^{\times}$ is of the form $g\mapsto \xi(g^x)$ for some $x\in \Z/(q-1)$. The action on $V$ is thus given by $(g_1,g_2)\cdot v=  \xi(g_1^xg_2^y)v$ for some $x,y\in \Z/(q-1)$. 
\begin{enumerate}
\item $x=y=0$. This implies that $V$ is trivial, and it functor morphs to $(0,\one)$, the trivial representation of the (trivial) automorphism group of the trivial functor.
\item $x=0, y\neq 0$. In this case we have functor morphing   $(M_1\oplus M_2,V)\leadsto (F_2,\wt{V})$. This follows from the fact that the stabiliser of the generic point in $F_2(M_1\oplus M_2)$ is $H:=\{\left(\begin{smallmatrix} * & * \\ 0 & 1 \end{smallmatrix}\right)\}$, and $V^H= V\neq 0$. Since this representation is not trivial and $F_2$ is a simple functor we get that the corresponding functor is $F_2$. Summing up: we get $\Aut_{\Cp}(F_2)\cong \GL_1(\k)$, $\wt{V}=V^H=V$, and the action is given by $g^y\cdot v= \xi(g^y)v$. 
\item $x\neq 0, y=0$. This is similar to the previous case, and we get functor morphing  with the functor $F_1$.
\item $x\neq 0\neq y$. In this case the functors $F_1$ and $F_2$ are not enough, but rather $F_1\oplus F_2$, since the stabiliser of a generic point in $(F_1\oplus F_2)(M_1\oplus M_2)$ is $H' = U_{1,1}(\k)$, and $V^{U_{1,1}(\k)} = V$. The automorphism group is  $\Aut_{\Cp}(F_1\oplus F_2)\cong \GL_1(\k)\times \GL_1(\k)$, and the action is the one coming from $P_{1,1}(\k)$.  
\end{enumerate}
Notice that in all four cases here we did not reduce the dimension of $V$, but we did reduce the cardinality of the automorphism group.

We next consider the case where the group $U_{1,1}(\k)$ acts non-trivially. As all nontrivial characters are in the same orbit, we may assume without loss of generality it acts by the character that maps $\left(\begin{smallmatrix} 1 & u \\ 0 & 1 \end{smallmatrix}\right)$ to $\psi(u)$, where $\psi:\k \to K^\times$ is a fixed non-trivial additive character. The stabiliser in $P_{1,1}(\k)$ of a non-trivial character is $U_{1,1}(\k)$ is $S:= \{\left(\begin{smallmatrix} g & u \\ 0 & g \end{smallmatrix}\right) \mid g \in \k^\times, u \in \k\}$, and by Clifford Theory we know that the representation $V$ is induced from a representation of $S$. Since $S$ is abelian, every such representation is one-dimensional, and maps $\left(\begin{smallmatrix} g & u \\ 0 & g \end{smallmatrix}\right)$ to $\xi(g^x)\psi(u)$ for some $x\in \Z/(q-1)\Z$.  

We claim that in this case the representation $V$ admits a functor morphing to the functor $F_3$. Indeed, this follows from the fact that the stabiliser of the generic point in $F_3(M_1\oplus M_2)$ is $H''=\{\left(\begin{smallmatrix} * & 0 \\ 0 & 1\end{smallmatrix}\right)\}$, and $H''\cap S = \{1\}$. Since the representation $V$ is induced from $S$, a simple application of Mackey's formula implies that $V^{H''}\neq 0$. Since $V$ does not appear in any proper subquotient of $F_3$, we deduce that $F_3$ is the associated functor of $V$. We have $\Aut_{\Cp}(F_3)\cong \GL_1(\k)$, and a direct calculation shows that the resulting representation of $\Aut_{\Cp}(F_3)$ is the one that sends $g$ to $\xi(g^x)$. We thus see that all irreducible representations of $P_{1,1}(\k)$ admit non-trivial functor morphing, since $F_1\oplus F_3 = \Hom_R(M_1\oplus M_2,-)$ is not the associated functor of any irreducible representation. 

\subsection{The groups $\GL_n(\o_\l)$}
We next consider the groups $\GL_n(\o_\l)$, where $\l>1$. They can be realised as $\Aut_{\o_\l}(\o_\l^n)$. We consider here the category $\C = \langle \o_\l\rangle$. The category $\Cp$ is then equivalent to the category $\o_\l\lmod$ of all $\o_\l$-modules, where the equivalence is given by sending a functor $F$ to $F(\o_\l)$. The indecomposable functors in $\Cp$ are then the ones that correspond to the cyclic modules $\o_1,\o_2,\ldots,\o_\l$. We denote them by $F_1,\ldots, F_\l$, respectively. The minimal resolution of $F_i$, $i=1,\ldots, \l-1$, is given by $$\Hom(\o_\l,-)\xrightarrow{(\pi^i)^*} \Hom(\o_\l,-)\to F_i\to 0.$$ The functor $F_\l$ is already projective. 

Let now $V$ be an irreducible representation of $\GL_n(\o_\l)$. Let $N := 1+ \pi^{\l-1}\M_n(\o_\l)\cong \left(\M_n(\o_1),+\right)$ be the smallest principal congruence subgroup. The dual group of $N$ is again isomorphic to $N$ using the trace form on $\M_n(\o_1)$. The action of $\GL_n(\o_\l)$ on this group factors through $\GL_n(\o_\l)\to \GL_n(\o_1)$, and is given by conjugation. In this way one associates a similarity class of matrices in $\M_n(\o_1)$ to any irreducible representation. This plays a fundamental role in the analysis of the irreducible representations; see \cite{Hill_Jord} for a \lq Jordan decomposition\rq~ of irreducible representations and \cite{CMO3} for a treatment inspired by Zelevinsky's work \cite{Zelevinsky}. Both~\cite{Hill_Jord} and~\cite{CMO3} point at the representations associated with nilpotent matrices as the most difficult to understand.
By Theorem \ref{thm:degenerate} we know that if the associated matrix of an irreducible representation $V$ is not invertible, let alone nilpotent, then $V$ admits a non-trivial functor morphing . We will give a concrete example for this now. 

Consider the case $\l=n=2$. We write $M = \o_2^2$ and  consider irreducible representations of $\GL_2(\o_2)$ with associated nilpotent matrix $e_{12} = \left(\begin{smallmatrix} 0 & 1 \\ 0 & 0 \end{smallmatrix}\right)$. The centraliser of $e_{12}$ in $\GL_2(\o_1)$ is by $S=\{\left(\begin{smallmatrix} a & b \\ 0 & a \end{smallmatrix}\right) \mid a\in \o_1^{\times}, b\in \o_1\}$. Thus, $V$ is induced from a representation of the subgroup $T = \{\left(\begin{smallmatrix} a & b \\ \pi c & d\end{smallmatrix}\right) \mid b,c \in \o_2,  a,d\in \o_2^\times, a \equiv d\text{ mod } \pi\}.$
We have a short exact sequence $$1\to N\to T\to S\to 1.$$ 

Let $\phi: N\to K^{\times}$ the character that sends $1+\pi X$ to $\psi(\Tr(e_{12}X))=\psi(X_{21})$, where $\psi:\k \to K^\times$ is an additive character and let $E$ be the associated idempotent in $KN$. By Clifford Theory, the representations of $\GL_2(\Ow_2)$ with associated matrix $e_{12}$ are in one to one correspondence with representations of $T$ whose restriction to $N$ is $\phi$, and these in turn are in one to one correspondence with representations of $K^{\beta}S$ for some two-cocycle $\beta$. A direct verification shows that in this case the cocycle $\beta$ is trivial (see the discussion at the end of Subsection 7.1. in \cite{CMO3}). The representations that are of interest for us are thus in one to one correspondence with representations of $S\cong \Ow_1\times \Ow_1^{\times}$.  

We have $\Hom_{\Ow_2}(\Ow_2^2,-) = F_2^2$. The proper subquotients of this functor are $0,F_1,F_1\oplus F_1, F_2,F_1\oplus F_2$. The group $N$ acts trivially on $K0(M)=K, KF_1(M),$ and $ KF_1^2(M)$. So if the associated matrix of $V$ is $e_{12}$, the functor morphing  of $V$ can be associated with the functor $F_1\oplus F_2$ or with $F_2$.

Consider first the smaller one of these functors, $F_2$. Recall that $$KF_2(M) = K\o_2^2 = \Span_K\{u_{\left(\begin{smallmatrix} a \\ b \end{smallmatrix}\right)} \mid a,b \in \o_2\}.$$
Let $V$ be an irreducible representation of $\GL_2(\Ow_2)$ with associated matrix $e_{12}$, and let $W$ be the corresponding irreducible representation of $S \cong \Ow_1\oplus \Ow_1^{\times}$. 
\begin{lemma}
    The representation $V$ appears in $KF_2(M)$ if and only if the action of $\Ow_1$ on $W$ is trivial.
    \end{lemma}
\begin{proof}
We first apply the idempotent $E$ to $KF_2(M)$, We get 
$$E\cdot u_{\left(\begin{smallmatrix} a \\ b \end{smallmatrix}\right)} = \frac{1}{q^4} \sum_{r_{ij}} \psi(-r_{21})u_{\left(\begin{smallmatrix} a+ \pi r_{11}a + \pi r_{12} b \\ b + \pi r_{21}a + \pi r_{22}b\end{smallmatrix}\right)}.$$
If $b\neq 0$ in $\Ow_1$ we can do a change of variables $r_{22} = r_{22} - b^{-1}ar_{21}$, and we get that the above sum equals 
$$\frac{1}{q^4} \sum_{r_{ij}} \psi(-r_{21})u_{\left(\begin{smallmatrix} a+ \pi r_{11}a + \pi r_{12} b \\ b + \pi r_{22}b\end{smallmatrix}\right)}=0$$
by the orthogonality of characters. We are thus left with the subspace spanned by all vectors of the form 
$$\psi(-r_2)u_{\left(\begin{smallmatrix} a + \pi r_1a\\ \pi r_2a \end{smallmatrix}\right)}, \quad r_1,r_2 \in \o_1.$$ But a direct calculation now shows that the subgroup $U_{1,1}(\o_1)$ acts trivially on all such vectors. 

This already proves one direction of the lemma. The second direction follows by considering the action of the group $Q:=\{\left(\begin{smallmatrix} a & 0 \\ 0 & a \end{smallmatrix}\right)\}$, and showing that any irreducible representation $W$ such that the action of $\o_1$ is trivial appears inside $KF_2(M)$. 
\end{proof}
If the group $Q$ acts non-trivially on $W$, we get a representation with associated functor $F_1\oplus F_2$. The automorphism group of this functor is isomorphic to  $\Aut_{\o_2}(\o_1\oplus \o_2)$, which is a group that does not appear as $\GL_n(\o_\l)$ for any $n,\l$. 

\subsection{The Grassmann representations of $\GL_n(\o_\l)$}\label{subsec:Grassmann} In \cite{Bader-Onn} the Grassmann representations of $\GL_n(\o_\l)$ were analysed. These are defined as follows. Recall that isomorphism classes of submodules of $\o_\l^n$ are parameterised by partitions $(\la_1,\ldots,\la_n) \leq \l^n$. Let $\mathrm{Gr}(\lambda)=\left\{M \leq \o_\l^n \mid M \cong  \o_\lambda :=\oplus_{i=1}^n \o/\p^{\la_i}\right\}$ and let $K\mathrm{Gr}(\lambda)$ be the corresponding permutation representation of $\GL_n(\o_\l)=\Aut_{\o_\l}(\o_\l^n)$; in \cite{Bader-Onn} this representation is denoted $\cF_\la$. Theorem 1 in \cite{Bader-Onn} asserts that there exist a family of irreducible representations $\left\{\mathcal{U}_\la \mid \lambda \leq \l^n\right\}$ such that for $m \le n/2$ 
\begin{equation}\label{BaderOnn}
\begin{split}
  & K\mathrm{Gr}(\l^m) =\bigoplus_{\lambda \leq \l^m} \mathcal{U}_\lambda, \quad \text{and} \\
  &\langle \mathcal{U}_\lambda, K\mathrm{Gr}(\mu) \rangle =|\{\lambda \hookrightarrow \mu \}|, \quad \forall \la,\mu \leq \l^m,
\end{split}
\end{equation}
where $|\{\lambda \hookrightarrow \mu \}|$ stands for the number of non-equivalent embeddings of a module of type $\la$ in a module of type $\mu$. The partial order on partitions here is containment of the corresponding diagrams. We shall explain now how to analyze these representations using the tools we developed here.

The category $\C$ is $\langle \o_\l\rangle$, as in the previous example. The category $\Cp$ is then equivalent to $\o_\l\lmod$. We give here a different parametrization of the objects in $\Cp$ as follows: to every partition $\la\leq \l^n$ we have the functor $F_{\la} = \Hom_{\o_{\l}}(M_{\la},-)$, where $M_{\la} = \oplus_i \o_{\la_i}$. If $f:M_{\la}\to M_{\mu}$ is a homomorphism of $\o_\l$-modules, then we have induced maps 
\[
\begin{split}
&f_*:KF_{\la}(M)\to KF_{\mu}(M), \quad f_*(\uu_g) = \sum_{h:M_{\mu}\to M | hf = g}\uu_h, \\
&f^*:KF_{\mu}(M)\to KF_{\la}(M), \quad f^*(\uu_h) = \uu_{hf}.
\end{split}
\]

Next, we describe explicitly the representations $\ol{KF_{\la}(M)}$. To do so, we first notice that if $i:M_{\nu}\to M_{\la}$ is an embedding, then $i_*(\uu_f) = \sum_{g:gi =f} \uu_g$, or in other words, it is the sum over all possible extensions of $f$ to~$M_{\la}$. We let $\Inj(M_{\la},M)$ denote the set of all embeddings $M_\la \to M$.
\begin{lemma}
The representation $\ol{KF_{\la}(M)}$ is isomorphic to the quotient $K\mathrm{Inj}(M_{\la},M)/(\sum_{i:M_\nu\to M_\la} \mathrm{Im}(i_*))$, where the sum is taken over all proper embeddings of submodules $M_\nu$ into $M_{\la}$. 
\end{lemma}
\begin{proof}
Write $V:=K\mathrm{Inj}(M_{\la},M)/(\sum_{i:\nu\to \la} \Im(i_*))$. 
We think here of $\ol{KF_{\la}(M)}$ as a quotient of $KF_{\la}(M)$. If $f:M_{\la}\to M$ is not injective, then it splits as $M_{\la}\stackrel{pr}{\to} M_{\nu}\stackrel{\ol{f}}{\to} M$ where $M_{\nu} = \text{Coker}(f)$. In particular, it is contained in the image of $pr^*:F_{\nu}\to M_{\la}$, and so it vanishes in the quotient $\ol{KF_{\la}(M)}$. This means that we can think of $\ol{KF_{\la}(M)}$ as a quotient of $K\Inj(M_{\la},M)$. For a similar reason, the image of $i_*$ should also vanish in $\ol{KF_{\la}(M)}$ for every proper embedding $i:M_{\nu}\to M_{\la}$. We thus see that we have a surjective map $V\to \ol{KF_{\la}(M)}.$ We now show that it is an isomorphism. 

To do so, we show that if $KF_{\nu}(M)\to KF_{\la}(M)$ is any constructible map with $\nu<\la$ then it vanishes in the quotient $V$. 
We know that such a map has the form $T_H$ for some $H\subseteq M_{\nu}\oplus M_{\la}$. Write $\alpha:H\to M_{\nu}$ and $\beta:H\to M_{\la}$ for the compositions of the inclusion of $H$ in $M_{\nu}\oplus M_{\la}$ with the natural projections. 
If $\beta$ is not surjective, then the image of $T_H$ is contained in the image of $i^*$, where $i:\Im(\beta)\to M_{\la}$ is the inclusion.
If $\alpha$ is not surjective then we can replace $M_{\nu}$ with the image of $\alpha$, without altering the image of $T_H$. We can thus assume that both $\alpha$ and $\beta$ are surjective.

We thus have the following pushout diagram:
$$\xymatrix{ H\ar[r]^{\beta}\ar[d]^{\alpha} & M_{\la}\ar[d]^{\gamma} \\ M_{\nu}\ar[r]^{\delta} & M_{\rho} }.$$ By Lemma \ref{lem:pushouts} we know that since $\alpha$ and $\beta$ are surjective, $\gamma$ and $\delta$ are surjective as well. Moreover, we know that $T_H = \gamma^*\delta_*$. 

If $\gamma$ is not an isomorphism, then the image of $T_H$ is contained in the span of maps $M_{\la}\to M$ that split through $M_{\rho}$. Again, this vanishes in $V$. So the only possibility that the image of $T_H$ will not vanish in $V$ is if $M_{\nu}$ surjects on $M_{\rho}$, and $M_{\rho}\cong M_{\la}$. But this implies that $\nu\geq \la$, contradicting the assumption that $\nu<\la$. 
\end{proof}

Fix $n>0$ and write $M=\o_\l^n$. Let $\mu=(\mu_1,\ldots, \mu_m)\leq l^n$ be a partition satisfying $m\leq n/2$. Our goal is to decompose $K\Gr(\mu)$. We require  following lemma.
\begin{lemma}\label{lem:basis1}
Let $M = \o_\l^n$ and  
    let $\la = (\la_1,\ldots, \la_a)\leq \l^n$ and $\mu = (\mu_1,\ldots, \mu_b)\leq\l^n$. If $a+b\leq n$ then the set $\{T_H\}_{H\subseteq M_{\la}\oplus M_{\mu}}$ is a basis for $\Hom_{\Aut_R(M)}(K\Hom(M_{\la},M),K\Hom(M_{\mu},M))$. 
\end{lemma}
\begin{proof}
    We have 
    \[
    \begin{split}  
    \Hom_{\Aut_R(M)}(K\Hom(M_{\la},M),&K\Hom(M_{\mu},M)) \\&\cong \Hom_{\Aut_R(M)}(\one,K\Hom(M_{\la},M)^*\ot K\Hom(M_{\mu},M)) \\ &\cong \Hom_{\Aut_R(M)}(\one,K\Hom(M_{\la},M)\ot K\Hom(M_{\mu},M))\\ &\cong \Hom_{\Aut_R(M)}(\one,K\Hom(M_{\la}\oplus M_{\mu},M)),
    \end{split}
    \]
    where we have used the fact that the representation $K\Hom(M_{\la},M)$ is self dual, because it is a permutation representation. It will thus be enough to show that $\{T_H\}_{H\subseteq M_{\la}\oplus M_{\mu}}$ is a basis for the last space. 

    A direct calculation shows that $$T_H(1) = \sum_{f: f(H)=0} \uu_f .$$ Using the inclusion-exclusion principle, we can show that $\mathrm{Span}\{T_H\}_{H\subseteq M_{\la}\oplus M_{\mu}}$ is the same as span $\{\wt{T}_H\}_{H\subseteq M_{\la}\oplus M_{\mu}}$, where $$\wt{T}_H(1) = \sum_{f: \Ker(f)=H} \uu_f.$$ 
    
    Since $a+b\leq n$, it follows that for every $\o_\l$-sumbmodule $H\subseteq M_{\la}\oplus M_{\mu}$, the quotient $(M_{\la}\oplus M_{\mu})/H$ is isomorphic to a submodule of $M$. This implies that $\wt{T}_H\neq 0$ for every $H\subseteq M_{\la}\oplus M_{\mu}$. Finally, for $f\in \Hom(M_{\la}\oplus M_{\mu},M)$ it holds that $\uu_f$ appears with coefficient 1 in $\wt{T}_{\Ker(f)}$ and zero in all the other $\wt{T}_H$. This implies that $\{\wt{T}_H\}_{H\subseteq M_{\la}\oplus M_{\mu}}$ is linearly independent, and since by Proposition~\ref{prop:hom.space.basis} it is spanning, we are done.   
\end{proof}
Let $B \subset A$ be the sets 
\[
\begin{split} 
A &= \{H\subseteq M_{\la}\oplus M_{\mu} \mid H\to M_{\la} \text{ is surjective, } H\to M_{\mu} \text{ is injective}\}, \\
B &= \{H\subseteq M_{\la}\oplus M_{\mu} \mid H\to M_{\la} \text{ is an isomorphism, } H\to M_{\mu} \text{ is injective}\}.
\end{split}
\]

By taking complements we have a direct sum decomposition 
\[
\begin{split}
    KF(M) &= \ol{KF(M)}\oplus \wt{KF(M)},\\ 
K\Hom(M_{\mu},M) &= K\mathrm{Inj}(M_{\mu},M)\oplus K\text{nInj}(M_{\mu},M),
\end{split}
\]
where $\text{nInj}(M_{\mu},M)$ stands for the non-injective maps from $M_{\mu}$ to $M$. By using the inclusion and projection maps we get induced maps between the direct summands. 
\begin{lemma}
The following assertions hold:
\begin{enumerate}
    \item The space $\Hom_{\Aut_R(M)}(\ol{K\Hom(M_{\la},M)},K\Hom(M_{\mu},M))$ has a basis given by $\{T_H\}_{H\in A}$; 
    \item The space $\Hom_{\Aut_R(M)}(\ol{K\Hom(M_{\la},M)},K\mathrm{Inj}(M_{\mu},M))$ has a basis $\{T_H\}_{H\in B}$. 
\end{enumerate}
\end{lemma}
\begin{proof}
The first assertion follows from the more general Lemma~\ref{lem:basis1}.  
For the second assertion, notice first that if $H\notin B$, then the induced image of $T_H:\ol{K\Hom(M_{\la},M)} \to K\Inj(M_{\mu},M))$ is zero. This implies that $\{T_H\}_{H\in B}$ spans $\Hom_{\Aut_R(M)}(\ol{K\Hom(M_{\la},M)},K\Inj(M_{\mu},M))$. We need to show that it is also independent.

We already know that the set $\{T_H\}_{H\in B}$ is linearly independent when we consider its elements as maps $\ol{K\Hom(M_{\la},M)}\to K\Hom(M_{\mu},M)$. So we need to consider the case that for some non-zero linear combination $T:=\sum_{H\in B}a_HT_H$ we have $\Im(T)\subseteq K\text{nInj}(M_{\mu},M)$.

\smallskip

For every non-zero module $M'\subseteq M_{\mu}$ write $\mathrm{pr}_{M'}:M_{\mu}\to M_{\mu}/M'$. 
We thus have a surjective map $$\bigoplus_{M'\neq 0}K\Hom(M_{\mu}/M',M)
\xrightarrow{\oplus \mathrm{pr}_{M'}^*}
K\text{nInj}(M_{\mu},M).$$ This implies that $T$ can be written as a linear combination of compositions of the form $$\ol{K\Hom(M_{\la},M)}\xrightarrow{T_{H'}} K\Hom(M_{\mu}/M',M)\xrightarrow{\mathrm{pr}_{M'}^*} K\text{nInj}(M_{\mu},M),$$ where $H'\subseteq M_{\la}\oplus M_{\mu}/M'$. A direct verification shows that this composition is equal to $T_H$, where $H$ is the inverse image of $H'$ under the projection $M_{\la}\oplus M_{\mu}\to M_{\la}\oplus (M_{\mu}/M')$. In particular, this means that the resulting $H$ is not in the set $B$, since the kernel of $H\to M_{\la}$ contains $0\oplus M'$ and is therefore non-trivial. But this means that as maps $\ol{K\Hom(M_{\la},M)}\to K\Hom(M_{\mu},M)$, the linear combination $T = \sum_{H\in B}a_HT_H$ is a linear combination of maps in $\{T_H\}_{H\in A \smallsetminus B}$. But this contradicts the first assertion. 
\end{proof}

The set $B$ is in one to one correspondence with injective maps $M_{\la}\to M_{\mu}$. We shall henceforth identify $B$ with this set. We have an action of $\Aut(M_{\la})\times \Aut(M_{\mu})$ on this set. We choose orbit representatives $f_1,\ldots, f_k$. For every $f_i:M_{\la}\to M_{\mu}$ we write $$P_i = \{g\in \Aut(M_{\mu}) \mid \Im(g f_i) = \Im(f_i)\}.$$ We think of this group as a \lq parabolic\rq~ subgroup of $\Aut(M_{\mu})$ with respect to the embedding $f_i$. Every $g\in P_i$ gives by restriction an automorphism of $M_{\la}$. This gives a homomorphism $\phi_i:P_i\to \Aut(M_{\la})$. Notice that in the general, if $f_i$ is not an embedding of $M_{\la}$ as a direct summand of $M_{\mu}$, it might happen that $\phi_i$ will not be surjective.  
We claim the following:
\begin{lemma} The stabilizer of $f_i$ in $\Aut(M_{\la})\times \Aut(M_{\mu})$ is $\{(\phi_i(g),g)\mid g\in P_i\}$ and is isomorphic to $P_i$. 
\end{lemma}
\begin{proof}
    The action of $(g_1,g_2)$ on $f_i$ is given by $(g_1,g_2)\cdot f_i = g_2f_ig_1^{-1}$. Assume that $(g_1,g_2)$ stabilizes $f_i$. 
    By comparing images we see that $\Im(f_i) = \Im(g_2 f_i g_1^{-1}) = \Im(g_2f_i)$, which implies that $g_2\in P_i$. Since $f_i$ is injective, this also implies that $g_1$ is defined uniquely by $g_2$ as $g_1= \phi_i(g_2)$. The result follows.
\end{proof}
Next, consider the quotient set $B/\Aut(M_{\mu})$ which still carries an action of $\Aut(M_{\la})$. For $f\in B$ we write $\ol{f}\in B/\Aut(M_{\mu})$ for its image in the quotient set. The following lemma is easy to prove:
\begin{lemma}
    The elements $\ol{f_i}$ are orbit representatives for the action of $\Aut(M_{\la})$ on $B/\Aut(M_{\mu})$. The stabilizer of $\ol{f_i}$ is $\mathrm{Im}(P_i)\subseteq \Aut(M_{\la})$. 
\end{lemma}
Next, we use functor morphing to determine the Grassmanian representation $K\Gr(\mu)$. Recall that it is isomorphic to $K\Inj(M_{\mu},M)/\Aut(M_{\mu})$. 
Write $\Rep(\Aut_R(M)) = \bigoplus_{\la} \Rep(\Aut_R(M))_{\la}$, where $\Rep(\Aut_R(M))_{\la}$ corresponds to the functor $\Hom(M_{\la},-)$. The sum is taken here over all $\la\leq \l^n$. 
This category is equivalent to a subcategory $\Rep(\Aut_R(M_{\la}))$. 
If $V$ is a representation of $\Aut_R(M)$ then functor moprhing means that $V$ splits uniquely as $V = \bigoplus V_{\la}$, where $V_{\la}\in \Rep(\Aut_R(M))_{\la}$. The representation $V_{\la}$ is thus a direct sum of irreducible representation that functor morph to irreducible representations of the automorphism group of the functor $\Hom(M_{\la},-)$. Thus, $V_{\la}$ gives a representation of $\Aut_R(M_{\la})$. We can write this representation explicitly as $\Hom(\ol{K\Hom(M_{\la},M)},V_{\la}) = \Hom(\ol{K\Hom(M_{\la},M)},V)$, where $\Aut_R(M_{\la})$ acts by pre-composing on $\ol{K\Hom(M_{\la},M)}$. We call this $\Aut_R(M_{\la})$-representation the $\la$-functor morphing of $V$, and write it as $\wt{V_{\la}}$.  We claim the following: 
\begin{lemma}
    If $\la \not\leq \mu$ then $\wt{V_{\la}} = 0$. If $\la \leq \mu$ then $\wt{V_{\la}} = \bigoplus_{i=1}^k \Ind_{\mathrm{Im}(\phi_i)}^{\Aut(M_{\la})} \one$.  
\end{lemma}
\begin{proof}
    Since $K\Gr(\mu)\subseteq K\Hom(M_{\mu},M)$ the assertion about the case $\la\not\leq \mu$ is clear.  If $\la\leq \mu$ then we have that  
    \[
    \begin{split}    
    \wt{(K\Gr(\mu))_{\la}} &= \Hom(\ol{K\Hom(M_{\la},M)},K\Gr(\mu))= \Hom(\ol{K\Hom(M_{\la},M)},K\Inj(M_{\mu},M)/\Aut_R(M_{\mu})) \\ &\cong  \Hom(\ol{K\Hom(M_{\la},M)},K\Inj(M_{\mu},M))/\Aut_R(M_{\mu}) \cong KB/\Aut_R(M_{\mu}).
    \end{split}
    \]
    By the above lemma we get the desired result. 
\end{proof}
Consider now the case $\mu= \l^m$ for $m\leq n/2$. In this case there is a unique embedding $M_{\la}\to M_{\mu}$ up to the action of $\Aut(M_{\la})\times \Aut(M_{\mu})$, and every automorphism of $\Aut(M_{\la})$ can be extended to an automorphism of $M_{\mu}$. We thus get that $\wt{V_{\la}} = \one$ for every $\la\leq \l^m$. The corresponding representation $V_{\la}$ is then exactly the representation $\mathcal{U}_{\la}$ of \eqref{BaderOnn} that appears in \cite{Bader-Onn}. For a general $\mu$, the representation $\one$ appears in $\wt{V_{\la}}$ exactly $k$ times, where $k$ is the number of $\Aut(M_{\la})\times \Aut(M_{\mu})$-orbits in $\Inj(M_{\la},M_{\mu})$. This translates to the other result of \cite{Bader-Onn}, namely $\langle \mathcal{U}_{\la},K\Gr(\mu)\rangle = k$.










    



\section{Interpolations}\label{sec:interpolation}
In this section we explain how to construct interpolations of the categories $\Rep(\Aut_R(M))$, using the methods of \cite{meir21}. We thus consider algebraic structures of the form $(W,\mu,\Delta,\eta,\epsilon,(T_r)_{r\in R})$. So far we only consider cases where $W=KM$ and $M$ is a finite $R$-module. In \cite{meir21} it was shown that there is a universal ring of invariants $\U$, and that $W=KM$ defines a character of invariants $\chi=\chi_W:\U\to K$. The main construction of \cite{meir21} gives a symmetric monoidal category $\C_{\chi}$ out of $\chi$, that is equivalent to the category $\Rep(\Aut_R(M))$. It also enables to construct \textit{interpolations} of the symmetric monoidal categories $\Rep(\Aut_R(M))$, as we shall now describe. 

We first need the following lemma:
\begin{lemma}
Let $M$ be a finite $R$-module. For every closed diagram $Di$ in the structure tensors of $W$ it holds that $\chi_{KM}(Di)$ is an integer power of $q$. 
\end{lemma}
\begin{proof}
In Lemma \ref{lem:constructibles} we showed that every constructible element in $\Hom_{\Aut_R(M)}(\one,W^{\ot n})$ is, up to multiplication by an integer power of $q$, of the form $\vv_F(M)$ for some $F\subseteq \Fr_n$. The invariants $\chi(Di)$ can be considered as constructible elements in $\Hom_{\Aut_R(M)}(\one,W^{\ot 0})$. Since $\Fr_0=0$, the only subfunctor is the zero functor, and so we get that every invariant is an integer power of $q$, as desired.
\end{proof}

Let now $M_1,\ldots, M_n$ be non-isomorphic indecomposable $R$-modules. Write $\chi_i = \chi_{M_i}$ for the character of invariants of $M_i$. By the above lemma, for every diagram $Di$ there is an integer $n_i(Di)$ such that $\chi_i(Di) = q^{n_i(Di)}$. We define a one-parameter family of characters by $\chi^{(i)}_t(Di) = t^{n_i(Di)}$.  

\begin{lemma} \quad
\begin{enumerate}
    \item The point-wise multiplication of characters gives $\chi^{(i)}_t\chi^{(i)}_s = \chi^{(i)}_{ts}$. In other words, for every $i$, $\chi^{(i)}_t$ is a multiplicative family of characters. 
    \item The character of invariants of $KM$, where $M = M_1^{a_1}\oplus\cdots\oplus M_n^{a_n}$ is $\chi^{(1)}_{q^{a_1}}\cdots \chi^{(n)}_{q^{a_n}}$.
\end{enumerate}
\end{lemma}
\begin{proof}
The first part of the lemma is straightforward, by comparing the evaluation of characters on a given diagram $Di$. 
The second part follows from the fact that $\chi^{(i)}_q = \chi_i$, and that $KM'\ot KM''\cong K(M'\oplus M'')$ for every two $R$-modules $M',M''$. 
\end{proof}

Recall that a character $\chi$ is called \textit{good} if the category $\C_{\chi}$ is abelian with finite dimensional hom-spaces. We know that in this case the category $\C_{\chi}$ is also semisimple. We claim the following:
\begin{proposition}
    For every $t_1,\ldots, t_n\in K^{\times}$ it holds that $\chi_{t_1,\ldots, t_N}:=\chi^{(1)}_{t_1}\cdots \chi^{(n)}_{t_n}$ is a good character. The categories $\C_{chi_{t_1,\ldots, t_n}}$ are therefore abelian and semisimple. In case $t_i = q^{a_i}$ for some $a_i\in \N$, the category $\C_{\chi_{t_1,\ldots, t_n}}$ is equivalent to $\Rep(\Aut_R(M_1^{a_1}\oplus\cdots\oplus M_n^{a_n}))$.
    \end{proposition}
    \begin{proof}
    We begin by proving that $\chi^{(i)}_t$ is a good character for every $i=1,\ldots, n$ and every $t\in K^{\times}$. By Lemma 6.1. in \cite{meir21} the set of good characters is closed under multiplication, so this will be enough. Lemma \ref{lem:constructibles} tells us that for every $n$ the dimensions of the hom-space $\Hom_{\C_{\chi^{(i)}_{q^a}}}(\one,W^{\ot n})$ is bounded above by the number of subfunctors of $\Fr_n$, which is finite. Since the family $\chi^{(i)}_t$ is a multiplicative family of characters, Proposition~7.7 in \cite{meir21} implies that $\chi^{(i)}_t$ is a good character for every $t\neq 0$. Finally, the character $\chi_{q^{a_1},\ldots, q^{a_n}}$ is the character of invariants of $KM$ where $M = M_1^{a_1}\oplus \cdots\oplus M_n^{a_n}$. By Theorem 1.2. in \cite{meir21} we get an equivalence $\C_{\chi_{q^{a_1},\ldots, q^{a_n}}}\cong \Rep(\Aut_R(M))$.
    \end{proof}
We finish this section with the following description of the Grothendieck group of $\C_{\chi}$.
\begin{proposition}
Assume that $t_1,\ldots, t_n$ are algebraically independent over $\Q$. Then 
$$\text{K}_0(\C_{\chi_{t_1,\ldots, t_n}})\cong \bigoplus_{F\in \Cp} \text{K}_0(\Rep(\Aut_{\Cp}(F))).$$
\end{proposition}
\begin{proof}
    This follows from the fact that everything that is valid for a large enough $M$ is valid generically. More precisely, write $\chi = \chi_{t_1,\ldots, t_n}$. We can still construct the objects $KF(M)$ and $\ol{KF(M)}$ inside $\C_{\chi}$, even when $t_i$ are not integer powers of $q$. The epimorphism theorem is still valid in the categories $\C_{\chi}$. Since the map $\Phi_F(M): K\Aut_{\Cp}(F)\to \End_{\C_{\chi}}(\ol{KF(M)})$ is injective for large enough $M$, it will also be injective for generic values of $t_1,\ldots, t_n$ (see Lemma 7.5. in \cite{meir21}). This implies the proposition.
\end{proof}
\begin{remark}  \quad
\begin{enumerate}
\item The condition that $t_1,\ldots, t_n$ will be algebraically independent over $\Q$ seems to be a bit of an overkill. It is possible that the result of the Proposition will hold also for other values of $t_i$ that are not integer powers of $q$. However, we do not know what is the exact set of singular values of $t_i$ in this interpolation.
\item We can reach the same construction using the methods of Knop from \cite{Knop2007}. Instead of starting with the category $\C= \langle M_1,\ldots, M_n\rangle$, one should start from the category $\Cp$. This is an abelian category with finite hom-spaces, in which every object has finite length, and the simple objects are in one-to-one correspondence with $M_1,\ldots, M_n$. The advantage of the construction we presented here is that our starting point was the object that is most relevant to the constructions in this paper, namely $KM$ as a representation of $\Aut_R(M)$. 
\end{enumerate}
\end{remark}

\appendix
\section{Background from homological algebra}\label{appendix}

We collect some facts from homological algebra for the reader's convenience.
We follow the setup  and notation of the paper: $R$ is a finite $\Ow_\ell$-algebra, $M_1, \ldots M_n$ are indecomposables  in $R\lmod$; $\C$ is the full subcategory of $R\lmod$ whose objects are finite direct sums of the $M_i$'s; $\Cp=\Fun_{\o_\l}(\C,\o_\l\lmod)$.      



\begin{lemma}\label{lem:homXprojective} The object $\Hom_\C(X,-)$ in $\Cp$ is projective for every $X \in \C$. 
\end{lemma}

\begin{proof} By Yoneda's lemma we have natural equivalence in both $F$ and $X$: 
\[
\{\Hom_{\C}(X,-) \to F \} \xleftrightarrow{1:1} F(X), \quad \alpha \mapsto \alpha_X(1_X).
\]

We show that $\Hom_{\Cp}(\Hom_{\C}(X,-), -)$ is exact. Let $0 \to F_1 \to F_2 \to F_3 \to 0$
be a short exact sequence. Then, using Yoneda's lemma, we have a short exact sequence
\[
0 \to \Hom_{\Cp}(\Hom_{\C}(X,-),F_1) \to  \Hom_{\Cp}(\Hom_{\C}(X,-),F_2) \to \Hom_{\Cp}(\Hom_{\C}(X,-),F_3) \to 0,
\]
because the latter is naturally equivalent to 
 \[
 0 \to F_1(X) \to F_2(X) \to F_3(X) \to 0,
 \]
which is exact by definition. 
\end{proof}

Next, we classify the simple objects in $\Cp$. Let $R_i=\End_R(M_i)$. As $M_i$ is indecomposable, $R_i$ is a local (finite) ring. It follows that $E_i:=R_i/J_i$, the  quotient of $R_i$ modulo the Jacobson radical $J_i$, is a finite division algebra and therefore a finite field extension of $\Ow_1$. For $1 \leq i \leq n$,  let $S_i \in \Cp$ be the functor (object in $\Cp$) defined by  
 \[
S_i(M_j)=
\begin{cases}
0, \quad \text{if $i \ne j$}\\
E_i, \quad \text{if $i=j$},
\end{cases}
\]
extended linearly; and for a morphism $\phi:M_j \to M_k$  set $S_i(\phi)$ as zero if $j \ne k$ and as multiplication by $\ol{r}=r+J_i$ where $r \in R_i$, if $i=j=k$, and extended linearly.  

\begin{proposition}[Simples in $\Cp$] Every simple object in $\Cp$ is one of the $S_i$'s.  
\end{proposition}

\begin{proof} Clearly the $S_i$'s are simple. Let $S \in \Cp$ be simple. We note that for every morphism $f:M_i \to M_j$ in $\C$ with $i \ne j$, $S(f):S(M_i) \to S(M_j)$ must be zero; otherwise, the functor $S'$ defined on objects by $S'(M_k) := \sum_{g:M_j \to M_k} S(gf)(M_j)$ and on morphisms as the restriction of $S$ is a proper sub-object. It follows that $S$ is a direct sum of the $S_i$'s and therefore has to be one of them. 
\end{proof}

\begin{proposition} Every object in $\Cp$ is a quotient of a functor of the form $\Hom_R(X,-)$. 
\end{proposition}

\begin{proof} We first prove the assertion for the simples, that is, $F=S_i$ for some $i$.  Indeed in this case we have a surjection $$\Hom(M_i,-) \longrightarrow S_i \longrightarrow 0$$ 
defined by the zero map $\Hom(M_i,M_j) = 0\to S_i(E_j)=0 \to 0$   if $i \ne j$, and by the reduction map $\Hom(M_i,M_i) = R_i \to S_i(M_i)=E_i \to 0$ if $i=j$. In fact, $\Hom(M_i,-)$ is the projective cover of $S_i$. We now prove the assertion for arbitrary $F \in \Cp$ using induction on $\preccurlyeq$. Indeed, let $
0 \to F_1 \xrightarrow{\alpha}  F \xrightarrow{\beta} F_2 \to 0 $   
be a non-trivial short exact sequence. Choose $X_1,X_2$ with $\gamma_i:\Hom(X_i,-) \twoheadrightarrow  F_i$ and let $X_2 \xrightarrow{i_2} X_1\oplus X_2 \xrightarrow{p_1} X_1$ be the embedding and projection. Since $\Hom(X_i,-)$ is projective, there is a lift $\delta_2: \Hom(X_2,-) \to F$ with $\beta \delta_2=\gamma_2$.  It follows that $\delta_2i_2^* :   \Hom(X_1 \oplus X_2, - )  \to F$ is surjective and we are done. 
\[
 \xymatrix{ 
 0 \ar[r] & F_1 \ar[r]^\alpha & F \ar[r]^\beta & F_2 \ar[r] & 0 \\
0 \ar[r]& \Hom(X_1,-) \ar[u]^{\gamma_1}  \ar[r]^{p_1^*} & \Hom(X_1 \oplus X_2, - ) \ar@{-->} [u]^{\delta_2i_2^*} \ar[r]^{i_2^*}& \Hom(X_2,-) \ar[u]^{\gamma_2}\ar@{-->}[ul]_{\delta_2} \ar[r]& 0 \\
}
\] 
\end{proof}
	
\begin{proposition}[Projectives in $\Cp$]\label{prop:projectives.in.C+1} The projectives in $\Cp$ are of the form $\Hom_{\C}(X,-)$.  
\end{proposition}

\begin{proof} We already proved in Lemma~\ref{lem:homXprojective} that $\Hom(X,-)$ is projective for every $X \in \mathrm{ob}(\C)$. Conversely, suppose that $F$ is projective. There exists $X \in \mathrm{ob}(\C)$ with surjection $\alpha: \Hom_{\C}(X,-) \twoheadrightarrow F$. By projectivity of the domain, there exists a splitting $\beta$ with $\alpha \beta = 1$.  Consider $\phi=\beta \alpha : \Hom(X,-) \to \Hom(X,-)$ and note that $\phi^2=\beta(\alpha\beta)\alpha = \phi$. Let $p:X \to X$ be the morphism corresponding to $\phi$, that is, $\phi=p^*$. It follows that $p^2=p$ and that $F\cong\Hom_{\C}(\Im(p), - )$.
\end{proof}	

The next two lemmas are used in Section \ref{sec:p.order}.
\begin{lemma}
Every object in $\D$ has a projective cover. A projective cover is unique up to isomorphism.
\end{lemma}
\begin{proof}
Since $\D$ has enough projectives, for every $F$ in $\D$ there is a projective object $Q$ and a surjective homomorphism $Q\to F$. 
Since $Q$ has only finitely many subobjects, it also has finitely many direct summands. Take a minimal direct summand $P$ of $Q$ that maps onto $F$. The fact that we chose $P$ to be minimal guarantees that $P$ is indeed a projective cover.

For the uniqueness, assume that $p:P\to F$ is a projective cover. Let $q:Q\to P$ be a surjective morphism from some object $Q$ to $P$, such that the composition $pq:Q\to F$ is surjective. We will show that $q$ is surjective. Since $P$ is projective, this implies that $P$ is a direct summand of $Q$. By considering the case where $Q$ is also a projective cover, and $q:Q\to P$ arises from the projective property of $Q$, this will prove uniqueness of the projective cover. 

By the projectivity of $P$ there is a map $i:P\to Q$ such that $pqi=p$.
Write $qi=f:P\to P$. By the finiteness of the number of sub-objects of $P$ we get that the chains $\Ker(f)\subseteq \Ker(f^2)\subseteq \Ker(f^3)\subseteq\cdots $ and $\Im(f)\supseteq \Im(f^2)\supseteq\cdots$ stabilize eventually. Let $N$ be an integer such that $\Im(f^N) = \Im(f^{N+1})$ and $\Ker(f^N) = \Ker(f^{N+1})$. We claim that $P = \Im(f^N)\oplus \Ker(f^N)$. Indeed, if $t\in P$ then $f^N(t)\in \Im(f^N) = \Im(f^{2N})$ so $f^N(t) = f^{2N}(y)$ for some $y\in P$. This implies that $t-f^N(y)\in \Ker(f^N)$, and therefore $P = \Ker(f^N) + \Im(f^N)$. If $t\in \Ker(f^N)\cap \Im(f^N)$ then $t=f^N(s)$ for some $s\in P$. Since $t\in \Ker(f^N)$ it holds that $0=f^N(t) = f^{2N}(s)$, so $s\in \Ker(f^{2N}) = \Ker(f^N)$, which implies that $f^N(s)=t=0$, and so the sum is direct. Since $P$ is a projective cover of $F$, and the direct summand $\Im(f^N)$ projects onto $F$, it holds that $\Im(f^N)=P$. This implies that $f^N$ is surjective, and therefore $f$ is surjective, and therefore an isomorphism. This means that $i:P\to Q$ is a split injection. 
\end{proof}
The next lemma will be useful in determining the automorphism group of $F$:
\begin{lemma}\label{lem:lifting.isos}
Assume that $p:P\to F$ is a projective cover in $\D$ and that we have a commutative diagram of the form 
$$\xymatrix{P\ar[r]^p\ar[d]^{\phi} & F\ar[d]^{\psi} \\ P\ar[r]^p & F}$$
Then $\phi$ is invertible if and only if $\psi$ is invertible.
\end{lemma}
\begin{proof}
If $\phi$ is invertible, then it follows immediately that $\psi$ is at least surjective. Every surjective endomorphism in $\D$ is automatically an isomorphism, so $\psi$ is an isomorphism.

In the other direction, assume that $\psi$ is an isomorphism. Following the previous proof, there is some $N>0$ such that we can write $P = \Ker(\phi^N)\oplus \Im(\phi^N)$. 
The commutativity of the diagram implies that $p$ maps $\Im(\phi^N)$ onto $\Im(\psi^N)$. Since $\psi$ is an isomorphism, $\Im(\phi^N)$ maps onto $F$. Since $P$ is the projective cover of $F$, it holds that $\Im(\phi^N)=P$, so that $\phi$ is surjective, and therefore an isomorphism.
\end{proof}

\bibliographystyle{amsplain}
\bibliography{CMO}

\end{document}